\documentclass[aos,noinfoline]{imsart}

\RequirePackage{amsthm,amsmath,amsfonts,amssymb}
\RequirePackage[authoryear]{natbib}
\RequirePackage{bm}
\RequirePackage{booktabs,longtable,multirow,array}
\RequirePackage{graphicx}
\RequirePackage{capt-of}
\RequirePackage{xcolor,subcaption}
\RequirePackage[
  colorlinks,
  citecolor=blue,
  urlcolor=blue,
  linkcolor=blue
]{hyperref}
\InputIfFileExists{supp_xrefs.tex}{}{%
  \PackageWarning{arxivxref}{Missing static supplementary cross-reference map}%
}

\makeatletter
\setpkgattr{copyright}{text}{}
\def\journal@name{}
\def\journal@url{}
\AtBeginDocument{\let\info@line\@empty}
\makeatother

\graphicspath{{./images/}{./}}

\startlocaldefs
\newcommand{\V}{\operatorname{V}}
\newcommand{\E}{\mathbb{E}}
\newcommand{\Var}{\operatorname{Var}}
\newcommand{\KL}{\operatorname{KL}}
\newcommand{\argmax}{\operatorname*{arg\,max}}

\newcommand{\phiN}{\phi}
\newcommand{\dH}{d_{\mathrm H}}

\newcommand{\post}{\mathrm{post}}

\newcommand{\Pn}{\mathbb{P}_n}
\newcommand{\Pzero}{\mathbb{P}_0}

\newcommand{\R}{\mathbb{R}}
\newcommand{\paperfigure}[2]{%
  \IfFileExists{#1}{\includegraphics[width=#2]{#1}}{%
    \fbox{\parbox[c][0.22\textheight][c]{#2}{\centering
      Figure file not supplied:\\[4pt]\texttt{\detokenize{#1}}}}}}
\providecommand{\email}[1]{\texttt{#1}}

\theoremstyle{plain}
\newtheorem{theorem}{Theorem}
\newtheorem{proposition}{Proposition}
\newtheorem{lemma}{Lemma}
\newtheorem{corollary}{Corollary}

\theoremstyle{definition}
\newtheorem{assumption}{Assumption}
\newtheorem{remark}{Remark}
\endlocaldefs

\begin{document}

\begin{frontmatter}

\title{Regularized-Likelihood Deconvolution with Posterior-Density Reconstruction}
\runtitle{Regularized-Likelihood Deconvolution}
\begin{aug}
\author[A]{\fnms{Marco}~\snm{Di Marzio}\ead[label=e1]{marco.dimarzio@unich.it}\orcid{0000-0003-3500-1890}}
\author[A]{\fnms{Stefania}~\snm{Fensore}\ead[label=e2]{stefania.fensore@unich.it}}
\author[B]{\fnms{Chiara}~\snm{Passamonti}\ead[label=e3]{chiara.passamonti@uni.lu}}
\author[A]{\fnms{Serena}~\snm{Pulcini}\ead[label=e4]{serena.pulcini@phd.unich.it}}

\address[A]{Department of Socio-Economic, Management and Statistical Studies,
University of Chieti-Pescara
\printead[presep={ ,\ }]{e1,e2,e4}}

\address[B]{Department of Mathematics,
University of Luxembourg
\printead[presep={,\ }]{e3}}
\end{aug}

\begin{abstract}
Density deconvolution is an ill-posed inverse problem because recovering the latent density amplifies high-frequency variation in the observed data. We propose a two-stage likelihood procedure. The first estimator maximizes the convolution likelihood over a regularized Gaussian-mixture sieve; the second treats this fit as an empirical prior and averages the fitted conditional latent densities. Under fixed-model misspecification, the estimators generally have different population targets. For bounded local multiplicative perturbations of the latent density, the population posterior map is locally contractive in chi-square divergence under identifiable convolution. Under correct finite-dimensional specification, however, reconstruction adds a nonnegative first-order variance component. For growing sieves, we separate observed-domain estimation, inverse stability, and posterior empirical variation. We derive explicit rates under Gaussian error and recover the classical ordinary-smooth exponent, up to logarithmic factors, with a constructive verification for Laplace error. Simulations and a Framingham blood-pressure application provide empirical performance evidences.
\end{abstract}


\end{frontmatter}

\section{Introduction}

In many statistical applications, the variable of interest is observed through a contaminated measurement. We consider the additive measurement-error model
\begin{equation}
Y_i=X_i+\varepsilon_i,\qquad i=1,\ldots,n,
\label{eq:additive-model}
\end{equation}
where \((X_i,\varepsilon_i)_{i=1}^n\) are independent copies of \((X,\varepsilon)\), and \(X_i\) is independent of \(\varepsilon_i\). The latent variable and the measurement error have densities \(f_X\) and \(f_\varepsilon\), respectively. The error density is assumed known, and only \(Y_1,\ldots,Y_n\) are observed. Throughout, the \(X\)-scale is called the latent domain and the \(Y\)-scale the observed domain.
 The objective is to estimate \(f_X\), while the observed density is the convolution \(f_Y(y)=\int f_X(x)f_\varepsilon(y-x)\,dx\). 
 Equivalently, in the Fourier domain, \(\varphi_Y(t)=\varphi_X(t)\varphi_\varepsilon(t)\). The classical kernel deconvolution estimator (see \cite{liutaylor1989,
stefanski1990}) has the form 
$\frac{1}{2\pi}\int e^{-itx}\widehat\varphi_Y(t) \frac{\varphi_W(ht)}{\varphi_\varepsilon(t)}\,dt,$    
where \(\varphi_\varepsilon \neq 0\) is typically known, \(\widehat\varphi_Y(t)=n^{-1}\sum_{j=1}^n e^{itY_j}\) is the empirical characteristic function,  \(\varphi_W\) is a damping factor depending on the kernel function \(W\) and a smoothing parameter \(h\). However, the variance of this estimator depends on the inverse factor
\(1/\varphi_\varepsilon(t)\). Thus, the stronger the decay of
\(|\varphi_\varepsilon(t)|\), the more severe the amplification of
high-frequency empirical noise \(\widehat\varphi_Y(t)\). 

Several alternatives avoid explicit Fourier inversion, including constrained quadratic programming, penalized contrasts, finite-basis likelihood approximations, and penalized likelihood methods \citep{yang2020,comte2006,guan2021,cai2025}. Related approaches estimate mixing densities by maximum penalized likelihood; for example, \citet{liu2009} use a functional EM algorithm for continuous mixing densities on a compact interval. Bayesian deconvolution instead stabilizes the problem through mixture priors and prior-induced regularity \citep{sarkar2014,sarkar2018}. These methods typically produce either a single regularized latent-density estimate or a posterior distribution over latent laws.

The Bayesian inverse-problem literature provides a broader perspective on regularization through Gaussian priors. In Gaussian white-noise models, \citet{knapik2011} study how posterior contraction depends on ill-posedness, signal regularity, and prior smoothness, while \citet{knapik2016} develop adaptive empirical- and hierarchical-Bayes procedures. Our setting differs because we estimate the latent density by regularized likelihood and apply a single posterior-averaging step rather than defining a full posterior on the unknown density.

We combine these perspectives through a two-stage likelihood framework producing two distinct latent-density estimators. The first maximizes the observed convolution likelihood over a regularized real-line mixture sieve and therefore provides a direct likelihood-regularized estimate. The second uses this fit as an empirical prior and averages the corresponding conditional latent densities over the observations, producing a posterior reconstruction that need not remain in the original finite-mixture sieve. This reconstruction can be interpreted as a data-driven correction of the first-stage latent estimate: discrepancies between the fitted and true observed distributions are reflected in the conditional latent densities and transported back to the latent scale through posterior averaging. This mechanism can reduce approximation bias when the first-stage sieve is too restrictive, but, because the same observations also enter the empirical average, it introduces an additional source of sampling variability. These two effects motivate the bias--variance comparison developed below.
The posterior reconstruction can be viewed as a deconvolution-specific instance of the posterior-averaging update studied by \citet{chae2018} in a general mixture-estimation framework, where repeated updates are shown to increase the likelihood and converge toward the nonparametric maximum-likelihood estimator as the number of iterations increases, for a fixed sample. By contrast, in our setting, the same map enters the statistical analysis of an ill-posed deconvolution problem: its input is a random deconvolution estimator obtained by maximizing the convolution likelihood over a regularized and data-selected mixture sieve. We study the resulting estimator as \(n\) increases, characterizing its population target and sampling behavior under both misspecification and correct specification, and develop a growing-sieve theory in which the number of mixture components and the admissible parameter space may vary with the sample size.

The likelihood-based deconvolution methods most closely related to our first-stage estimator are those of \citet{guan2021} and \citet{cai2025}. \citet{guan2021} represents the latent density on a bounded interval using Bernstein polynomials with data-driven degree selection, whereas \citet{cai2025} maximizes a roughness-penalized likelihood over an infinite-dimensional function space. The penalty in \citet{cai2025} is based on the integrated squared second derivative of the latent density, and their theoretical analysis is developed under smoothness, boundary, and compact-support conditions. Our first-stage estimator instead maximizes the convolution likelihood over a regularized real-line Gaussian-mixture sieve, with component constraints and iterated-log BIC selection. It therefore avoids prespecifying compact latent support and is followed by posterior-density reconstruction, which generally yields an estimator outside the fitted sieve. Finally, unlike Bayesian and empirical-Bayes deconvolution procedures \citep{sarkar2014,efron2016,madridpadilla2018}, the posterior reconstruction considered here does not arise from a prior placed on the unknown latent density. It is constructed from the regularized convolution-likelihood estimator itself, so that the two stages can be analyzed jointly as a pair of frequentist deconvolution estimators.

Our first theoretical contribution concerns fixed-dimensional models. Under misspecification, the direct estimator converges to the density that best fits the observed convolution model, whereas posterior reconstruction converges to a generally different population target.
 We show that the population posterior map locally contracts every nonzero bounded multiplicative perturbation of the true density in chi-square divergence when convolution is identifiable. Under correct specification, the two estimators are both consistent, but posterior reconstruction adds an orthogonal, nonnegative first-order variance component.
Our second contribution concerns growing sieves. We separate the analysis into three modules: estimation of the convolved density on the observed scale, stability of the inverse map from the observed to the latent domain, and empirical fluctuation of the fitted posterior kernels. This decomposition yields conditions under which posterior reconstruction preserves the latent \(L^1\) convergence order of the direct estimator even though it generally lies outside the fitted sieve.
We consider the modular theory under Gaussian and ordinary-smooth measurement error. For Gaussian error, we obtain explicit convergence rates on an infinite-dimensional Gaussian-smoothed latent class. 
For ordinary-smooth errors, we recover the classical polynomial deconvolution exponent, up to logarithmic factors, and verify all required conditions for Laplace error. The simulations exhibit the corresponding bias--variance pattern: posterior reconstruction is most useful when the first-stage sieve is appreciably misspecified, whereas the direct estimator is preferable when a low-dimensional Gaussian mixture already provides an accurate representation.

The remainder of the paper is organized as follows. Section~\ref{sec:methodology} introduces the direct estimator, posterior reconstruction, and the regularized Gaussian-mixture sieve. Section~\ref{sec:theory} compares the two estimators in fixed-dimensional models, under both misspecification and correct specification. Section~\ref{sec:growing-sieve-theory} develops the growing-sieve theory, and Section~\ref{sec:explicit-sieve-rates} specializes it to Gaussian and ordinary-smooth measurement error. Sections~\ref{sec:simulation} and~\ref{sec:realdata} report the simulation study and the Framingham application. Section~\ref{sec:future-work} discusses limitations and extensions. Proofs, deferred technical results, implementation details, and complete simulation tables are provided in the Supplementary Material.
\section{Likelihood estimation and posterior reconstruction}
\label{sec:methodology}
In model~\eqref{eq:additive-model}, a candidate latent density \(f\) induces the observed one \(g_{f}=f*f_\varepsilon\).
We use the notation
   $ \ell_f(y)=\log g_f(y).$
Thus, the likelihood is evaluated on the observed scale, while the optimization is carried out over latent densities. 
Define the empirical convolution log-likelihood by
$
   \ell_n(f)
    =
    \sum_{i=1}^n \ell_f(Y_i).$
Let
\(
    \mathcal F_n
    =
    \bigcup_{K=1}^{K_n}\mathcal F_{K,n}
\)
be an indexed sieve of latent-density models, where
\(\mathcal F_{K,n}\) denotes the admissible \(K\)-component mixture
class at sample size \(n\). The index \(n\) indicates that both the
admissible parameter space underlying \(\mathcal F_{K,n}\) and the
maximum model order \(K_n\) may vary with the sample size.
We define the first-stage estimator by maximizing the
following penalized convolution likelihood:
\begin{equation}
\label{eq:penalized-sieve-estimator}
\widehat f
\in
\argmax_{f\in\mathcal F_n}
\left\{
    \frac1n\sum_{i=1}^n \ell_f(Y_i)
    -
    \frac{\operatorname{pen}_n(f)}{n}
\right\}.
\end{equation}
For \(f\in\mathcal F_{K,n}\), the penalty is the iterated-log BIC one given by
\(
    \operatorname{pen}_n(f)
    =
    (d_K/2)\,
    \mathsf L^{\circ 3}(n)\log n,
\)
where \(d_K\) denotes the number of free parameters in the
\(K\)-component model, \(a\vee b=\max\{a,b\}\), and
\(\mathsf L^{\circ 3}(n)\) denotes the third iterate of
\(\mathsf L(x)=\log(e\vee x)\), that is,
\(\mathsf L^{\circ 3}(n)
=\mathsf L\bigl(\mathsf L(\mathsf L(n))\bigr)\).
Accordingly, denoting by
\(\widehat f_K\) the convolution-likelihood maximizer over
\(\mathcal F_{K,n}\), the selected mixture order is
\begin{equation}
\label{eq:bic}
\widehat K
=
\min\operatorname*{arg\,min}_{1\leq K\leq K_n}
\left\{
    -2\ell_n(\widehat f_K)
    +
    d_K\mathsf L^{\circ 3}(n)\log n
\right\}.
\end{equation}
Criterion~\eqref{eq:bic} is the \(\nu\)-BIC of
\citet{nguyennguyen2026} with \(\nu=3\). The additional iterated-log factor provides the asymptotic penalty used in the model-selection argument. Since
\(\mathsf L^{\circ3}(n)=1\) for \(n<3.8\times10^6\), the criterion is numerically identical to the ordinary BIC at all sample sizes considered below. Section~\ref{supp-app:fixed-order-selection} of the Supplementary Material establishes consistency of the corresponding idealized global selector over a fixed candidate range under Gaussian and Laplace measurement error. This fixed-range result is distinct from the growing-sieve analysis used to derive the nonparametric convergence rates.
For the selected order \(\widehat K\), the first-stage estimator in \eqref{eq:penalized-sieve-estimator} has the
representation
\begin{equation}\label{eq:first-stage-mixture-representation}
    \widehat f(x)
=
\sum_{k=1}^{\widehat K}
\widehat\pi_k \widehat h_k(x),
\qquad
\widehat\pi_k\ge 0,
\quad
\sum_{k=1}^{\widehat K}\widehat\pi_k=1,
\end{equation}
where \(\widehat h_k\) denotes the fitted density of \(X\) conditional on
component \(k\). The corresponding fitted observed component is
\(
\widehat g_k(y)
=
\int \widehat h_k(u) f_\varepsilon(y-u)\,du,
\)
so that
\(
g_{\widehat f}(y)
=
\sum_{k=1}^{\widehat K}
\widehat\pi_k \widehat g_k(y).
\)
Let \(f_0\) denote the true latent density and let
\(g_0=g_{f_0}\) denote the corresponding observed density.
Let \(\Pzero\) denote the true distribution of \(Y\), and let
\(\Pn\) denote the empirical measure of \(Y_1,\ldots,Y_n\). For any candidate latent density \(f\) such that
\(g_f(y)>0\), define
\[
q_f(x\mid y)
=
\frac{f(x)f_\varepsilon(y-x)}{g_f(y)},
\]
the conditional density of \(X\) given \(Y=y\) when \(f\) is used as the
latent law. When convenient, we write
\(
q_{f,x}(y)=q_f(x\mid y).
\)

We distinguish the population and empirical posterior-averaging maps,
\[
\mathcal R_0(f)(x)
=
\Pzero q_f(x\mid Y),
\qquad
\mathcal R_n(f)(x)
=
\Pn q_f(x\mid Y)
=
\frac1n\sum_{i=1}^n q_f(x\mid Y_i).
\]
At the true latent density, the law of total probability gives the self-consistency identity \(\mathcal R_0(f_0)(x)=\Pzero q_{f_0}(x\mid Y)=f_0(x)\), which motivates replacing the unknown \(f_0\) by the convolution-likelihood estimate \(\widehat f\) and the population average by its empirical counterpart.
 We therefore define the posterior
reconstruction by
\begin{equation}
\label{eq:posterior-density-estimator}
\widehat f_{\post}(x)
=
\mathcal R_n(\widehat f)(x)
=
\frac1n\sum_{i=1}^n
\frac{\widehat f(x)f_\varepsilon(Y_i-x)}
     {g_{\widehat f}(Y_i)}.
\end{equation}

Each \(q_{\widehat f}(\cdot\mid Y_i)\) is a density, so
\(\widehat f_{\post}\) is nonnegative and integrates to one without
additional normalization. The construction consists of a single empirical
posterior-averaging update and need not remain in the finite-mixture sieve
used to obtain \(\widehat f\). The distinction between
\(\mathcal R_0\) and \(\mathcal R_n\) will also be useful in the theoretical
analysis: the former describes the population effect of posterior
reconstruction, whereas the latter includes the empirical fluctuation
introduced by averaging over the observed sample. Sections~\ref{sec:theory}
and~\ref{sec:growing-sieve-theory} study these effects in the
fixed-dimensional and growing-sieve regimes, respectively. The construction
is also connected to compound-decision and nonparametric maximum-likelihood
methods
\citep{robbins1956,kieferwolfowitz1956,laird1978,efron2011}.

\begin{remark}
The fitted mixture structure gives a componentwise representation of
\eqref{eq:posterior-density-estimator}. Let
$\widehat r_{ik}=\widehat\pi_k\widehat g_k(Y_i)/\sum_{\ell=1}^{\widehat K}\widehat\pi_\ell\widehat g_\ell(Y_i),$ with $k=1,\ldots,\widehat K$
be the fitted posterior responsibility of component \(k\) for observation
\(Y_i\). Then we can write
\[
\widehat f_{\post}(x)
=
\frac1n\sum_{i=1}^n\sum_{k=1}^{\widehat K}
\widehat r_{ik}
\widehat f_{X\mid Y,k}(x\mid Y_i),
\]
where
\(
\widehat f_{X\mid Y,k}(x\mid Y_i)
=
\widehat h_k(x) f_\varepsilon(Y_i-x)/
\widehat g_k(Y_i)
.
\)
Thus each observation contributes a full fitted conditional density, with
mixture components weighted by their posterior responsibilities.
\end{remark}

We now specialize the generic sieve  to the finite Gaussian
mixtures. 
Thus, a candidate latent density is represented by the parameter vector \(\theta=(\pi_1,\ldots,\pi_K,\mu_1,\ldots,\mu_K,\sigma_1,\ldots,\sigma_K)\), with the usual simplex constraint on the weights. The regularized parameter space is
\begin{equation}
\label{eq:regularized-parameter-space}
\Theta_{K,n}^{\mathrm{reg}}
=
\left\{
\theta:
\pi_k\ge \pi_{\min,n},\quad
\sigma_{\min,n}\le \sigma_k\le \sigma_{\max,n},\quad
|\mu_k|\le M_n,
\quad k=1,\ldots,K
\right\},
\end{equation}
that contains deterministic tuning sequences chosen so that the sieve expands
with \(n\). 
\label{rem:calibrazione}
The bounds in \(\Theta_{K,n}^{\mathrm{reg}}\) have complementary roles. The
upper bound \(K_n\) controls the size of the sieve. The weight constraint
\(\pi_k\ge\pi_{\min,n}\) excludes asymptotically negligible components while
keeping the simplex constraint feasible. The lower scale bound
\(\sigma_{\min,n}\) regularizes the latent sieve and enters the latent
approximation, inverse-stability, and posterior empirical-process arguments.
Under Gaussian measurement error, it is not needed to prevent the usual
finite-mixture singularity of the observed likelihood, because convolution
gives every observed component a variance bounded below by
\(\sigma_\varepsilon^2\). In particular,
Proposition~\ref{supp-prop:gaussian-global-likelihood} of the Supplementary
Material shows that the global Gaussian observed-likelihood bound does not use
\(\sigma_{\min,n}^{-1}\). The upper scale bound
\(\sigma_{\max,n}\le\overline\sigma<\infty\) excludes arbitrarily diffuse
components, while the location bound \(|\mu_k|\le M_n\) controls escaping
components and latent tails. As \(n\) increases, \(K_n\) and \(M_n\) may grow
and \(\sigma_{\min,n}\) may decrease, so the sieve becomes richer while
retaining the regularity used in the rate arguments. These restrictions are
sufficient for the approximation, covering, inverse-stability, and posterior
empirical-process arguments used below.
The corresponding fixed-\(K\) density class is \(\mathcal F_{K,n}^{\mathrm{reg}} = \{f_\theta:\theta\in\Theta_{K,n}^{\mathrm{reg}}\}\). The full sieve allows the number of components to vary up to a deterministic upper bound \(K_n\), which increases with \(n\):
\begin{equation}\label{Fn}
\mathcal F_n
=
\bigcup_{K\le K_n}\mathcal F_{K,n}^{\mathrm{reg}}.
\end{equation}
 For fixed \(K\), let \(\widehat\theta_K\) denote the maximizer of the convolution likelihood over \(\Theta_{K,n}^{\mathrm{reg}}\). In finite samples, \(K\) is selected by fitting the candidate models and minimizing the
iterated-log BIC criterion in~\eqref{eq:bic}, with $d_K=3K-1$.
The generic first-stage representation
\eqref{eq:first-stage-mixture-representation} takes the form
$\widehat f(x)
=
\sum_{k=1}^{\widehat K}
\widehat\pi_k\phiN(x;\widehat\mu_k,\widehat\sigma_k^2).$
The preceding parameter restrictions belong to the growing-sieve
framework, in which \(K_n\) and the admissible parameter ranges may
depend on \(n\). In particular, the positive lower bound
\(\pi_{\min,n}\) is a vanishing regularization device. By comparison, the
fixed-order model-selection result in Section~\ref{supp-app:fixed-order-selection} of the Supplementary Material considers a fixed
candidate range \(1\leq K\leq\bar K<\infty\) and is formulated for the
corresponding idealized selector on the closed mixture simplex, where
zero weights are allowed.
The general formulation permits sample-size-dependent scale bounds. The  Gaussian result in
Section~\ref{sec:gaussian-smoothed-end-to-end} analyzes a specialized sub-sieve whose component scales
remain in a fixed compact interval bounded away from zero, whereas the broad observed-domain approximation
result uses a different shrinking-scale calibration.

 We use
Gaussian mixtures as a baseline because they combine approximation flexibility,
standard entropy control, and stable finite-mixture computation
\citep{ghosalvdv2001,kruijer2010,mclachlan2000}. 

\section{Fixed-dimensional theory}
\label{sec:theory}

This section isolates the fixed-dimensional statistical effect of posterior reconstruction. We first describe the population targets of the two density estimators. We then study three questions specific to posterior reconstruction: whether posterior reconstruction changes the population target under misspecification, whether this change can reduce local misspecification bias, and what first-order variance is associated with the reconstruction under correct specification. The growing-sieve analysis is developed separately in Section~\ref{sec:growing-sieve-theory}.


 We use throughout the posterior kernel \(q_f\) and the population and empirical posterior-averaging maps \(\mathcal R_0\) and \(\mathcal R_n\) introduced in Section~\ref{sec:methodology}. The distinction between the two maps is central here: \(\mathcal R_0\) determines the population effect of posterior reconstruction, whereas \(\mathcal R_n\) is the empirical update that defines \(\widehat f_{\post}\).

\subsection{Population targets}

Let \(\mathcal F_K=\{f_\theta:\theta\in\Theta_K\}\) denote a fixed-dimensional
mixture family, with \(K\) not depending on \(n\). Under correct specification
and standard fixed-dimensional regularity, consistency of the
convolution-likelihood parameter estimator implies pointwise consistency of
both \(\widehat f\) and \(\widehat f_{\post}\); integrated \(L^1\) and \(L^2\)
consistency follow under the corresponding domination conditions. A precise
statement is given in Proposition~\ref{supp-prop:supp-fixed-consistency} of the
Supplementary Material.

The more substantive distinction appears when \(\mathcal F_K\) is used as an
approximation rather than as the data-generating model. Let
\(
    f^\star
    \in
    \argmax_{f\in\mathcal F_K}
    \Pzero\log g_f(Y)
\)
be a population convolution-likelihood projection of \(f_0\) onto
\(\mathcal F_K\). The direct estimator then targets \(f^\star\), whereas
posterior reconstruction generally has a different population limit, as made
explicit by the following proposition.

\begin{proposition}[Pseudo-targets under fixed-model misspecification]
\label{prop:misspecification}
Suppose that \(\widehat f \xrightarrow{p}f^\star\) and that, for every fixed \(x\),
\(
q_{\widehat f}(x\mid\cdot)\xrightarrow{p}q_{f^\star}(x\mid\cdot)\) in $L^1(\Pzero)$.
Assume also that posterior averaging is locally stable at \(f^\star\), in the sense that
\(
(\Pn- \Pzero)q_{\widehat f}(x\mid\cdot)=o_p(1)
\)
 in the relevant integrated norm. Then, for every fixed \(x\),
\(
\widehat f_{\post}(x) \xrightarrow{p} \mathcal R_0(f^\star)(x).
\)
Thus, under fixed-model misspecification, the direct first-stage estimator targets
the convolution-likelihood projection \(f^\star\), whereas the posterior-density
estimator targets its posterior-reweighted version \(\mathcal R_0(f^\star)\).
\end{proposition}
\begin{proof}
See Section~\ref{supp-app:main-proofs} of the Supplementary Material. 
\end{proof}
A simple algebraic identity makes the distinction between the two pseudo-targets explicit. Let \(f\) be a density such that \(g_f(y)>0\) for
\(g_{f_0}\)-almost every \(y\), and assume that the integral below is well
defined. Then
\(
    \mathcal R_0(f)(x)
=
f(x)\{1+\Delta_f(x)\},
\)
where
\(
    \Delta_f(x)
    =
    \int
    f_\varepsilon(y-x)
    \frac{g_{f_0}(y)-g_f(y)}{g_f(y)}\,dy .
\)
Evaluated at \(f=f^\star\), the identity shows how the observed-domain discrepancy modifies the likelihood projection. In particular, the two targets coincide when \(g_{f^\star}=g_{f_0}\), and need not coincide otherwise.
In the misspecified M-estimation framework, $f^\star$
 is the pseudo-true element selected by the observed convolution likelihood (see \cite{neweymcfadden1994} and \cite{vandervaart1998}). What is specific to the present construction is that the second-stage estimator does not inherit this likelihood projection as its population target. Instead, posterior reconstruction generates the distinct target
\(\mathcal R_0(f^\star)\). Thus, misspecification leads here to two systematically related population limits rather than to a single pseudo-true density.
\subsubsection{Local bias contraction under misspecification}
\label{subsec:bias-contraction}

Proposition~\ref{prop:misspecification} 
alone does not establish whether \(\mathcal R_0(f^\star)\) is closer to the true
latent density than \(f^\star\) is. We first derive an exact criterion for when the population posterior update reduces the latent \(L^2\) bias.
For \(h\in L^2(\mathbb R)\), define the convolution operator
\(
\mathcal A [h](y)
=
\int_{\mathbb R} f_\varepsilon(y-x)h(x)\,dx.
\)
Since \(f_\varepsilon\in L^1(\mathbb R)\), the operator
\(\mathcal A :L^2(\mathbb R)\to L^2(\mathbb R)\) is bounded, and
\(g_f=\mathcal A [f]\) for every candidate latent density \(f\in L^2(\mathbb R)\).
By Fubini's theorem, its \(L^2\)-adjoint is
\(
\mathcal A ^\ast[v](x)
=
\int_{\mathbb R}f_\varepsilon(y-x)v(y)\,dy.
\)
Using this notation, the population posterior-averaging operator can be written as
\(
\mathcal R_0(f)(x)
=
f(x)
\mathcal A^\ast\left[g_0/g_f\right](x),
\)
whenever \(g_f>0\) \(\Pzero\)-almost everywhere. The ratio \(g_0/g_f\) measures the discrepancy between the true observed density and that induced by the candidate \(f\), while \(\mathcal A ^\ast\) transports this multiplicative correction from the observed domain back to the latent domain. 
For a density \(f\) such that \(g_f>0\) \(\Pzero\)-almost everywhere, write
\(e=f-f_0\) and, whenever the expression is well defined, set
\(
C_f[e]
=
f\,\mathcal A^\ast\left[\mathcal A[e]/g_f\right].
\)
If \(e,C_f[e]\in L^2(\mathbb R)\), the population posterior update satisfies
the exact decomposition
\(
\mathcal R_0(f)-f_0
=
e-C_f[e],
\)
and hence
\(
\|\mathcal R_0(f)-f_0\|_2^2
=
\|f-f_0\|_2^2
-
2\langle e,C_f[e]\rangle
+
\|C_f[e]\|_2^2.
\)
The exact identity is established in
Lemma~\ref{supp-lem:supp-posterior-exact-error} of the Supplementary Material. Thus, posterior reconstruction reduces the
latent \(L^2\) error exactly when
\[
\|\mathcal R_0(f)-f_0\|_2
<
\|f-f_0\|_2
\quad\Longleftrightarrow\quad
2\langle e,C_f[e]\rangle
>
\|C_f[e]\|_2^2.
\]
This criterion depends on the shape of the misspecification because the
correction \(C_f[e]\) itself depends on \(f\).

We therefore consider a generic local misspecification of the true density
of the form
\(
f_t(x)=f_0(x)\{1+t s(x)\},
\)
where \(t\) controls the magnitude of the deviation from \(f_0\) and \(s\)
describes its relative shape. To preserve the normalization of \(f_t\), \(s\) must have mean
zero under \(f_0\). Accordingly, let
\[
L_0^2(f_0)
=
\left\{
s:
\int_{\mathbb R}s(x)^2f_0(x)\,dx<\infty,
\quad
\int_{\mathbb R}s(x)f_0(x)\,dx=0
\right\}
\]
be the mean-zero subspace of \(L^2(f_0)\).
To describe how a latent perturbation is transmitted to the observed domain,
define the operator
\(
K:L^2(f_0)\longrightarrow L^2(g_0)
\)
by
\[
(Ks)(y)
=
\mathbb E_0\{s(X)\mid Y=y\}
=
\frac{\mathcal A[f_0s](y)}{g_0(y)},
\qquad g_0(y)>0.
\]
Thus, \(Ks\) is the observed-domain counterpart of \(s\). Indeed, the perturbation \(f_t=f_0(1+ts)\) induces the observed density
\(
g_t=g_0(1+tKs).
\)
To map an observed-domain function back to the latent domain, let
\(
K^\ast:L^2(g_0)\longrightarrow L^2(f_0)
\)
denote the adjoint of \(K\). Conversely, the adjoint \(K^\ast\) maps an observed-domain function
\(v\in L^2(g_0)\) back to the latent domain by conditional expectation:
\(
(K^\ast v)(x)=\mathbb E_0\{v(Y)\mid X=x\}=\mathcal A^\ast[v](x).
\)
The composition
\(
T=K^\ast K
\)
therefore maps a latent perturbation \(s\) to the component that remains
after passing from the latent domain to the observed domain and back.
The operator \(K\) is contractive, while \(T\) is positive and
self-adjoint on \(L^2(f_0)\), with
\[
\langle s,Ts\rangle_{L^2(f_0)}
=
\|Ks\|_{L^2(g_0)}^2,
\qquad
\|Ts\|_{L^2(f_0)}
\leq
\|Ks\|_{L^2(g_0)}
\leq
\|s\|_{L^2(f_0)}.
\]
These properties are established in
Lemma~\ref{supp-lem:supp-conditional-information-operator} of the
Supplementary Material.

For the local family \(f_t=f_0(1+ts)\), assume that
\(s\in L_0^2(f_0)\cap L^\infty(f_0)\) and
\(|t|\|s\|_\infty<1\). The mean-zero condition ensures that \(f_t\)
integrates to one, while the restriction on \(t\) guarantees its
nonnegativity. By linearity of \(\mathcal A\) and the definition of \(K\),
\(g_t=\mathcal A[f_t]=g_0(1+tKs)\). Moreover,
\(\|Ks\|_\infty\leq\|s\|_\infty\), and hence
\(g_t/g_0\geq1-|t|\|s\|_\infty>0\).
For densities \(p\ll q\), let
\(\chi^2(p,q)=\int_{\mathbb R}(p(x)/q(x)-1)^2q(x)\,dx\)
denote the chi-square divergence of \(p\) from \(q\). For the local family
\(f_t=f_0(1+ts)\), this becomes
\(
\chi^2(f_t,f_0)=t^2\|s\|_{L^2(f_0)}^2,
\)
so that chi-square divergence directly measures the size of the local
misspecification. The next theorem examines how the population posterior
update changes this quantity.
\begin{theorem}[Local chi-square contraction]
\label{thm:local-posterior-contraction}
Let \(I\) denote the identity operator on \(L_0^2(f_0)\), \(s\in L_0^2(f_0)\cap L^\infty(f_0)\), and set
\(
f_t=f_0(1+ts),
\) and
$|t|\|s\|_\infty<1.
$
Then, as \(t\to0\),
\[
\frac{\mathcal R_0(f_t)-f_0}{f_0}
=
t(I-T)s
+
O_{L^2(f_0)}(t^2).
\]
Consequently,
\[
\chi^2\{\mathcal R_0(f_t),f_0\}
-
\chi^2(f_t,f_0)
=
-t^2
\langle s,(2T-T^2)s\rangle_{L^2(f_0)}
+
O(|t|^3),
\]
where
\(
\langle s,(2T-T^2)s\rangle_{L^2(f_0)}
\geq
\|Ks\|_{L^2(g_0)}^2.
\)
Hence the contraction is strict whenever \(Ks\neq0\).
If, in addition, \(g_0>0\) almost everywhere with
\(
\varphi_\varepsilon(u)\neq0
\) and \(
u\in\mathbb R,
\)
then every nonzero
\(s\in L_0^2(f_0)\cap L^\infty(f_0)\)
satisfies
\[
\chi^2\{\mathcal R_0(f_t),f_0\}
<
\chi^2(f_t,f_0)
\]
for all sufficiently small \(t\neq0\).
\end{theorem}
\begin{proof}
See Section~\ref{supp-app:main-proofs} of the Supplementary Material.
\end{proof}
Theorem~\ref{thm:local-posterior-contraction} provides a
population-level justification for the posterior reconstruction step.
Under an identifiable convolution model, every nonzero bounded local
departure from \(f_0\) is moved closer to \(f_0\) by the posterior update
in chi-square divergence. Thus, when the population target of the
first-stage estimator lies in a sufficiently small multiplicative
neighborhood of the true density, the posterior-density estimator does
more than converge to a different target: its population target is locally
less misspecified.
In this chi-square geometry, conditional expectation acts as a contraction
between \(L^2(f_0)\) and \(L^2(g_0)\). Posterior reconstruction may therefore
reduce misspecification bias while introducing additional sampling
variability, a tradeoff made explicit by the first-order analysis below.

\subsubsection{First-order expansions}
\label{sec:fixed-K-M-expansion}
We now turn to the first-order sampling behavior of the two estimators under correct fixed-dimensional specification. In this setting, both estimators are consistent, but posterior reconstruction introduces an additional empirical averaging step whose contribution is not captured by the direct likelihood estimator. The following result compares their asymptotic linear representations and makes the resulting variance cost explicit. Consistency under the corresponding fixed-(K) regularity conditions is established in Proposition~\ref{supp-prop} of the Supplementary Material. 
\begin{theorem}[Fixed-dimensional first-order expansions and variance cost]
\label{thm:mestimator-expansion}
Suppose that \(f_0=f_{\theta_0}\in\mathcal F_K\) in a correctly specified,
regular fixed-\(K\) mixture model, after fixing an admissible label ordering,
and that \(\widehat\theta\) is an interior convolution-likelihood maximizer
satisfying the usual fixed-dimensional M-estimation regularity conditions.

Write \(g_\theta=g_{f_\theta}\) and \(q_\theta=q_{f_\theta}\). Let
\(S_{\theta_0}(Y)=
\left.\partial_\theta\log g_\theta(Y)\right|_{\theta=\theta_0}\) and
\(I_0=\E_0\{S_{\theta_0}(Y)S_{\theta_0}(Y)^\top\}\).
Define
\(\psi_{\mathrm{dir},x}(Y)=
\dot f_{\theta_0}(x)^\top I_0^{-1}S_{\theta_0}(Y)\).
Also let
\(\dot{\mathcal R}_{0,\theta_0}(x)=
\left.\partial_\theta\mathcal R_0(f_\theta)(x)\right|_{\theta=\theta_0}
=\Pzero\dot q_{\theta_0}(x\mid Y)\)
and define
\(\psi_{\post,x}(Y)=q_{\theta_0}(x\mid Y)-f_0(x)
+\dot{\mathcal R}_{0,\theta_0}(x)^\top
I_0^{-1}S_{\theta_0}(Y)\).

Set \(\mathfrak f_{\mathrm{dir}}=\widehat f\) and
\(\mathfrak f_{\post}=\widehat f_{\post}\). Then, for every fixed \(x\) at
which the relevant derivatives exist and each
\(r\in\{\mathrm{dir},\post\}\),
\[
\sqrt n\{\mathfrak f_r(x)-f_0(x)\}
=
n^{-1/2}\sum_{i=1}^n\psi_{r,x}(Y_i)+o_p(1).
\]
Consequently,
\(\sqrt n\{\mathfrak f_r(x)-f_0(x)\}
\Rightarrow N(0,\Var_0\{\psi_{r,x}(Y)\})\)
for \(r\in\{\mathrm{dir},\post\}\).

Moreover, define
\(A_x=\E_0\{q_{\theta_0}(x\mid Y)S_{\theta_0}(Y)\}\) and
\(\xi_x(Y)=q_{\theta_0}(x\mid Y)-f_0(x)
-A_x^\top I_0^{-1}S_{\theta_0}(Y)\).
Then
\(\dot{\mathcal R}_{0,\theta_0}(x)=\dot f_{\theta_0}(x)-A_x\),
\(\psi_{\post,x}=\psi_{\mathrm{dir},x}+\xi_x\), and
\(\E_0\{\xi_x(Y)S_{\theta_0}(Y)\}=0\). Hence
\[
\Var_0\{\psi_{\post,x}(Y)\}
=
\Var_0\{\psi_{\mathrm{dir},x}(Y)\}
+
\Var_0\{\xi_x(Y)\}.
\]
\end{theorem}

\begin{proof}
See Section~\ref{supp-app:main-proofs} of the Supplementary Material.
\end{proof}

Thus, in a correctly specified regular fixed-dimensional model, posterior
reconstruction cannot improve the first-order pointwise efficiency of the
direct likelihood estimator. Its potential advantage under
misspecification is a reduction of projection bias, obtained at the cost of
the additional residual sampling variance above.

\section{Growing-sieve theory}
\label{sec:growing-sieve-theory}
The growing-sieve theory is organized around three linked components. We first establish an observed-domain rate for the penalized convolution-likelihood estimator. We then use inverse-stability arguments to obtain the corresponding latent-density rate. Finally, we quantify the empirical fluctuation introduced by posterior averaging and identify conditions under which posterior reconstruction preserves the first-stage convergence order.

\subsection{Observed-domain likelihood rate}
\label{sec:primitive-sieve}

The fixed-dimensional results above do not rely on a specific choice of
mixture components. We now return to the regularized Gaussian-mixture sieve
\(\mathcal F_n\) defined in~\eqref{Fn}, allowing its complexity to increase
with \(n\). To obtain consistency over richer classes, the sieve must expand with \(n\),
while its growth remains sufficiently slow for the likelihood, entropy, and
inverse-stability arguments to apply. For densities \(p\) and \(q\), write
\(
\KL(p,q)=\int p\log(p/q),
\) and \(
\V(p,q)=\int p\{\log(p/q)\}^2,
\)
whenever these quantities are well defined.
\begin{assumption}[Gaussian-sieve approximation]
\label{ass:primitive-approximation}
There exists a sequence \(f_n^\circ\in\mathcal F_n\) such that
    $\KL(g_{f_0},g_{f_n^\circ})
    \le
    b_n^2,
    b_n\to0.$
\end{assumption}
\noindent The approximation is stated on the observed density space because this is the domain of the likelihood function; standard Gaussian-mixture approximation results support this condition \citep{ghosalvdv2001,kruijer2010}. Proposition~\ref{supp-prop:supp-obs-scale-appr} of the
Supplementary Material shows that this condition follows from a latent
Gaussian-mixture approximation.
To complement the approximation condition above, we bound the empirical
process over the full class of log-likelihood differences, avoiding a
separate treatment of a central region of the observation space and its
complement.
\begin{assumption}[Global log-likelihood complexity]
\label{ass:primitive-ulln}
Let
\(\mathcal D_n=\{\ell_f-\ell_{f_n^\circ}:f\in\mathcal F_n\}\).
Assume that \(\mathcal D_n\) is pointwise measurable and has a measurable
envelope \(H_n\) satisfying
\(|d(y)|\leq H_n(y)\) for every \(d\in\mathcal D_n\), with
\(\|H_n\|_{L^2(\Pzero)}\leq F_n\) and \(F_n\geq1\).
For some \(V_n\geq1\), \(L_n\geq e\), and \(C>0\), suppose that
$\log N_{[]}\!\left(
\epsilon F_n,\mathcal D_n,L^2(\Pzero)
\right)
\leq
C V_n\log\!\left(\frac{L_n}{\epsilon}\right),
0<\epsilon<1.$
Set \(\Gamma_n=V_n\log L_n\) and
\(c_n=F_n\{\sqrt{\Gamma_n/n}+\Gamma_n/n\}\), and assume that
\(c_n\to0\).
\end{assumption}
The standard \(L^2(\Pzero)\)-bracketing entropy integral gives the square-root term in \(c_n\); since
\(c_n\) also contains a nonnegative linear remainder, it follows that
\(
\sup_{f\in\mathcal F_n}
\left|
(\Pn-\Pzero)(\ell_f-\ell_{f_n^\circ})
\right|
=O_p(c_n).
\)

\begin{assumption}[Penalty calibration]\label{ass:primitive-penalty}The penalty is nonnegative and satisfies
$\frac{\operatorname{pen}_n(f_n^\circ)}{n}
    =
    o(1).$
\end{assumption}
The penalty may regularize the finite-sample search, provided that it remains
asymptotically negligible along the approximating sequence. For the
iterated-log BIC,
\(\operatorname{pen}_n(f_n^\circ)
=(d_{K_n^\circ}/2)\mathsf L^{\circ 3}(n)\log n\),
so Assumption~\ref{ass:primitive-penalty} holds whenever
\(d_{K_n^\circ}\mathsf L^{\circ 3}(n)\log n=o(n)\).
We collect the approximation, empirical-process, and penalty
contributions in
\[
    e_n
    =
    b_n
    +
    c_n^{1/2}
    +
    \left\{
        \frac{\operatorname{pen}_n(f_n^\circ)}{n}
    \right\}^{1/2}.
\]

\begin{proposition}[Observed-density likelihood rate]
\label{prop:observed-kl-rate}
Suppose that Assumptions~\ref{ass:primitive-approximation}-- \ref{ass:primitive-penalty} hold. Then
    $\KL(g_{f_0},g_{\widehat f})^{1/2}
    =
    O_p(e_n).$
\end{proposition}

\begin{proof}
See Section~\ref{supp-app:main-proofs} of the Supplementary Material.
\end{proof}

Under Gaussian measurement error, Proposition~\ref{supp-prop:gaussian-global-likelihood} verifies
Assumption~\ref{ass:primitive-ulln}; hence the displayed likelihood rate does not require an additional
empirical-process assumption on observation tails.
The proposition is stated for the exact direct sieve estimator;
the effect of approximate model optimization is handled in Section~\ref{supp-supp_EMoptim} of the Supplementary Material.

The likelihood analysis follows the general architecture of sieve estimation and likelihood-ratio localization developed in, among others, \cite{wongshen1995}, \cite{ghosalvdv2001}, and \cite{chen2007}. In the present problem, however, the sieve is imposed on the latent density, whereas the likelihood and its empirical complexity are evaluated after convolution. A central additional step is therefore to establish the required likelihood
bounds for the induced convolution class. Under Gaussian measurement error, this verification can be carried out directly over the whole real line, without introducing a separate tail condition.

\subsection{Inverse stability and latent-density rates}
\label{sec:latent-rate-transfer}
We now transfer the observed-density rate from \(g_{\widehat f}\) to the latent estimator \(\widehat f\).
The likelihood rate in Proposition~\ref{prop:observed-kl-rate} is stated
in Kullback--Leibler distance on the observed domain. Under the mild
condition that \(f_\varepsilon\) is globally Lipschitz,
Lemma~\ref{supp-lem:supp-kl-to-l2} of the Supplementary Material converts
this rate into
\[
\|g_{\widehat f}-g_{f_0}\|_1=O_p(e_n),
\qquad
\|g_{\widehat f}-g_{f_0}\|_2=O_p(e_n),
\qquad
\|g_{\widehat f}-g_{f_0}\|_\infty=O_p(e_n^{1/2}).
\]
The condition holds for both error laws considered explicitly below:
one may take
\(L_\varepsilon=e^{-1/2}(2\pi)^{-1/2}\sigma_\varepsilon^{-2}\)
for Gaussian error and \(L_\varepsilon=(2b^2)^{-1}\) for Laplace error.

The next step is to translate the observed-domain error into an error for
the latent density. This is the deconvolution part of the argument, and
therefore we consider the cost of inverting the convolution operator up to
frequency \(T\),
\[
\kappa_\varepsilon(T)
=
\left\{
\inf_{|t|\le T}|\varphi_\varepsilon(t)|
\right\}^{-1}.
\]

The inverse-stability argument below balances two terms. On low frequencies,
$|t|\le T$, the observed $L^2$ error can be transferred to the latent space at the
price $\kappa_\varepsilon(T).$ High frequencies, \(|t|>T\), are controlled by Sobolev smoothness rather than by direct inversion. For
$s>0$, write 
$J_s(f)^2
=
\int
(1+t^2)^s|\varphi_f(t)|^2dt$
for the squared $H^s$-Sobolev norm. Let
$R_0$ denote a Sobolev radius for the true latent density $f_0$, and
$R_n$ a possibly growing Sobolev radius for the sieve
$\mathcal F_n$. Proposition~\ref{prop:latent-rate-transfer} transfers the observed
\(L^2\) rate supplied by
Lemma~\ref{supp-lem:supp-kl-to-l2} of the Supplementary Material to a
latent \(L^1\) bound, first through low-frequency Fourier inversion and
then through truncation and tail control.

The low- and high-frequency decomposition underlying Proposition~\ref{prop:latent-rate-transfer} is standard in classical deconvolution theory (see \cite{carrollhall1988}, \cite{fan1991}, and \cite{meister2009}). Unlike Fourier-based estimators, which are defined through a regularized inverse transform, our estimator is obtained directly from the convolution likelihood. The frequency decomposition is introduced only afterward, as a stability tool for transferring an observed-domain rate to the latent domain. This modular formulation separates the statistical analysis of the convolution likelihood from the analytic cost of inversion.

\begin{proposition}[Latent-density rate transfer]
\label{prop:latent-rate-transfer}
Assume that:
\begin{itemize}
    \item[(i)] \(\varphi_\varepsilon(t)\neq0\) \(\forall\) \(t\);

    \item[(ii)] \(f_0\) satisfies \(J_s(f_0)\le R_0\) for some \(s>1/2\);

    \item[(iii)] \(J_s(f)\le R_n\) \(\forall\)  \(f\in\mathcal F_n\);

    \item[(iv)] there exists a deterministic sequence \(d_n\) such that
    \(
        \|g_{\widehat f}-g_{f_0}\|_2
        =
        O_p(d_n);
    \)

    \item[(v)] there exist \(A_n\to\infty\) and \(\eta_n\to0\) such that
    \(
        \sup_{f\in\mathcal F_n}
        \int_{|x|>A_n}f(x)\,dx
        +
        \int_{|x|>A_n}f_0(x)\,dx
        \le
        2\eta_n .
    \)
\end{itemize}
Then, for every deterministic \(T_n\to\infty\),
\begin{align*}
\|\widehat f-f_0\|_2
&=
O_p\!\left[
\kappa_\varepsilon(T_n)d_n+(R_n+R_0)T_n^{-s}
\right],\\
\|\widehat f-f_0\|_1
&=
O_p\!\left[
A_n^{1/2}
\left\{
\kappa_\varepsilon(T_n)d_n+(R_n+R_0)T_n^{-s}
\right\}
+\eta_n
\right].
\end{align*}
In particular, if \(T_n\) and \(A_n\) can be chosen so that the second displayed
right-hand side tends to zero, then
\(
    \|\widehat f-f_0\|_1=o_p(1).
\)
If the first displayed right-hand side tends to zero, then also
\(
    \|\widehat f-f_0\|_2=o_p(1).
\)
\end{proposition}

\begin{proof}
See Section~\ref{supp-app:main-proofs} of the Supplementary Material.
\end{proof}
Combining Proposition~\ref{prop:latent-rate-transfer} with
Lemma~\ref{supp-lem:supp-kl-to-l2} allows \(d_n\) to be taken of the
same order as the observed-domain rate \(e_n\). The resulting \(L^1\)
bound separates the inverse cost \(\kappa_\varepsilon(T_n)\), the Fourier
truncation bias \((R_n+R_0)T_n^{-s}\), and the latent-tail term
\(\eta_n\).

For the Gaussian-mixture sieve,
Lemma~\ref{supp-lem:gaussian-sieve-budgets} of the Supplementary Material
provides the required Sobolev and tail bounds. Under exponential inverse
growth, Corollary~\ref{supp-cor:supp-generic-supersmooth-transfer} gives
the corresponding explicit choice of \(T_n\) and the resulting
supersmooth transfer rates.
A Gaussian-mixture-specific refinement can separate approximation
bias from stochastic inverse amplification. In particular,
Lemma~\ref{supp-lem:supp-bias-separated-inverse} of the Supplementary
Material applies the inverse factor only to the observed-domain estimation
error around the approximating density \(f_n^\circ\). These refinements are
developed in Section~\ref{supp-app:fourier-inverse-material} of the
Supplementary Material.

\subsection{Rate preservation under posterior empirical stability}
\label{sec:posterior-rate-strong}

The posterior reconstruction admits the decomposition
\[
\widehat f_{\post}-f_0
=
\underbrace{
\mathcal R_n(\widehat f)-\mathcal R_0(\widehat f)
}_{\text{empirical posterior fluctuation}}
+
\underbrace{
\mathcal R_0(\widehat f)-f_0
}_{\text{population reconstruction error}}.
\]
The first term is generated by empirical posterior averaging, whereas the
second describes the population effect of applying the reconstruction map
to the first-stage estimator.

\begin{assumption}[Posterior empirical stability]
\label{ass:primitive-posterior-stability}
There exists a sequence \(\tau_n\), with \(\tau_n\to0\) for consistency,
such that
\[
\sup_{f\in\mathcal F_n}
\|\mathcal R_n(f)-\mathcal R_0(f)\|_1
=
O_p(\tau_n).
\]
Equivalently,
\(\sup_{f\in\mathcal F_n}\|\Pn q_f-\Pzero q_f\|_1=O_p(\tau_n)\).
\end{assumption}

Posterior empirical stability is logically separate from the
log-likelihood empirical-process condition in
Assumption~\ref{ass:primitive-ulln}. For the regularized Gaussian-mixture
sieve, a truncated sufficient verification is given in
Proposition~\ref{supp-prop:posterior-stability-gaussian} of the
Supplementary Material. The fixed-scale Gaussian branch considered in
Section~\ref{sec:explicit-sieve-rates} instead admits a direct global
verification based on Gaussian conditional structure.

The population component is controlled by the following deterministic
inequality.

\begin{lemma}[Population posterior stability]
\label{lem:R-stability}
For every density \(f\) such that \(g_f>0\)
\(g_{f_0}\)-almost everywhere,
\(\|\mathcal R_0(f)-f\|_1\le\|g_f-g_{f_0}\|_1\).
Consequently,
\(\|\mathcal R_0(f)-f_0\|_1\le2\|f-f_0\|_1\).
\end{lemma}

\begin{proof}
See Section~\ref{supp-app:main-proofs} of the Supplementary Material.
\end{proof}

Combining the population bound with Assumption~\ref{ass:primitive-posterior-stability}
gives the posterior-density rate.

\begin{theorem}[Posterior-density rate]
\label{thm:posterior-rate}
Suppose that
\(\|\widehat f-f_0\|_1=O_p(r_n)\)
and that Assumption~\ref{ass:primitive-posterior-stability} holds with
rate \(\tau_n\). Then
\[
\|\widehat f_{\post}-f_0\|_1
=
O_p(r_n+\tau_n).
\]
\end{theorem}

\begin{proof}
See Section~\ref{supp-app:main-proofs} of the Supplementary Material.
\end{proof}

Thus, if \(\tau_n=O(r_n)\), posterior reconstruction preserves the
first-stage latent \(L^1\) rate. The abstract result separates the rate
\(r_n\) inherited from likelihood estimation and inverse stability from
the additional empirical posterior fluctuation \(\tau_n\).

The same population-stability inequality gives a direct comparison
between the two estimators:
\[
\|\widehat f_{\post}-\widehat f\|_1
\le
\|\mathcal R_n(\widehat f)-\mathcal R_0(\widehat f)\|_1
+
\|g_{\widehat f}-g_{f_0}\|_1.
\]
Hence, if Assumption~\ref{ass:primitive-posterior-stability} holds with
rate \(\tau_n\) and
\(\KL(g_{f_0},g_{\widehat f})^{1/2}=O_p(e_n)\), Pinsker's inequality
implies
\(\|\widehat f_{\post}-\widehat f\|_1=O_p(\tau_n+e_n)\).
In particular, the two estimators are asymptotically equivalent in
\(L^1\) whenever \(\tau_n\to0\) and \(e_n\to0\), although they may still
differ in finite samples or under fixed-model misspecification.

Section~\ref{sec:explicit-sieve-rates} now specializes the three modules
above---observed-domain likelihood control, inverse stability, and
posterior empirical stability---to Gaussian and ordinary-smooth
measurement error. 

\section{Explicit rates under Gaussian and ordinary-smooth errors}
\label{sec:explicit-sieve-rates}
This section specializes the preceding growing-sieve theory to Gaussian and ordinary-smooth measurement error and derives explicit convergence rates for the direct and posterior estimators. The main analysis focuses on latent-domain recovery and on whether posterior reconstruction preserves the first-stage convergence rate. Complementary observed-domain results for Gaussian-mixture sieves with shrinking component scales are provided in Propositions~\ref{supp-prop:supp-obs-scale-appr} and~\ref{supp-prop:supp-self-contained-en} of the Supplementary Material.

\subsection{Rates under Gaussian error for a Gaussian-smoothed latent class}
\label{sec:gaussian-smoothed-end-to-end}

Assume Gaussian measurement error. Fix
\(0<\underline\tau\leq\overline\tau<\infty\), \(c_0>0\), and
\(C_0\geq1\), and let \(\phi_\tau\) denote the
\(N(0,\tau^2)\) density. Consider the latent class
\[
\mathcal C_G
=
\left\{
f_0=\phi_\tau*Q_0:
\begin{array}{l}
\tau\in[\underline\tau,\overline\tau],\\
Q_0([-t,t]^c)\leq C_0e^{-c_0t^2},\quad t\geq0
\end{array}
\right\}.
\]
This class is infinite-dimensional because the mixing distribution \(Q_0\)
is unrestricted apart from the uniform subgaussian tail bound.

We use the fixed-scale sieve
\[
\mathcal F_n^G
=
\bigcup_{K\leq K_n}
\left\{
\sum_{k=1}^K
\pi_k\phiN(\cdot;\mu_k,\sigma_k^2):
\begin{array}{l}
\pi_k\geq\pi_{\min,n},\quad \sum_{k=1}^K\pi_k=1,\\
|\mu_k|\leq M_n,\quad
\underline\sigma_G\leq\sigma_k\leq\overline\sigma_G
\end{array}
\right\},
\]
where \(0<\underline\sigma_G\leq\underline\tau\) and
\(\overline\sigma_G\geq\overline\tau\) are fixed constants. The lower scale bound supplies the analytic regularity used in the Hölder-type inverse-stability bound developed in the Supplementary Material. The estimator is the penalized convolution-likelihood maximizer over \(\mathcal F_n^G\). The proof strategy separates observed-scale approximation, likelihood control, analytic inversion, latent-tail conversion, and the empirical fluctuation introduced by posterior averaging. 
The required approximation, likelihood, inverse-stability, latent-tail, and posterior-complexity bounds are established in the Supplementary Material. Together, these ingredients yield the two Gaussian rate results below.

The remaining fixed-scale Gaussian ingredients are deferred to the
Supplementary Material. Lemma~\ref{supp-lem:supp-gaussian-analytic-stability}
gives the Hölder-type inverse bound from observed \(L^2\) error to latent
\(L^2\) error, while
Lemma~\ref{supp-lem:supp-gaussian-smoothed-l1-transfer} converts the
resulting latent \(L^2\) rate to \(L^1\). Gaussian conjugacy yields the
posterior-kernel regularity in
Lemma~\ref{supp-lem:supp-gaussian-posterior-lipschitz}, which in turn
gives the empirical posterior bound in
Proposition~\ref{supp-prop:supp-gaussian-smoothed-posterior-complexity}.
These results provide the inverse-stability and posterior-fluctuation
inputs for the two Gaussian theorems below.

\begin{theorem}[Global Gaussian observed- and latent-domain rates]
\label{thm:explicit-sieve-rates}
Suppose that \(f_0\in\mathcal C_G\), the measurement error is Gaussian,
and the estimator is computed over \(\mathcal F_n^G\) with
\(K_n=\max\{2,\lfloor n^{1/5}(\log n)^{-1/5}\rfloor\}\),
\(M_n=C_M\sqrt{\log n}\), and \(\pi_{\min,n}=n^{-u}\), where
\(u\geq3/5\), \(K_n\pi_{\min,n}\leq1/2\) for all sufficiently large
\(n\), and \(C_M\) is sufficiently large. Assume exact global maximization of the penalized convolution-likelihood
criterion with the iterated-log BIC penalty, and let
\(\vartheta=\underline\sigma_G^2/(\underline\sigma_G^2+\sigma_\varepsilon^2)\). Then
\[
\begin{aligned}
\KL(g_0,g_{\widehat f})^{1/2}
&=O_p\!\left\{n^{-1/5}(\log n)^{7/10}\right\},
&
\|g_{\widehat f}-g_0\|_2
&=O_p\!\left\{n^{-1/5}(\log n)^{7/10}\right\},\\
\|\widehat f-f_0\|_2
&=O_p\!\left\{n^{-\vartheta/5}(\log n)^{19\vartheta/20-1/4}\right\},\\
\|\widehat f-f_0\|_1
&=O_p\!\left\{n^{-\vartheta/5}(\log n)^{19\vartheta/20}\right\},
&
\|\widehat f_{\post}-f_0\|_1
&=O_p\!\left\{n^{-\vartheta/5}(\log n)^{19\vartheta/20}\right\}.
\end{aligned}
\]
\end{theorem}
\begin{proof}
See Section~\ref{supp-app:main-proofs} of the Supplementary Material.
\end{proof}
The exponent \(\vartheta\) quantifies the inverse-stability cost of Gaussian deconvolution. The observed-domain rate is transferred to the latent domain through the analytic inverse bound, while posterior reconstruction preserves the direct \(L^1\) order because its empirical term, of order \(n^{-2/5}(\log n)^{2/5}\), has strictly smaller polynomial order.

\begin{remark}
\label{rem:gaussian-smoothed-approximate-optimization}
If the normalized penalized-objective gap \(\Delta_n\), defined in
Section~\ref{supp-supp_EMoptim} of the Supplementary Material, satisfies
\(\Delta_n=O_p(a_n)\), set
\(\bar d_n=n^{-1/5}(\log n)^{7/10}+a_n^{1/2}\). Then the observed
\(L^2\) rate is \(O_p(\bar d_n)\). The inverse and latent-tail bounds in
Lemmas~\ref{supp-lem:supp-gaussian-analytic-stability}
and~\ref{supp-lem:supp-gaussian-smoothed-l1-transfer}, together with
Proposition~\ref{supp-prop:supp-gaussian-smoothed-posterior-complexity},
give the corresponding latent and posterior rates. In particular, the
rates in Theorem~\ref{thm:explicit-sieve-rates} are preserved if
\(a_n^{1/2}=O\{n^{-1/5}(\log n)^{7/10}\}\). The result is nonparametric
because \(Q_0\) is infinite-dimensional and \(K_n\to\infty\), and it
assumes a latent Gaussian smoothing scale bounded away from zero. The
theorem applies to implementations satisfying the stated sieve constraints
and penalized criterion.
\end{remark}

\subsubsection{Improved Gaussian rate via Hellinger localization}
\label{sec:gaussian-sharp-rate}

Write $d_{\mathrm H}$ for Hellinger distance. 
The preceding Gaussian result controls the full log-likelihood process. A sharper observed-domain rate is obtained by localizing the likelihood in Hellinger distance and balancing the resulting local complexity against the approximation error. The technical approximation, bracketing, and localization bounds are established in the Supplementary Material (see Proposition~\ref{supp-prop:supp-gaussian-smoothed-klv-approximation}, Lemma~\ref{supp-lem:penalized-likelihood-localization} and Lemma~\ref{supp-lem:gaussian-fixed-scale-hellinger}).


\begin{theorem}[Improved Gaussian observed- and latent-domain rates]
\label{thm:gaussian-sharp-end-to-end}
Suppose that \(f_0\in\mathcal C_G\), the measurement error is Gaussian, and the estimator is computed over \(\mathcal F_n^G\). Write \(L_3(n)=\mathsf L^{\circ3}(n)\), and let
\begin{equation}
\label{eq:gaussian-sharp-calibration}
K_n
=
\max\left\{
2,
\left\lfloor
\left\{\frac{n}{L_3(n)}\right\}^{1/3}
\right\rfloor
\right\},
\quad
M_n=C_M\sqrt{\log n},
\quad
\pi_{\min,n}=n^{-u},
\quad u>1.
\end{equation}
Choose \(C_M\) sufficiently large and assume exact global maximization of
the penalized convolution-likelihood criterion with the iterated-log BIC
penalty. Define
\(
r_n^\sharp
=
n^{-1/3}(\log n)^{1/2}\{L_3(n)\}^{1/3},
\)
and
\(
\vartheta
=
\underline\sigma_G^2/(\underline\sigma_G^2+\sigma_\varepsilon^2).
\)
Then
\begin{equation}
\label{eq:gaussian-sharp-observed-conclusion}
\dH(g_{\widehat f},g_0)=O_p(r_n^\sharp),
\qquad
\|g_{\widehat f}-g_0\|_2=O_p(r_n^\sharp).
\end{equation}
Moreover,
\[
\begin{aligned}
\|\widehat f-f_0\|_2
&=
O_p\!\left[
n^{-\vartheta/3}
(\log n)^{3\vartheta/4-1/4}
\{L_3(n)\}^{\vartheta/3}
\right],\\
\|\widehat f-f_0\|_1
&=
O_p\!\left[
n^{-\vartheta/3}
(\log n)^{3\vartheta/4}
\{L_3(n)\}^{\vartheta/3}
\right],\\
\|\widehat f_{\post}-f_0\|_1
&=
O_p\!\left[
n^{-\vartheta/3}
(\log n)^{3\vartheta/4}
\{L_3(n)\}^{\vartheta/3}
\right].
\end{aligned}
\]
\end{theorem}
\begin{proof}
See Section~\ref{supp-app:main-proofs} of the Supplementary Material.
\end{proof}

Under the stronger calibration in \eqref{eq:gaussian-sharp-calibration}, Hellinger localization improves the observed polynomial exponent from \(1/5\) to \(1/3\), up to the displayed logarithmic factors, while \(K_n\to\infty\). Under related Gaussian-mixture conditions, faster Hellinger rates are available for direct estimation of the observed density \citep{ghosalvdv2001,zhang2009}; \citet{kim2014} gives a minimax lower bound under Hellinger loss for normal location mixtures with sub-Gaussian mixing distributions. Here, the exponent \(1/3\) comes from the approximation--complexity balance of the regularized finite-mixture sieve and its penalty calibration; we do not claim that it is minimax optimal.

\begin{remark}
\label{rem:gaussian-sharp-optimization}
If the normalized penalized-objective gap satisfies \(\Delta_n=O_p(a_n^\sharp)\), set
$\overline r_n^\sharp=r_n^\sharp+(a_n^\sharp)^{1/2}$.
Then \(\|g_{\widehat f^{\mathrm{EM}}}-g_0\|_2=O_p(\overline r_n^\sharp)\), and the corresponding latent-domain bounds apply with
\(\delta_n=\overline r_n^\sharp\). The posterior rate additionally contains
\(O_p\{(K_n\log n/n)^{1/2}\}
=
O_p\{n^{-1/3}(\log n)^{1/2}L_3(n)^{-1/6}\}.\)
The rates in Theorem~\ref{thm:gaussian-sharp-end-to-end} are preserved whenever \((a_n^\sharp)^{1/2}=O(r_n^\sharp)\). The stronger condition \(u>1\) makes the error induced by the weight floor negligible relative to \((r_n^\sharp)^2\). The exponent \(1/3\) results from balancing the discretization term \(M_n/K_n\) with the square root of the normalized iterated-log BIC penalty.
\end{remark}

\subsection{Ordinary-smooth errors}
\label{sec:localized-ordinary-concrete}
For ordinary-smooth errors, we separate finite observed-scale approximation,
localized likelihood fluctuation, inverse amplification, and posterior empirical
averaging. This yields the classical polynomial deconvolution exponent under
modular conditions, which are verified constructively for Laplace error below.

For \(0<h<e^{-1}\), put \(L_h=\log(e/h)\) and define the local regularized sieve
\begin{equation}
\mathcal F_h^{\mathrm{reg}}
=
\bigcup_{1\le K\le K_h}
\left\{
\sum_{j=1}^K\pi_j\phiN(\,\cdot\,;\mu_j,\sigma_j^2):
\pi_j\ge\pi_{\min,h},\ \sum_{j=1}^K\pi_j=1,
\ |\mu_j|\le M_h,\ h\le\sigma_j\le\bar\sigma
\right\}.
\label{eq:local-regularized-sieve}
\end{equation}
For a deterministic resolution sequence \(h_n\), identify the sample-size-indexed
sieve of Section~\ref{sec:methodology} with the local sieve by setting
\(\mathcal F_n=\mathcal F_{h_n}^{\mathrm{reg}}\), \(K_n=K_{h_n}\),
\(M_n=M_{h_n}\), \(\pi_{\min,n}=\pi_{\min,h_n}\), and
\(\sigma_{\min,n}=h_n\). Under this calibration, \(\widehat f_{h_n}\) is the
first-stage estimator in \eqref{eq:penalized-sieve-estimator}, and
\(\widehat f_{\post,h_n}\) is the posterior reconstruction in
\eqref{eq:posterior-density-estimator}. Assume \(K_h\pi_{\min,h}\le1/2\),
\(K_h\le Ch^{-1}L_h^{a_K}\), \(M_h\asymp L_h\), and
\(0<\bar\sigma<\infty\). For \(f\in\mathcal F_h^{\mathrm{reg}}\), write
\(g_f=f*f_\varepsilon\) and retain the posterior kernel \(q_f\) defined in
Section~\ref{sec:methodology}. Let \(\widehat f_h\) maximize the penalized
convolution likelihood in \eqref{eq:penalized-sieve-estimator} over
\(\mathcal F_h^{\mathrm{reg}}\), and let \(\widehat f_{\post,h}\) be the
posterior reconstruction computed from \(\widehat f_h\).

\begin{assumption}[Ordinary-smooth inversion]
\label{ass:local-os}
For some \(\beta>0\) and constants \(0<c_\varepsilon<C_\varepsilon<\infty\),
\(c_\varepsilon(1+|t|)^{-\beta}\le|\varphi_\varepsilon(t)|\le
C_\varepsilon(1+|t|)^{-\beta}\) for all \(t\in\mathbb R\),
\(\varphi_\varepsilon(t)\ne0\), and \(f_\varepsilon\in L^\infty(\mathbb R)\).
\end{assumption}

\begin{assumption}[Finite observed-scale approximation]
\label{ass:local-ap}
For some \(s>0\) and finite \(a_A\ge0\), for every sufficiently small \(h\)
there exists \(f_h^\circ\in\mathcal F_h^{\mathrm{reg}}\) such that, with
\(g_h^\circ=g_{f_h^\circ}\),
\[
\KL(g_0,g_h^\circ)+\V(g_0,g_h^\circ)
\le C h^{2(s+\beta)}L_h^{2a_A}.
\]
\end{assumption}

\begin{assumption}[Localized observed-likelihood complexity]
\label{ass:local-ll}
Let \(\mathcal G_h=\{g_f:f\in\mathcal F_h^{\mathrm{reg}}\}\), write
\(d_{\mathrm H}\) for Hellinger distance, and let
\(H_{[]}(u,\mathcal G_h,d_{\mathrm H})\) be the corresponding bracketing
entropy. There are finite \(a_E,a_P\ge0\) and constants
\(c_0,c_1,c_2>0\) such that, with
\(r_{n,h}=\{K_hL_h^{a_E}/n\}^{1/2}\), one has
\(K_hL_h^{a_E}=o(n)\) and, for every sufficiently small
\(\epsilon\ge r_{n,h}\),
\[
\int_{c_0\epsilon^2}^{c_1\epsilon}
\{1+H_{[]}(u,\mathcal G_h,d_{\mathrm H})\}^{1/2}\,du
\le c_2\sqrt n\,\epsilon^2.
\]
The penalty along the approximating sequence satisfies
\(\operatorname{pen}_n(f_h^\circ)/n\le CK_hL_h^{a_P}/n\). When a numerical
solution replaces the global maximizer, the same condition is imposed after
adding the corresponding optimization-gap term.
\end{assumption}

\begin{assumption}[Posterior empirical complexity]
\label{ass:local-pe}
For a finite \(a_Q\) and some \(\gamma\ge0\),
\[
\sup_{f\in\mathcal F_h^{\mathrm{reg}}}
\|(\Pn-\Pzero)q_f\|_1
=O_p\{n^{-1/2}h^{-\gamma}L_h^{a_Q}\}.
\]
The supremum is interpreted in outer probability unless a measurable
separable version of the class has been fixed.
\end{assumption}

Under these conditions, the preceding components combine as follows.

\begin{theorem}[Localized ordinary-smooth rates for the two estimators]
\label{thm:local-ordinary-rates}
Suppose Assumptions~\ref{ass:local-os}--\ref{ass:local-ll} hold,
\(f_0\in H^s(\mathbb R)\), and \(K_h\asymp h^{-1}L_h^{a_K}\). Let
\(h_n=n^{-1/(2s+2\beta+1)}\), up to a fixed logarithmic calibration. Then,
for a finite \(C\),
\[
\|\widehat f_{h_n}-f_0\|_2
=O_p\!\left\{n^{-s/(2s+2\beta+1)}(\log n)^C\right\}.
\]
If, in addition, \(f_0\) has a subgaussian tail, \(M_h\le C_ML_h\),
and the component scales are bounded above by \(\bar\sigma\), then
\[
\|\widehat f_{h_n}-f_0\|_1
=O_p\!\left\{n^{-s/(2s+2\beta+1)}(\log n)^C\right\}.
\]
If moreover Assumption~\ref{ass:local-pe} holds and
\(\gamma\le\beta+1/2\), then
\[
\|\widehat f_{\post,h_n}-f_0\|_1
=O_p\!\left\{n^{-s/(2s+2\beta+1)}(\log n)^C\right\}.
\]
If \(\gamma<\beta+1/2\), the empirical posterior-averaging term has
strictly smaller polynomial order than the direct latent-density error.
\end{theorem}
\begin{proof}
See Section~\ref{supp-app:main-proofs} of the Supplementary Material.
\end{proof}

\subsubsection{Verification under Laplace measurement error}
\label{sec:concrete-laplace-ap}
Throughout this subsection, the measurement error has Laplace density, denoted as 
$k_b$. 
Its characteristic function is \(\varphi_\varepsilon(t)=(1+b^2t^2)^{-1}\), so the error is ordinary smooth of
order \(\beta=2\). For fixed constants
\(0<\tau_-\le\tau_+<\infty\) and \(c_0,C_0>0\), consider the class
\begin{equation}
\mathcal C_{\mathrm{sg}}
=\left\{f_0=\phiN_\tau*Q_0:\
\tau\in[\tau_-,\tau_+],\quad
Q_0([-t,t]^c)\le C_0e^{-c_0t^2}\ \text{for all }t\ge0\right\}.
\label{eq:concrete-class}
\end{equation}
Here \(\phi_\tau=\phi(\,\cdot\,;0,\tau^2)\) denotes the centered Gaussian density with standard 
deviation \(\tau\). The mixing law \(Q_0\) is an arbitrary probability measure and may be 
continuous, discrete, or mixed. Thus \(\mathcal C_{\mathrm{sg}}\) is infinite-dimensional: it 
consists of Gaussian-smoothed probability laws whose smoothing scale is bounded away from zero and 
whose mixing distribution has a uniform subgaussian tail. It is smoother than a generic Sobolev 
class.

The approximation, localized likelihood, and posterior empirical-process
conditions required by Theorem~\ref{thm:local-ordinary-rates} are verified
for Laplace measurement error in
Propositions~\ref{supp-prop:supp-concrete-laplace-ap}--
\ref{supp-prop:supp-concrete-laplace-pe} of the Supplementary Material.
A compatible sieve calibration is used in the corollary below.

\begin{corollary}[Laplace rates for the two estimators]
\label{cor:concrete-laplace-rates}
Fix \(s>0\) and suppose \(f_0\in\mathcal C_{\mathrm{sg}}\). Let the
measurement error be Laplace with scale \(b>0\), and compute the estimators
over \(\mathcal F_h^{\mathrm{reg}}\) with
\(K_h\asymp h^{-1}L_h^{a_K}\), \(M_h=C_ML_h\),
\(\pi_{\min,h}=h^{2s+6}\), and \(\bar\sigma\ge\tau_+\), where
\(a_K\ge0\) is finite. Let \(h_n=n^{-1/(2s+5)}\), up to a fixed
logarithmic calibration. Then, for a finite \(C\),
\[
\begin{aligned}
\|\widehat f_{h_n}-f_0\|_2
&=O_p\!\left\{n^{-s/(2s+5)}(\log n)^C\right\},\\
\|\widehat f_{h_n}-f_0\|_1
&=O_p\!\left\{n^{-s/(2s+5)}(\log n)^C\right\},\\
\|\widehat f_{\post,h_n}-f_0\|_1
&=O_p\!\left\{n^{-s/(2s+5)}(\log n)^C\right\}.
\end{aligned}
\]
The empirical posterior-averaging term is
\(O_p\{n^{-1/2}h_n^{-3/2}(\log n)^C\}
=O_p\{h_n^{s+1}(\log n)^C\}\), and is therefore smaller than the direct
latent rate by the factor \(h_n\), up to logarithmic terms.
\end{corollary}

\begin{proof}
See Section~\ref{supp-app:main-proofs} of the Supplementary Material. 
\end{proof}
\section{Simulation study}
\label{sec:simulation}
In this simulation study, we assess how accurate the proposed convolution-likelihood density and posterior estimators are, hereafter referred to as MIX and POST, and whether posterior reconstruction improves upon the direct estimator.
The main experiment uses a constrained Gaussian-mixture sieve and compares the proposed
estimators with 
Fourier deconvolution estimator, the quadratic-programming of
\citet{yang2020}, and the penalized-contrast projection estimator of
\citet{comte2006}, denoted by DEC, QP, and PC, respectively. We do not include other recent estimators because they address deconvolution
problems formulated under different structural information. For example,
\citet{liu2009}, \citet{guan2021}, and \citet{cai2025} represent the latent
density on a pre-specified compact interval, so the latent support is assumed
to be known in advance and can play an important role in the performance of the
estimator.
We consider the following thirteen latent distributions, all of which have been used in the deconvolution literature:

{\small
\begin{enumerate}
\item Standard Normal: $X\sim N(0,1)$, always used as baseline.
\item Student-$t_5$: $X\sim \sqrt{3/5}\,T_5$, with $T_5\sim t_5$, used by \cite{yang2020}.
\item Laplace: $X\sim \mathrm{Lap}(1,1/\sqrt{2})$, used by \cite{comte2006} and \cite{cai2025}.
\item Cauchy: $X\sim \mathrm{Cauchy}(0,1)$, used by \cite{comte2006} and \cite{cai2025}.
\item Nearly normal: $X\sim NN(4)$,    used by \cite{guan2021}.
\item Gamma: $X\sim \Gamma(5,1/\sqrt{5})$, proposed in \cite{yang2020}.
\item Comte chi-square: $X\sim \chi^2_3/\sqrt{6}$, used by \cite{comte2006}.
\item Cai chi-square: $X\sim \chi^2_4/\sqrt{8}$, used by \cite{cai2025}.
\item Beta: $X\sim \sqrt{39.2}\,\mathrm{Beta}(2,5)$, used by \cite{cai2025}.
\item Comte mixed gamma: $X\sim 5.48^{-1/2}\{0.4\,\Gamma(5,1)+0.6\,\Gamma(13,1)\}$, by \cite{comte2006}.
\item Cai mixed gamma: $X\sim 25.16^{-1/2}\{0.4\,\Gamma(5,1)+0.6\,\Gamma(13,1)\}$, by \cite{cai2025}.
\item Mixed normal G: $X\sim 0.6\,N(-2,1^2)+0.4\,N(2,0.8^2)$, used by \cite{guan2021}.
\item Mixed normal Y: $X\sim 0.8\,N(s_x^{-1},s_x^{-2})+0.2\,N(5s_x^{-1},s_x^{-2})$, where $s_x=\{0.8\cdot0.2\cdot4^2+1\}^{1/2}$, used by \cite{yang2020}.
\end{enumerate}
}

Most of these latent densities are not themselves finite Gaussian mixtures and therefore do not belong exactly to the fitted finite Gaussian-mixture models at any fixed candidate order. They can nevertheless be approximated as the sieve expands, so the experiment assesses both finite-sample estimation and robustness to sieve misspecification.

As in most deconvolution studies, we consider Gaussian and Laplace measurement
errors as standard examples of supersmooth and ordinary-smooth error densities,
respectively. For the perturbation levels, the error standard deviation is set
to $\sigma_\varepsilon=c\,\sigma_X$, with the noise-to-signal ratio
$c\in\{0.35,0.60\}$. The only exception is the Cauchy design, for which the
variance is not defined; in that case, we use unit scale and set
$\sigma_\varepsilon=c$. The observed-tail regimes for the POST denominators are summarized for the
Monte Carlo designs in Table~\ref{supp-tab:tail-verification} of the Supplementary Material.

For each combination of latent density, error distribution, 
noise level, and sample size $n\in\{200,500,1000,5000\}$, we consider $N=500$ replications. All restrictions defining the theoretical
Gaussian-mixture sieve in~\eqref{eq:regularized-parameter-space}, including the bounds on component weights, locations
and scales, and the maximum admissible number of components, are
explicitly imposed (see Table~\ref{supp-tab:numerical-safeguards} of the Supplementary Material).
For the proposed estimators, the number of
components is selected by criterion~\eqref{eq:bic}, which is numerically identical to the classical BIC because the considered sample sizes are below
\(\exp\{\exp(e)\}\). The likelihood is optimized numerically by EM; the theoretical rate statements concern a global
maximizer unless the optimization-gap condition (see Section~\ref{supp-supp_EMoptim} of the Supplementary Material) is verified. Numerical safeguards are detailed in Table~\ref{supp-tab:numerical-safeguards} of the Supplementary Material.
Performance is measured by the mean integrated squared error (MISE).
Complete scenario-level MISE values, together with their standard errors, are reported in Tables~\ref{supp-tab:symmetric-unimodal-mise-se}--\ref{supp-tab:bimodal-mise-se} of the Supplementary Material. Section \ref{GS} evaluates the proposed estimators based on the regularized Gaussian-mixture sieve and compares them with the benchmark deconvolution methods. Section \ref{SNS} complements this main analysis with a skew-normal sieve diagnostic for the asymmetric positive-support designs.

\subsection{Gaussian-mixture sieve results} \label{GS}
Table~\ref{tab:mise-win-error-noise} reports the average MISE and the percentage of wins for each estimator, aggregated by error distribution and noise level. POST attains the lowest average MISE and the highest win frequency in all four aggregations. Its win share is 36.5\% under Gaussian error and 30.8\% under Laplace error, and it increases from 32.7\% at $c=0.35$ to 34.6\% at $c=0.60$. Thus, POST improves both average risk and the scenario-wise ranking, with the clearest advantage under Gaussian error and stronger contamination.
\begin{table}[!htbp]
\centering
\caption{Average $10^3\times\mathrm{MISE}$ and percentage of wins by error law and noise level. In each row, bold values indicate the best estimator within each panel.}
\label{tab:mise-win-error-noise}
\footnotesize
\setlength{\tabcolsep}{3.2pt}
\renewcommand{\arraystretch}{1.1}
\begin{tabular}{lrrrrrr|rrrrr}
\toprule
& & \multicolumn{5}{c}{Average $10^3\times\mathrm{MISE}$} & \multicolumn{5}{c}{Wins} \\
\cmidrule(lr){3-7}\cmidrule(lr){8-12}
Group & No. & MIX & POST & DEC & PC & QP & MIX & POST & DEC & PC & QP \\
\midrule
Gaussian error & 104 & 13.70 & \textbf{12.25} & 16.78 & 14.79 & 15.25 & 17.3 & \textbf{36.5} & 12.5 & 26.9 & 6.7 \\
Laplace error & 104 & 11.42 & \textbf{10.11} & 12.63 & 10.91 & 12.45 & 22.1 & \textbf{30.8} & 15.4 & 26.9 & 4.8 \\
\midrule
\(c=0.35\) & 104 & 9.28 & \textbf{8.21} & 11.41 & 9.38 & 11.34 & 25.0 & \textbf{32.7} & 10.6 & 27.9 & 3.8 \\
\(c=0.60\) & 104 & 15.84 & \textbf{14.15} & 18.01 & 16.33 & 16.36 & 14.4 & \textbf{34.6} & 17.3 & 26.0 & 7.7 \\
\bottomrule
\end{tabular}
\end{table}

The effect of posterior reconstruction depends strongly on latent-density shape. POST improves upon MIX in all 64 asymmetric-unimodal scenarios, 49 of 64 bimodal scenarios, and 53 of 80 symmetric-unimodal scenarios, for 166 of 208 cases overall. The median gain is 10.1\%, whereas the mean is 3.1\%, because a few large losses depress the average. Gains are strongest for asymmetric-unimodal and mixed-gamma designs, where the Gaussian sieve is least well adapted to the latent shape. The losses in some symmetric-unimodal cases are consistent with added finite-sample variability when the first-stage approximation is already accurate.

Table~\ref{tab:mise-win-by-n} summarizes average performance and scenario-wise wins across sample sizes.
\begin{table}[!htbp]
\centering
\caption{Average $10^3\times\mathrm{MISE}$ and percentage of wins by sample size.}
\label{tab:mise-win-by-n}
\footnotesize
\setlength{\tabcolsep}{3.2pt}
\renewcommand{\arraystretch}{1.1}
\begin{tabular}{lrrrrr|rrrrr}
\toprule
& \multicolumn{5}{c}{Average $10^3\times\mathrm{MISE}$} & \multicolumn{5}{c}{Wins (\%)} \\
\cmidrule(lr){2-6}\cmidrule(lr){7-11}
$n$ & MIX & POST & DEC & PC & QP & MIX & POST & DEC & PC & QP \\
\midrule
200 & 25.00 & \textbf{21.70} & 21.92 & 22.53 & 22.91 & 17.3 & 17.3 & 23.1 & \textbf{34.6} & 7.7 \\
500 & 12.70 & \textbf{11.51} & 16.29 & 12.77 & 15.10 & 21.2 & \textbf{28.8} & 17.3 & \textbf{28.8} & 3.8 \\
1000 & 8.31 & \textbf{7.63} & 12.88 & 9.47 & 11.34 & 26.9 & \textbf{34.6} & 7.7 & 26.9 & 3.8 \\
5000 & 4.22 & \textbf{3.90} & 7.73 & 6.64 & 6.04 & 13.5 & \textbf{53.8} & 7.7 & 17.3 & 7.7 \\
\midrule
Overall & 12.56 & \textbf{11.18} & 14.71 & 12.85 & 13.85 & 19.7 & \textbf{33.7} & 13.9 & 26.9 & 5.8 \\
\bottomrule\end{tabular}\end{table}

Taken together, Tables~\ref{tab:mise-win-error-noise}--\ref{tab:mise-win-by-n} yield a strong empirical conclusion: POST attains the lowest average MISE under both error distributions, at both contamination levels, and at every sample size. It also records the highest overall percentage of wins, 33.7\%. The scenario-wise rankings remain heterogeneous at small and moderate sample sizes: PC leads at $n=200$, POST and PC tie at $n=500$, and POST leads at $n=1000$ and $n=5000$, with a pronounced advantage in the last case.
Paired replication-level comparisons for aggregate MISE summaries, reported in Table~\ref{supp-tab:pairwise-all-expanded} of the Supplementary Material, confirm that in the symmetric-unimodal case, MIX is preferred for the Gaussian baseline, whereas POST performs better for the heavy-tailed designs. In the asymmetric-unimodal setting, POST significantly improves on MIX in all scenarios, although DEC and PC often remain more competitive. In the bimodal case, results are more heterogeneous, whereas the clearest gains occur for the two mixed-gamma designs. 
Finally, Table~\ref{tab:selected-k-by-shape-n} summarizes the selected
number of Gaussian components by latent-density shape and sample size.
The selected order increases with $n$ while remaining moderate overall, with the clearest increases occurring for the heavy-tailed Cauchy and asymmetric positive-support densities. 
\begin{table}[!htbp]
\centering
\caption{Selected number of mixture components by latent-density shape and sample size. Entries are averages over latent densities within each class, error laws, noise levels and Monte Carlo replications.}
\label{tab:selected-k-by-shape-n}
\footnotesize
\renewcommand{\arraystretch}{1.10}
\begin{tabular}{lrrrr}
\toprule
Shape & $n=200$ & $n=500$ & $n=1000$ & $n=5000$ \\
\midrule
Sym. unimodal & 1.53 & 1.87 & 1.98 & 2.35 \\
Asym. unimodal & 1.62 & 1.98 & 2.24 & 2.87 \\
Bimodal & 1.54 & 1.79 & 1.91 & 2.01 \\
\midrule
Overall & 1.56 & 1.88 & 2.04 & 2.41 \\
\bottomrule
\end{tabular}
\end{table}
\subsection{Skew-normal sieve diagnostic for asymmetric designs}\label{SNS}
For asymmetric densities, we report a diagnostic experiment in which Gaussian components are replaced by
skew-normal components.  We write
\(\mathrm{MIX}^{G}\) and \(\mathrm{POST}^{G}\) for the Gaussian-sieve estimators and
\(\mathrm{MIX}^{SN}\) and \(\mathrm{POST}^{SN}\) for the skew-normal-sieve estimators. The
component-specific shape parameters are regularized through the Azzalini--Arellano-Valle penalized likelihood;
see \citet{azzalini} for more details. 

\begin{table}[!htbp]
\centering
\caption{Average $10^3\times\mathrm{MISE}$ and scenario-wise wins on the asymmetric scenarios common to the Gaussian-sieve and skew-normal-sieve experiments. Row minima/maxima are shown in bold.}
\label{tab:flex-asym-overall}
\footnotesize
\setlength{\tabcolsep}{5pt}
\renewcommand{\arraystretch}{1.10}
\begin{tabular}{lrrrrrrr}
\toprule
& $\mathrm{MIX}^{G}$ & $\mathrm{POST}^{G}$ & DEC & PC & QP & $\mathrm{MIX}^{SN}$ & $\mathrm{POST}^{SN}$ \\
\midrule
MISE & 19.53 & 15.23 & 14.03 & 13.52 & 15.99 & 8.14 & \textbf{6.53} \\
Wins & 0 & 0 & 1 & 11 & 0 & 6 & \textbf{46} \\
\bottomrule
\end{tabular}
\end{table}
We see from Table~\ref{tab:flex-asym-overall} that  replacing the
Gaussian components with skew-normal components substantially reduces
the average MISE, with \({\mathrm{POST}}^{SN}\) attaining the lowest
average MISE and winning in 46 of the 64 scenarios.
Table~\ref{tab:flex-asym-by-latent} reports the results separately for
each asymmetric latent density and shows where the gain from the
skew-normal sieve comes from. For Cai chi-square and Comte chi-square,
the flexible posterior estimator dominates almost uniformly and
substantially reduces the average MISE relative both to
\({\mathrm{POST}}^{G}\) and to the best classical benchmark. The Beta case is less clear-cut:
\({\mathrm{POST}}^{SN}\) has the lowest average MISE, but the classical
benchmarks remain competitive and obtain several scenario-wise wins,
as expected for a bounded asymmetric density. Gamma follows the same
qualitative pattern, with smaller gains and stronger competition from
\({\mathrm{MIX}}^{SN}\) in some configurations, consistent with its
more moderate skewness.
\begin{table}[!htbp]
\centering
\caption{Average $10^3\times\mathrm{MISE}$ by asymmetric latent density. ``Best benchmark'' is the scenario-wise minimum among DEC, PC, and QP, averaged within each latent design. ``$\mathrm{POST}^{SN}$ wins'' counts the cases in which $\mathrm{POST}^{SN}$ is strictly the smallest-MISE method.}
\label{tab:flex-asym-by-latent}
\footnotesize
\setlength{\tabcolsep}{4pt}
\renewcommand{\arraystretch}{1.10}
\begin{tabular}{lrrrrrr}
\toprule
Latent density & $\mathrm{POST}^{G}$ & Best benchmark & $\mathrm{MIX}^{SN}$ & $\mathrm{POST}^{SN}$ & $\mathrm{POST}^{SN}$ wins & Mean $\widehat K$ \\
\midrule
Gamma & 5.01 & 4.48 & 3.10 & \textbf{2.68} & 6/16 & 1.25 \\
Comte chi-square & 31.59 & 24.79 & 16.02 & \textbf{13.03} & 16/16 & 1.54 \\
Cai chi-square & 17.46 & 12.63 & 7.49 & \textbf{5.89} & 16/16 & 1.45 \\
Beta & 6.85 & 5.17 & 5.96 & \textbf{4.54} & 8/16 & 1.28 \\
\midrule
Overall & 15.23 & 11.77 & 8.14 & \textbf{6.53} & 46/64 & 1.38 \\
\bottomrule
\end{tabular}
\end{table}
Results in Table~\ref{tab:flex-asym-by-n} confirm that the improvement persists
as \(n\) increases. At every sample size, \({\mathrm{POST}}^{SN}\)
has the lowest average MISE and improves upon \(\mathrm{MIX}^{SN}\)
in at least 14 of the 16 scenarios. 
The mean number of selected skew-normal mixture
components increases
with \(n\), indicating that the skew-normal sieve becomes gradually
more flexible while remaining moderate in size.
\begin{table}[!htbp]
\centering
\caption{Average $10^3\times\mathrm{MISE}$ by sample size. ``Best benchmark'' is the scenario-wise minimum among DEC, PC, and QP, averaged within each sample size. ``POST/MIX wins'' counts the cases in which $\mathrm{POST}^{SN}$ has smaller MISE than $\mathrm{MIX}^{SN}$. ``Gain vs MIX'' is the average scenario-wise percentage reduction in MISE.}
\label{tab:flex-asym-by-n}
\footnotesize
\setlength{\tabcolsep}{3.5pt}
\renewcommand{\arraystretch}{1.10}
\begin{tabular}{rrrrrrrr}
\toprule
$n$ & $\mathrm{POST}^{G}$ & Best benchmark & $\mathrm{MIX}^{SN}$ & $\mathrm{POST}^{SN}$ & POST/MIX wins & Gain vs MIX & Mean $\widehat K$ \\
\midrule
200 & 30.18 & 18.49 & 13.86 & \textbf{11.05} & 14/16 & 16.2\% & 1.04 \\
500 & 19.33 & 12.93 & 8.88 & \textbf{6.95} & 14/16 & 18.8\% & 1.19 \\
1000 & 14.78 & 10.26 & 6.71 & \textbf{5.36} & 14/16 & 19.9\% & 1.34 \\
5000 & 8.81 & 6.35 & 3.11 & \textbf{2.77} & 16/16 & 13.6\% & 1.95 \\
\bottomrule
\end{tabular}
\end{table}

\section{Framingham systolic blood pressure}
\label{sec:realdata}
We analyze the Framingham systolic blood pressure data used by \citet{cai2025}. The sample contains two second-examination readings, $B_{i1}$ and $B_{i2}$, for $n=1615$ men. Following their construction, we use the average
\(
Y_i=(B_{i1}+B_{i2})/2
\)
as the contaminated observation and the half-difference
\(
E_i=(B_{i2}-B_{i1})/2
\)
as the error sample. Under the usual symmetric replicate-error model, $E_i$ has the same distribution as the measurement error in the average of the replicate measurements. All estimates therefore use the same preprocessing and are reported on the original SBP scale.

We adopt the working specification $\varepsilon_i\sim N(0,\widehat\sigma_\varepsilon^2)$, where $\widehat\sigma_\varepsilon^2$ is the sample variance of the half-differences and is subsequently treated as fixed. The sample variances of $E_i$ and $Y_i$ are approximately $29.27$ and $395.65$, respectively. Thus, the estimated measurement-error variance accounts for about $7.4\%$ of the observed variance, and $\widehat\sigma_\varepsilon/\widehat{\mathrm{sd}}(Y)\simeq0.272$. The application represents a moderate contamination regime. Criterion~\eqref{eq:bic} selects two Gaussian components.

We compare the direct mixture estimate $\widehat f$, the posterior reconstruction $\widehat f_{\post}$, and the penalized maximum-likelihood estimate of \citet{cai2025}. Figure~\ref{fig:framingham-case} reports both the latent-density comparison and an observed-domain calibration.

\begin{figure}[!t]
\centering
\subfloat[Latent-density estimates.\label{fig:framingham-latent}]{%
    \paperfigure{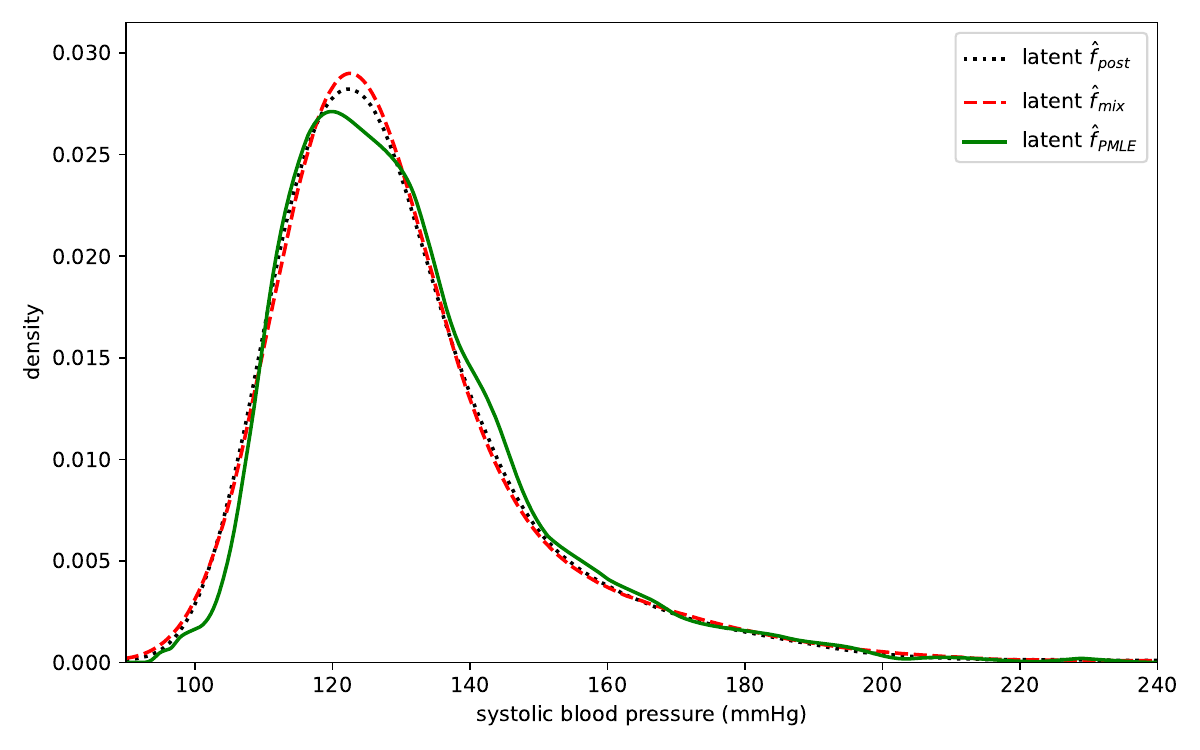}{0.45\linewidth}}
\hfill
\subfloat[Observed data and reconvolved estimates.\label{fig:framingham-reconvolved}]{%
    \paperfigure{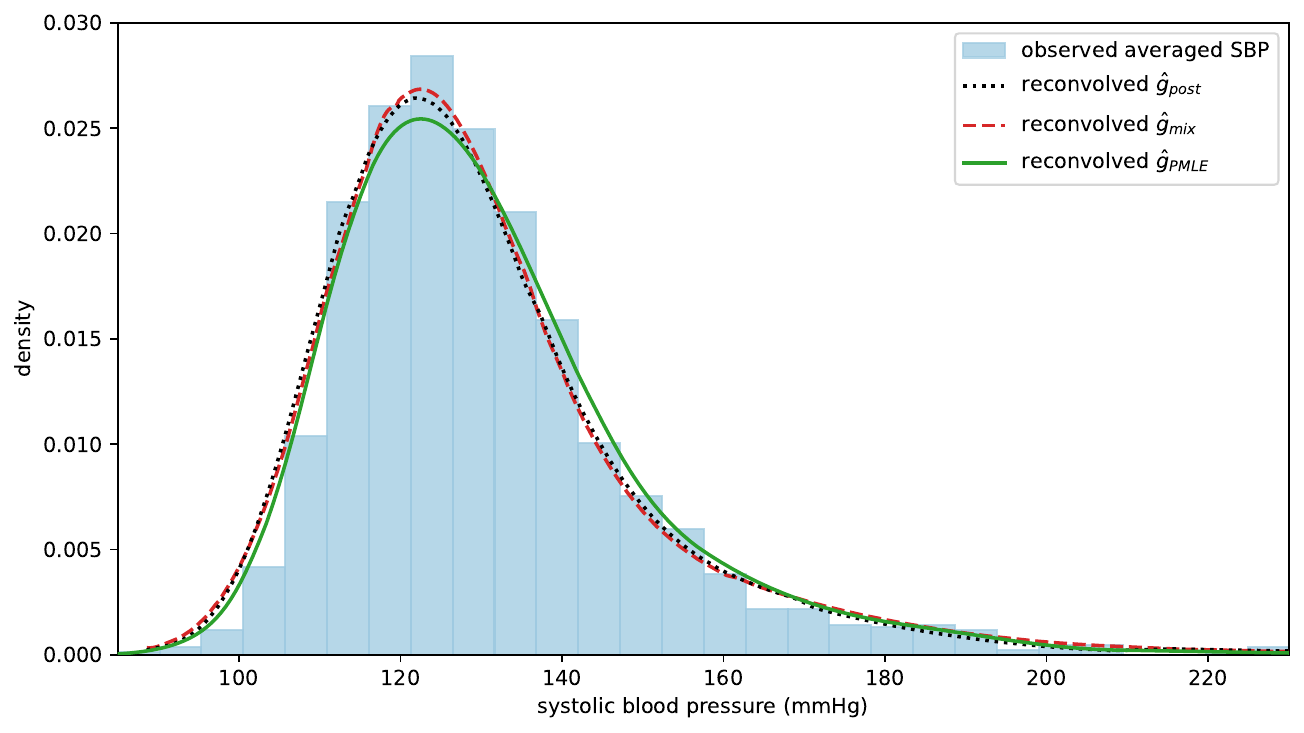}{0.50\linewidth}}
\caption{Framingham data analysis. Panel~\textup{(a)} compares the direct mixture estimate, posterior-density reconstruction, and PMLE estimate on the latent SBP scale. Panel~\textup{(b)} compares the histogram of the averaged readings $Y_i$ with the corresponding reconvolved densities.}
\label{fig:framingham-case}
\end{figure}

In panel~\textup{(a)}, the three estimates agree on the main location and the right-skewed shape of the latent distribution. The direct and posterior estimates are smoother than the PMLE curve, while the posterior update changes the direct mixture fit only slightly. This agreement is consistent with the two-component model selected from the data and indicates that the first-stage representation is already stable in this application.

Panel~\textup{(b)} checks the fitted latent densities on the observed scale by reconvolving each estimate with the working error distribution. The reconvolved direct and posterior curves track the observed histogram closely, particularly around the modal region and the descending right shoulder, while the PMLE reconvolution is somewhat flatter near the mode.

\section{Extensions and open problems}
\label{sec:future-work}
We have introduced a two-stage likelihood framework for density deconvolution that separates regularized estimation from posterior-density reconstruction. The direct estimator provides a stable likelihood fit, whereas the posterior step reuses the fitted conditional latent distributions and may reduce approximation bias when the first-stage sieve is misspecified, at the cost of additional sampling variability under correct finite-dimensional specification. The growing-sieve analysis shows that this reconstruction can preserve the convergence order of the direct estimator under Gaussian and ordinary-smooth measurement error. The simulations and the Framingham application support the same complementary interpretation: posterior reconstruction is most useful when the first-stage fit remains an imperfect representation of the latent shape, while the direct estimator is competitive when a parsimonious Gaussian mixture is already adequate.

The Gaussian analysis suggests several extensions. One is to weaken the fixed positive latent smoothing floor while preserving a conditional-stability bound that still yields vanishing latent error, possibly through scale regularization, roughness penalties, smoothest-near-maximizer selection, or a regularized latent readout. Refining the discretization of Gaussian-smoothed mixing laws may also reduce the required number of components and improve the observed-domain rate.
Other open problems include adaptation to unknown latent smoothness and scale bounds, estimation of the measurement-error distribution from auxiliary or replicated data, and optimization methods with verifiable gaps from the global penalized maximum. Further work could study repeated posterior updates with early stopping, asymmetric mixture sieves, broader ordinary-smooth error families, and uncertainty quantification.

\section*{Supporting Information}
Supporting information. Additional information for this article is available online. The Supporting Information contains deferred proofs, technical refinements, implementation details, and complete simulation results.

\section*{Author address}
Marco Di Marzio, Department of Socio-Economic, Management and Statistical Studies, University of Chieti-Pescara. Email: \email{marco.dimarzio@unich.it}.

\end{document}


\begin{frontmatter}
\title{Supporting Information for ``Regularized-Likelihood Deconvolution with Posterior-Density Reconstruction''}
\runtitle{Supporting Information}
\begin{aug}
\author[A]{\fnms{Marco}~\snm{Di Marzio}\ead[label=e1]{marco.dimarzio@unich.it}\orcid{0000-0003-3500-1890}}
\author[A]{\fnms{Stefania}~\snm{Fensore}\ead[label=e2]{stefania.fensore@unich.it}}
\author[B]{\fnms{Chiara}~\snm{Passamonti}\ead[label=e3]{chiara.passamonti@uni.lu}}
\author[A]{\fnms{Serena}~\snm{Pulcini}\ead[label=e4]{serena.pulcini@phd.unich.it}}

\address[A]{Department of Socio-Economic, Management and Statistical Studies,
University of Chieti-Pescara
\printead[presep={ ,\ }]{e1,e2,e4}}
\address[B]{Department of Mathematics,
University of Luxembourg
\printead[presep={,\ }]{e3}}
\end{aug}
\end{frontmatter}

\noindent
This supporting information contains the proofs of the main theoretical results,
technical arguments supporting the Gaussian and ordinary-smooth analyses,
implementation details, and additional computational and simulation results.
\section{Proofs of the theoretical results of the main paper}
\label{app:main-proofs}
This section contains the proofs of the propositions, lemmas, theorems, and
corollaries stated in the main paper.

\begin{proof}[Proof of Proposition~\ref{main-prop:misspecification}]
For fixed \(x\), write \(\widehat f_{\post}(x)=\Pn q_{\widehat f,x}\). Then
\[
\begin{aligned}
    \widehat f_{\post}(x)-\mathcal R_0(f^\star)(x)
    &=
    (\Pn-\Pzero)q_{\widehat f,x}
    +
    \Pzero\{q_{\widehat f,x}-q_{f^\star,x}\}.
\end{aligned}
\]
The first term is \(o_p(1)\) by posterior empirical stability at \(f^\star\). The second term is \(o_p(1)\) by local continuity of the posterior kernel and \(\widehat f\to f^\star\). Therefore \(\widehat f_{\post}(x)\xrightarrow{p}\mathcal R_0(f^\star)(x)\).
\end{proof}

\begin{proof}[Proof of Theorem~\ref{main-thm:local-posterior-contraction}]
Put \(a=Ks\). Since \(g_t=g_0(1+ta)\), we have
\(
\mathcal R_0(f_t)/{f_0}
=
(1+ts)K^\ast\left[1/(1+ta)\right].
\)
Using the identity
\(
(1+ta)^{-1}
=
1-ta+(t^2a^2)/(1+ta),
\)
set
\(
H_t
=
K^\ast\left[a^2/(1+ta)\right].
\)
Since \(K^\ast[1]=1\) and \(K^\ast a=K^\ast Ks=Ts\), it follows that
\(
K^\ast\left[1/(1+ta)\right]
=
1-tTs+t^2H_t.
\)
Therefore,
\[
\frac{\mathcal R_0(f_t)-f_0}{f_0}
=
t(I-T)s
+
t^2\{H_t-sTs+t\,sH_t\}.
\]
For \(|t|\|s\|_\infty\leq1/2\), the \(L^\infty\)-contraction property of
conditional expectation gives
\(
\|H_t\|_\infty
\leq
2\|s\|_\infty^2.
\)
Hence \(H_t-sTs+t\,sH_t\) remains bounded in \(L^2(f_0)\) as
\(t\to0\), which proves the stated linearization.

Taking the squared \(L^2(f_0)\)-norm gives
\[
\chi^2\{\mathcal R_0(f_t),f_0\}
=
t^2\|(I-T)s\|_{L^2(f_0)}^2
+
O(|t|^3).
\]
Since
\(
\chi^2(f_t,f_0)
=
t^2\|s\|_{L^2(f_0)}^2,
\)
we obtain
\[
\begin{aligned}
\|(I-T)s\|_{L^2(f_0)}^2
-
\|s\|_{L^2(f_0)}^2
&=
-\langle s,(2T-T^2)s\rangle_{L^2(f_0)},\\
\langle s,(2T-T^2)s\rangle_{L^2(f_0)}
&=
2\|Ks\|_{L^2(g_0)}^2
-
\|Ts\|_{L^2(f_0)}^2\\
&\geq
\|Ks\|_{L^2(g_0)}^2.
\end{aligned}
\]
This proves the first part of the theorem and shows that the contraction is
strict whenever \(Ks\neq0\).

Now assume additionally that \(g_0>0\) almost everywhere and that
\(\varphi_\varepsilon(u)\neq0\) for every \(u\in\mathbb R\).
If \(Ks=0\), then
\(
\mathcal A[f_0s]
=
g_0Ks
=
0
\)
almost everywhere. Since \(f_0s\in L^1(\mathbb R)\) by
Cauchy--Schwarz, taking Fourier transforms gives
\(
\widehat{f_0s}(u)\varphi_\varepsilon(u)
=
0.
\)
Because \(\varphi_\varepsilon\) does not vanish,
\(\widehat{f_0s}=0\). By uniqueness of the Fourier transform on
\(L^1(\mathbb R)\), \(f_0s=0\) almost everywhere, and hence
\(s=0\) in \(L^2(f_0)\).
Thus every nonzero \(s\) satisfies
\(
\|Ks\|_{L^2(g_0)}^2>0,
\)
and therefore
\[
\langle s,(2T-T^2)s\rangle_{L^2(f_0)}>0.
\]
The expansion above then implies
\[
\chi^2\{\mathcal R_0(f_t),f_0\}
-
\chi^2(f_t,f_0)
<
0
\]
for all sufficiently small \(t\neq0\), proving the strict-contraction
claim.
\end{proof}

\begin{proof}[Proof of Theorem~\ref{main-thm:mestimator-expansion}]
For the direct estimator, differentiability of
\(\theta\mapsto f_\theta(x)\) gives
\(\sqrt n\{\widehat f(x)-f_0(x)\}
=\dot f_{\theta_0}(x)^\top\sqrt n(\widehat\theta-\theta_0)+o_p(1)\).
The standard asymptotic linear expansion of the M-estimator is
\[
\sqrt n(\widehat\theta-\theta_0)
=
I_0^{-1}n^{-1/2}\sum_{i=1}^nS_{\theta_0}(Y_i)+o_p(1).
\]
Substitution gives the stated expansion with
\(\psi_{\mathrm{dir},x}\).

For the posterior estimator, let
\(\mathbb G_n=\sqrt n(\Pn-\Pzero)\). Since
\(\Pzero q_{\theta_0,x}=f_0(x)\),
\[
\sqrt n\{\widehat f_{\post}(x)-f_0(x)\}
=
\mathbb G_n q_{\theta_0,x}
+
\sqrt n\,\Pzero\{q_{\widehat\theta,x}-q_{\theta_0,x}\}
+
\mathbb G_n\{q_{\widehat\theta,x}-q_{\theta_0,x}\}.
\]
The differentiability and stochastic-equicontinuity conditions give
\(\sqrt n\,\Pzero\{q_{\widehat\theta,x}-q_{\theta_0,x}\}
=
\dot{\mathcal R}_{0,\theta_0}(x)^\top
\sqrt n(\widehat\theta-\theta_0)+o_p(1)\)
and
\(\mathbb G_n\{q_{\widehat\theta,x}-q_{\theta_0,x}\}=o_p(1)\).
Substituting the M-estimator expansion yields
\[
\sqrt n\{\widehat f_{\post}(x)-f_0(x)\}
=
n^{-1/2}\sum_{i=1}^n
\left\{
q_{\theta_0}(x\mid Y_i)-f_0(x)
+
\dot{\mathcal R}_{0,\theta_0}(x)^\top
I_0^{-1}S_{\theta_0}(Y_i)
\right\}
+
o_p(1).
\]
The asserted normal limits follow from the ordinary central limit theorem
for the corresponding scalar influence functions.

It remains to establish the variance decomposition. Differentiating
\(\mathcal R_0(f_\theta)(x)\) at \(\theta_0\) gives
\(\dot{\mathcal R}_{0,\theta_0}(x)
=\Pzero\dot q_{\theta_0}(x\mid Y)
=\dot f_{\theta_0}(x)-A_x\).
Substitution into the posterior influence function gives
\(\psi_{\post,x}(Y)=\psi_{\mathrm{dir},x}(Y)+\xi_x(Y)\).

Finally, since
\(\E_0S_{\theta_0}(Y)=0\) and
\(\E_0\{S_{\theta_0}(Y)S_{\theta_0}(Y)^\top\}=I_0\),
we have
\(\E_0\{\xi_x(Y)S_{\theta_0}(Y)\}
=A_x-I_0I_0^{-1}A_x=0\).
Thus \(\xi_x\) is orthogonal to the score and hence to
\(\psi_{\mathrm{dir},x}\), which gives the stated variance decomposition.
\end{proof}

\begin{proof}[Proof of Proposition~\ref{main-prop:observed-kl-rate}]
Since \(\widehat f\) maximizes the penalized criterion, it holds that
\[
\Pn\ell_{\widehat f}
-\frac{\operatorname{pen}_n(\widehat f)}{n}
\ge
\Pn\ell_{f_n^\circ}
-\frac{\operatorname{pen}_n(f_n^\circ)}{n}.
\]
Because the penalty is nonnegative,
\[
\Pzero(\ell_{f_n^\circ}-\ell_{\widehat f})
\le
(\Pn-\Pzero)(\ell_{\widehat f}-\ell_{f_n^\circ})
+\frac{\operatorname{pen}_n(f_n^\circ)}{n}.
\]
The left-hand side equals
$\KL(g_{f_0},g_{\widehat f})
-\KL(g_{f_0},g_{f_n^\circ}).$
Assumption~\ref{main-ass:primitive-approximation} bounds the second term by \(b_n^2\), while Assumption~\ref{main-ass:primitive-ulln} gives
\[
\sup_{f\in\mathcal F_n}
|(\Pn-\Pzero)(\ell_f-\ell_{f_n^\circ})|
=O_p(c_n).
\]
Therefore,
\[
\KL(g_{f_0},g_{\widehat f})
\le
b_n^2+O_p(c_n)
+\frac{\operatorname{pen}_n(f_n^\circ)}{n},
\]
and hence
\[
\KL(g_{f_0},g_{\widehat f})^{1/2}
=O_p(e_n).
\]
Under Gaussian measurement error, Proposition~\ref{prop:gaussian-global-likelihood}
verifies Assumption~\ref{main-ass:primitive-ulln} without a separate likelihood-tail
condition.
\end{proof}

\begin{proof}[Proof of Proposition~\ref{main-prop:latent-rate-transfer}]
Let \(h=f-f_0\) and let \(\mathcal A f=f*f_\varepsilon\) denote the convolution operator with $f_\varepsilon$, such that \(g_f=\mathcal A f\). By Plancherel's identity and the convolution theorem,
\[
    \int_{|t|\le T}|\varphi_h(t)|^2\,dt
    \le
    \kappa_\varepsilon(T)^2
    \int_{|t|\le T}|\varphi_{\mathcal A h}(t)|^2\,dt
    \le
    C\kappa_\varepsilon(T)^2\|\mathcal A h\|_2^2,
\]
where \(C\) depends only on the Fourier convention. On the complement \(|t|>T\), the Sobolev bound gives
\[
    \int_{|t|>T}|\varphi_h(t)|^2\,dt
    \le
    T^{-2s}\{J_s(f)+J_s(f_0)\}^2
    \le
    T^{-2s}(R_n+R_0)^2.
\]
Combining the low- and high-frequency bounds yields
\[
    \|f-f_0\|_2
    \le
    C\kappa_\varepsilon(T)\|\mathcal A f-\mathcal A f_0\|_2
    +
    C(R_n+R_0)T^{-s}.
\]
Applying this bound to \(f=\widehat f\), with \(T=T_n\), and using the observed \(L^2\) rate assumption \(\|g_{\widehat f}-g_{f_0}\|_2=O_p(d_n)\) gives
\[
    \|\widehat f-f_0\|_2
    =
    O_p
    \left\{
        \kappa_\varepsilon(T_n)d_n
        +
        (R_n+R_0)T_n^{-s}
    \right\}.
\]
On \([-A_n,A_n]\), Cauchy's inequality gives
\[
    \int_{-A_n}^{A_n}|\widehat f(x)-f_0(x)|\,dx
    \le
    (2A_n)^{1/2}\|\widehat f-f_0\|_2.
\]
The remaining tail contribution is bounded by \(2\eta_n\). This proves the \(L^1\) rate.
\end{proof}

\begin{proof}[Proof of Lemma~\ref{main-lem:R-stability}]
For any candidate density \(f\), we obtain, by the triangle inequality and Fubini's theorem
\[
\begin{aligned}
    \|\mathcal R_0(f)-f\|_1
    &\le
    \int f(x)
    \int f_\varepsilon(y-x)
    \left|
    \frac{g_{f_0}(y)}{g_f(y)}-1
    \right|
    dy\,dx                                                       \\
    &=
    \int g_f(y)
    \left|
    \frac{g_{f_0}(y)}{g_f(y)}-1
    \right|
    dy                                                           \\
    &=
    \|g_{f_0}-g_f\|_1.
\end{aligned}
\]
Since a convolution with a density is an \(L^1\)-contraction, Young's convolution inequality gives
\(
    \|g_{f_0}-g_f\|_1
    \le
    \|f_0-f\|_1.
\)
Therefore,
\[
    \|\mathcal R_0(f)-f_0\|_1
    \le
    \|\mathcal R_0(f)-f\|_1+
    \|f-f_0\|_1
    \le
    2\|f-f_0\|_1.
\]
\end{proof}

\begin{proof}[Proof of Theorem~\ref{main-thm:posterior-rate}]
Write \(\widehat f_{\post}=\Pn q_{\widehat f}\). Adding and subtracting \(\mathcal R_0(\widehat f)=\Pzero q_{\widehat f}\),
\[
    \|\widehat f_{\post}-f_0\|_1
    \le
    \|\Pn q_{\widehat f}-\Pzero q_{\widehat f}\|_1
    +
    \|\mathcal R_0(\widehat f)-f_0\|_1.
\]
The first term is \(O_p(\tau_n)\) by Assumption~\ref{main-ass:primitive-posterior-stability}, while the second term is \(O_p(r_n)\) by Lemma~\ref{main-lem:R-stability}. 
\end{proof}

\begin{proof}[Proof of Theorem~\ref{main-thm:explicit-sieve-rates}]
Under the theorem calibration,
Proposition~\ref{prop:supp-gaussian-smoothed-observed-rate} gives
\(\KL(g_0,g_{\widehat f})^{1/2}
=O_p\{n^{-1/5}(\log n)^{7/10}\}\) and
\(d_n:=\|g_{\widehat f}-g_0\|_2
=O_p\{n^{-1/5}(\log n)^{7/10}\}\).
Every \(f\in\mathcal F_n^G\), as well as \(f_0\), satisfies the analytic
Fourier envelope required by
Lemma~\ref{lem:supp-gaussian-analytic-stability}. Hence
\[
\|\widehat f-f_0\|_2
=
O_p\!
\left[d_n^\vartheta
\{\log(d_n^{-1})\}^{-(1-\vartheta)/4}\right]
=
O_p\!\left\{
n^{-\vartheta/5}(\log n)^{19\vartheta/20-1/4}
\right\}.
\]
Lemma~\ref{lem:supp-gaussian-smoothed-l1-transfer}, with any fixed
\(D>\vartheta/5\), then gives
\(\|\widehat f-f_0\|_1
=O_p\{n^{-\vartheta/5}(\log n)^{19\vartheta/20}\}\).

For posterior reconstruction,
Proposition~\ref{prop:supp-gaussian-smoothed-posterior-complexity} gives
\(\|(\Pn-\Pzero)q_{\widehat f}\|_1
=O_p\{n^{-2/5}(\log n)^{2/5}\}\), while
Lemma~\ref{main-lem:R-stability} gives
\(\|\mathcal R_0(\widehat f)-f_0\|_1
\leq2\|\widehat f-f_0\|_1\). Since \(0<\vartheta<1\), the empirical
posterior term has smaller polynomial order than the direct latent
\(L^1\) error, which proves the posterior rate.
\end{proof}

\begin{proof}[Proof of Theorem~\ref{main-thm:gaussian-sharp-end-to-end}]
Take \(M=M_n\), \(K=K_n\), and \(\pi_{\min}=\pi_{\min,n}\) in
Proposition~\ref{prop:supp-gaussian-smoothed-klv-approximation}, and denote the resulting approximating
density by \(f_n^\diamond\), with \(g_n^\diamond=g_{f_n^\diamond}\). By
\eqref{eq:gaussian-klv-approximation},
\[
\KL(g_0,g_n^\diamond)+\V(g_0,g_n^\diamond)
\le C\left\{
\frac{M_n^2}{K_n^2}
+\frac{M_n^4}{K_n^4}
+e^{-cM_n^2}
+K_n\pi_{\min,n}(1+M_n^4)
\right\}.
\]
Under~\eqref{main-eq:gaussian-sharp-calibration}, the first term is of order \((r_n^\sharp)^2\), the
second is \(o\{(r_n^\sharp)^2\}\), and \(C_M\) can be chosen so that
\(e^{-cM_n^2}=o\{(r_n^\sharp)^2\}\). Moreover,
\[
\frac{K_n\pi_{\min,n}(1+M_n^4)}{(r_n^\sharp)^2}
=O\left\{n^{1-u}\frac{\log n}{L_3(n)}\right\}=o(1)
\]
because \(u>1\). Thus
\(\KL(g_0,g_n^\diamond)+\V(g_0,g_n^\diamond)\le C_a(r_n^\sharp)^2\) for a fixed \(C_a\).

By Lemma~\ref{lem:gaussian-fixed-scale-hellinger}, for every fixed \(A\ge1\),
\[
\int_{c_0(Ar_n^\sharp)^2}^{c_1Ar_n^\sharp}
\{1+H_{[]}(v,\mathcal G_n^G,\dH)\}^{1/2}dv
\le
CAr_n^\sharp\{K_n\log n\}^{1/2}.
\]
The ratio of the right-hand side to \(\sqrt n\,(Ar_n^\sharp)^2\) is
\(O[\{A\sqrt{L_3(n)}\}^{-1}]\), and hence the entropy condition in Lemma~\ref{lem:penalized-likelihood-localization}
holds for all sufficiently large \(n\).

Because \(\widehat f\) maximizes the penalized criterion and the penalty is nonnegative,
\(\Pn\log g_{\widehat f}\ge\Pn\log g_n^\diamond-
\operatorname{pen}_n(f_n^\diamond)/n\). Along the approximating sequence,
\[
\frac{\operatorname{pen}_n(f_n^\diamond)}{n}
\le C\frac{K_nL_3(n)\log n}{n}
=O\{(r_n^\sharp)^2\}.
\]
Lemma~\ref{lem:penalized-likelihood-localization}, with \(\epsilon_n=r_n^\sharp\), now gives
\(\dH(g_{\widehat f},g_0)=O_p(r_n^\sharp)\).

Finally, all densities in \(\mathcal G_n^G\), as well as \(g_0\), are uniformly bounded above because
the observed Gaussian variances have a fixed positive lower bound. Therefore
\(\|p-q\|_2^2\le C\dH^2(p,q)\) on the relevant class, which proves the \(L^2\) assertion in~\eqref{main-eq:gaussian-sharp-observed-conclusion}.

Set
\[
d_n:=\|g_{\widehat f}-g_0\|_2=O_p(r_n^\sharp).
\]
Applying Lemma~\ref{lem:supp-gaussian-analytic-stability} and using
\(\log(1/r_n^\sharp)\asymp\log n\) yields
\[
\|\widehat f-f_0\|_2
=
O_p\left[
(r_n^\sharp)^\vartheta
\{\log(1/r_n^\sharp)\}^{-(1-\vartheta)/4}
\right]
=
O_p\left[
n^{-\vartheta/3}
(\log n)^{3\vartheta/4-1/4}
\{L_3(n)\}^{\vartheta/3}
\right].
\]
Lemma~\ref{lem:supp-gaussian-smoothed-l1-transfer}, with any fixed \(D>\vartheta/3\), adds the factor
\((\log n)^{1/4}\) and proves the direct \(L^1\) rate.

For the posterior estimator,
\[
\|\widehat f_{\post}-f_0\|_1
\le
\|(\Pn-\Pzero)q_{\widehat f}\|_1
+
\|\mathcal R_0(\widehat f)-f_0\|_1.
\]
Proposition~\ref{prop:supp-gaussian-smoothed-posterior-complexity} gives
\[
\sup_{f\in\mathcal F_n^G}
\|(\Pn-\Pzero)q_f\|_1
=
O_p\left[
n^{-1/3}(\log n)^{1/2}\{L_3(n)\}^{-1/6}
\right],
\]
while Lemma~\ref{main-lem:R-stability} bounds the second term by
\(2\|\widehat f-f_0\|_1\). Since \(0<\vartheta<1\), the posterior empirical term is of smaller
polynomial order than the direct latent \(L^1\) rate. This proves the result.
\end{proof}

\begin{proof}[Proof of Theorem~\ref{main-thm:local-ordinary-rates}]
Under \(K_h\asymp h^{-1}L_h^{a_K}\),
Proposition~\ref{prop:supp-local-observed-rate} gives, up to logarithmic
factors,
\[
\|g_{\widehat f_h}-g_0\|_2
=O_p\{h^{s+\beta}+n^{-1/2}h^{-1/2}\}.
\]
Apply Lemma~\ref{lem:supp-local-direct-inverse} with
\(T_h=h^{-1}\{A L_h\}^{1/2}\) and \(A>2s+1\). The approximation part of
the observed error becomes \(h^s\) up to logarithms, the stochastic part
becomes \(n^{-1/2}h^{-(\beta+1/2)}\) up to logarithms, the true-density
high-frequency term is \(O(h^sL_h^{-s/2})\), and the Gaussian-sieve tail is
\(o(h^s)\). Hence
\[
\|\widehat f_h-f_0\|_2
=O_p\{h^sL_h^C+n^{-1/2}h^{-(\beta+1/2)}L_h^C\}.
\]
Balancing the two polynomial terms gives
\(h_n=n^{-1/(2s+2\beta+1)}\) and proves the stated \(L^2\) rate.

Under the additional tail conditions,
Lemma~\ref{lem:supp-local-l2-l1} converts this bound into
\(\|\widehat f_{h_n}-f_0\|_1
=O_p\{n^{-s/(2s+2\beta+1)}(\log n)^C\}\).
For posterior reconstruction, Lemma~\ref{main-lem:R-stability} gives
\[
\|\widehat f_{\post,h_n}-f_0\|_1
\le
\|(\Pn-\Pzero)q_{\widehat f_{h_n}}\|_1
+2\|\widehat f_{h_n}-f_0\|_1.
\]
By Assumption~\ref{main-ass:local-pe}, the empirical term is
\(O_p\{n^{-1/2}h_n^{-\gamma}L_{h_n}^{a_Q}\}\). Since
\(n^{-1/2}=h_n^{s+\beta+1/2}\) up to logarithmic factors, this term is of
order \(h_n^{s+\beta+1/2-\gamma}\) up to logarithms. It is no larger than
\(h_n^s\) when \(\gamma\le\beta+1/2\), and has strictly smaller polynomial
order when \(\gamma<\beta+1/2\). This proves the two \(L^1\) conclusions.
\end{proof}

\begin{proof}[Proof of Corollary~\ref{main-cor:concrete-laplace-rates}]
The characteristic function of \(f_0=\phiN_\tau*Q_0\) satisfies
\(|\varphi_{f_0}(t)|\le e^{-\tau_-^2t^2/2}\), so
\(f_0\in H^s(\mathbb R)\) for every fixed \(s>0\). The random variable
\(X=U+\tau Z\) has a subgaussian tail uniformly over the class.
Proposition~\ref{prop:supp-concrete-laplace-ap} verifies
Assumption~\ref{main-ass:local-ap} with \(\beta=2\),
Proposition~\ref{prop:supp-concrete-laplace-ll} verifies
Assumption~\ref{main-ass:local-ll}, and
Proposition~\ref{prop:supp-concrete-laplace-pe} verifies
Assumption~\ref{main-ass:local-pe} with \(\gamma=3/2\).
Theorem~\ref{main-thm:local-ordinary-rates} therefore gives the direct
\(L^2\) and \(L^1\) rates and the posterior \(L^1\) rate. Since
\(3/2<2+1/2\), the empirical posterior-averaging term has strictly smaller
polynomial order; at \(h_n=n^{-1/(2s+5)}\), it is \(h_n^{s+1}\) up to
logarithmic factors.
\end{proof}

\section{Deferred fixed-dimensional results}
\label{app:fixed-dimensional-results}

This section collects fixed-dimensional results used in the
fixed-dimensional analysis of the main paper but deferred here to keep
the main argument focused.

\setcounter{proposition}{0}
\setcounter{lemma}{0}

\subsection{Fixed-model consistency}
\label{app:fixed-model-consistency}

The following result records the consistency statement used under correct
fixed-dimensional specification.

\begin{proposition}[Transfer of fixed-model consistency]
\label{prop:supp-fixed-consistency}
For a fixed \(K\), suppose that \(\theta_0\in\Theta_K\), where
\(f_0=f_{\theta_0}\). Assume that the parametrization is regular at
\(\theta_0\): the model is identifiable up to label switching, and the true
component weights and scales are bounded away from the usual finite-mixture
degeneracies. After fixing one admissible label ordering, assume that:
\begin{itemize}
    \item[(i)] the convolution-likelihood estimator satisfies
    \(\widehat\theta\xrightarrow{p}\theta_0.
    \)
    \end{itemize}
Moreover, for every continuity point \(x\) of \(f_0\), assume that:
\begin{itemize}
    \item[(ii)]  the map
    \(
        \theta\mapsto f(x;\theta)
    \)
    is continuous at \(\theta_0\);

    \item[(iii)] the posterior kernels
    \(q_{f(\cdot;\theta)}(x\mid Y)\)
    are locally continuous at \(\theta_0\) in \(L^1(\Pzero)\) and satisfy the local
    stochastic-equicontinuity condition
    \(
       \sup_{\theta\in U}
       \left|
       (\Pn-\Pzero)
       \left\{
       q_{f(\cdot;\theta)}(x\mid\cdot)
       -
       q_{f_0}(x\mid\cdot)
       \right\}
       \right|
       =
       o_p(1)
    \)
    for some neighborhood \(U\) of \(\theta_0\).
\end{itemize}
Then, for every continuity point \(x\) of \(f_0\),
\[
    \widehat f(x)\xrightarrow{p}f_0(x),
    \qquad
    \widehat f_{\post}(x)\xrightarrow{p}f_0(x).
\]
Under the corresponding integrated envelope and domination conditions, the
same conclusions hold in \(L^1\) and \(L^2\).
\end{proposition}

\begin{proof}
The first-stage consistency follows immediately from conditions \emph{(i)}
and \emph{(ii)}. For the posterior-density estimator, since
\(\Pzero q_{f_0,x}=f_0(x)\), after adding and subtracting
\(\Pn q_{f_0,x}\), we write
\[
\widehat f_{\post}(x)-f_0(x)
=
(\Pn-\Pzero)q_{f_0,x}
+
\Pn\{q_{\widehat\theta,x}-q_{f_0,x}\}.
\]
The first term is \(o_p(1)\) by the law of large numbers. For the second term,
\[
\begin{aligned}
|\Pn\{q_{\widehat\theta,x}-q_{f_0,x}\}|
&\le
|(\Pn-\Pzero)\{q_{\widehat\theta,x}-q_{f_0,x}\}| \\
&\quad+
\Pzero|q_{\widehat\theta,x}-q_{f_0,x}|.
\end{aligned}
\]
The first term on the right is \(o_p(1)\) by the local stochastic
equicontinuity assumption in \emph{(iii)}, since \(\widehat\theta\in U\)
with probability tending to one. The second is \(o_p(1)\) by the local
\(L^1(\Pzero)\)-continuity in \emph{(iii)} together with condition \emph{(i)}.
Therefore,
\(
\widehat f_{\post}(x)\xrightarrow{p}f_0(x).
\)
The integrated \(L^1\) and \(L^2\) conclusions follow by the same
decomposition under the corresponding envelope and domination assumptions.
\end{proof}

\subsection{Exact population posterior correction}
\label{app:exact-posterior-correction}

Let \(\mathcal A\) denote convolution with the measurement-error density,
\[
\mathcal A[h](y)
=
\int_{\mathbb R} f_\varepsilon(y-x)h(x)\,dx,
\]
and let
\[
\mathcal A^\ast[v](x)
=
\int_{\mathbb R} f_\varepsilon(y-x)v(y)\,dy
\]
denote its \(L^2\)-adjoint. The next result gives the exact correction
underlying the local misspecification analysis.

\begin{lemma}[Exact posterior-error decomposition]
\label{lem:supp-posterior-exact-error}
Let \(f\) and \(f_0\) be densities and write
\[
e=f-f_0,
\qquad
g_f=\mathcal A[f],
\qquad
g_0=\mathcal A[f_0].
\]
Suppose that \(g_f>0\) \(\Pzero\)-almost everywhere,
\(e\in L^2(\mathbb R)\), and
\[
C_f[e]
:=
f\,\mathcal A^\ast
\left[
\frac{\mathcal A[e]}{g_f}
\right]
\in L^2(\mathbb R).
\]
Then
\[
\mathcal R_0(f)-f_0
=
e-C_f[e].
\]
Consequently,
\[
\|\mathcal R_0(f)-f_0\|_2^2
=
\|f-f_0\|_2^2
-
2\langle e,C_f[e]\rangle
+
\|C_f[e]\|_2^2.
\]
\end{lemma}

\begin{proof}
Since \(g_0-g_f=-\mathcal A[e]\) and
\(\mathcal A^\ast[1](x)=1\),
\[
\begin{aligned}
\mathcal R_0(f)
&=
f\,\mathcal A^\ast
\left[
\frac{g_0}{g_f}
\right] \\
&=
f\,\mathcal A^\ast
\left[
1+\frac{g_0-g_f}{g_f}
\right] \\
&=
f-
f\,\mathcal A^\ast
\left[
\frac{\mathcal A[e]}{g_f}
\right] \\
&=
f-C_f[e].
\end{aligned}
\]
Subtracting \(f_0\) gives the first identity, and squaring in
\(L^2(\mathbb R)\) gives the second.
\end{proof}

It follows immediately that
\[
\|\mathcal R_0(f)-f_0\|_2
<
\|f-f_0\|_2
\]
if and only if
\[
2\langle e,C_f[e]\rangle
>
\|C_f[e]\|_2^2.
\]
Thus, whether posterior reconstruction reduces the latent \(L^2\) error
depends on the direction of the misspecification.

\subsection{Conditional-information operator}
\label{app:conditional-information-operator}

For \(s\in L^2(f_0)\) and \(v\in L^2(g_0)\), define
\(
(Ks)(y)
=
\mathbb E_0\{s(X)\mid Y=y\}
\) and \(
(K^\ast v)(x)
=
\mathbb E_0\{v(Y)\mid X=x\},
\)
and let
\(
T=K^\ast K.
\)

\begin{lemma}[Conditional-information operator]
\label{lem:supp-conditional-information-operator}
The operator
\(K:L^2(f_0)\to L^2(g_0)\)
is contractive. Moreover,
\(T=K^\ast K\)
is positive and self-adjoint on \(L^2(f_0)\), with
\[
\langle s,Ts\rangle_{L^2(f_0)}
=
\|Ks\|_{L^2(g_0)}^2,
\qquad
\|Ts\|_{L^2(f_0)}
\leq
\|Ks\|_{L^2(g_0)}
\leq
\|s\|_{L^2(f_0)}.
\]
\end{lemma}

\begin{proof}
By conditional Jensen's inequality,
\[
\|Ks\|_{L^2(g_0)}^2
=
\mathbb E_0\!\left[
\{\mathbb E_0(s(X)\mid Y)\}^2
\right]
\leq
\mathbb E_0\{s(X)^2\}.
\]
For \(s\in L^2(f_0)\) and \(v\in L^2(g_0)\),
\[
\begin{aligned}
\langle Ks,v\rangle_{L^2(g_0)}
&=
\mathbb E_0\!\left[
\mathbb E_0\{s(X)\mid Y\}v(Y)
\right] \\
&=
\mathbb E_0\{s(X)v(Y)\} \\
&=
\mathbb E_0\!\left[
s(X)\mathbb E_0\{v(Y)\mid X\}
\right] \\
&=
\langle s,K^\ast v\rangle_{L^2(f_0)}.
\end{aligned}
\]
Thus \(K^\ast\) is the adjoint of \(K\). The standard properties of
\(K^\ast K\) give positivity and self-adjointness, while
\[
\|T\|
\leq
\|K^\ast\|\,\|K\|
\leq 1.
\]
The stated identities and inequalities follow.
\end{proof}

\section{Technical sieve bounds and calibration}

This section collects auxiliary approximation, entropy, localization, inverse-stability, posterior-complexity, tail, and model-selection results supporting the growing-sieve analysis. The entropy and bracketing arguments follow standard sieve empirical-process calculations (see \citep{vandervaart1998,chen2007}). These results provide sufficient conditions for the consistency and rate arguments developed in the main paper.

\subsection{Full log-likelihood control under Gaussian measurement error}
\label{app:ass-primitive-ulln-material}
The next result verifies the global likelihood condition on the whole observation line. The proof uses the positive variance contributed by the Gaussian measurement error and therefore does not require a
sieve-uniform density floor on a truncated set.

\begin{proposition}[Global log-likelihood complexity under Gaussian measurement error]
\label{prop:gaussian-global-likelihood}
Suppose that \(\varepsilon\sim N(0,\sigma_\varepsilon^2)\), with \(\sigma_\varepsilon>0\), and that
\(E_0Y^4<\infty\). Let \(\mathcal F_n\) be the regularized Gaussian-mixture sieve defined by
\eqref{main-eq:regularized-parameter-space} and~\eqref{main-Fn}, with
\(\sigma_{\max,n}\le\overline\sigma<\infty\).
Then Assumption~\ref{main-ass:primitive-ulln} holds with
\[
H_n(y)\le C\{1+(|y|+M_n)^2\},
\qquad
F_n\le C(1+M_n^2),
\]
and one may take
\[
V_n=C K_n,
\qquad
L_n=C K_n(1+M_n+\overline\sigma)\{1+\log(\pi_{\min,n}^{-1})\}.
\]
Consequently, with \(\Gamma_n=V_n\log L_n\),
\[
\sup_{f\in\mathcal F_n}
\left|
(\Pn-\Pzero)(\ell_f-\ell_{f_n^\circ})
\right|
=O_p(c_n^{\mathrm G}),
\]
where
\[
c_n^{\mathrm G}
=
F_n\left\{
\sqrt{\frac{\Gamma_n}{n}}
+
\frac{\Gamma_n}{n}
\right\}.
\]
In particular, if \(K_n\asymp n^p\), \(\pi_{\min,n}^{-1}\) grows at most polynomially, \(0<p<1\), and
\(M_n\asymp(\log n)^m\), then
\[
c_n^{\mathrm G}
=O\left\{
n^{(p-1)/2}(\log n)^{2m+1/2}
+
n^{p-1}(\log n)^{2m+1}
\right\}
=o(1).
\]
No restriction on the rate at which \(\sigma_{\min,n}\) decreases is needed for this likelihood bound.
\end{proposition}
\begin{proof}
For a \(K\)-component latent Gaussian mixture, the observed density is
\[
g_\theta(y)=\sum_{k=1}^K\pi_k\phiN(y;\mu_k,v_k),
\qquad
v_k=\sigma_k^2+\sigma_\varepsilon^2.
\]
Uniformly over the sieve,
\[
\sigma_\varepsilon^2\le v_k\le\sigma_\varepsilon^2+\overline\sigma^2=: \overline v.
\]
Hence
\[
(2\pi\overline v)^{-1/2}
\exp\left\{-\frac{(|y|+M_n)^2}{2\sigma_\varepsilon^2}\right\} \le g_\theta(y)\le(2\pi\sigma_\varepsilon^2)^{-1/2}.
\]
It follows that
\[
|\ell_\theta(y)|\le C\{1+(|y|+M_n)^2\}.
\]
The difference class \(\mathcal D_n\) therefore has an envelope of the same order, and \(E_0Y^4<\infty\)
gives \(\|H_n\|_{\Pzero,2}\le C(1+M_n^2)=F_n\).
Parameterize the weights by logits \(a_1,\ldots,a_{K-1}\), with \(a_j=\log(\pi_j/\pi_K)\), and let
\[
r_k(y)=\frac{\pi_k\phiN(y;\mu_k,v_k)}{g_\theta(y)}.
\]
Direct differentiation gives
\[
\partial_{a_j}\ell_\theta(y)=r_j(y)-\pi_j,
\]
\[
\partial_{\mu_j}\ell_\theta(y)=r_j(y)\frac{y-\mu_j}{v_j},
\]
and
\[
\partial_{\sigma_j}\ell_\theta(y)
=r_j(y)\sigma_j\left\{\frac{(y-\mu_j)^2}{v_j^2}-\frac1{v_j}\right\}.
\]
Because \(0\le r_j\le1\), \(|\mu_j|\le M_n\), \(v_j\ge\sigma_\varepsilon^2\), and
\(\sigma_j\le\overline\sigma\), every score coordinate is bounded by
\[
B_n(y)=C\{1+(|y|+M_n)^2\}.
\]
For two parameters in the same \(K\)-component model, the mean-value theorem yields
\[
|\ell_\theta(y)-\ell_{\widetilde\theta}(y)|
\le C K B_n(y)\|\eta-\widetilde\eta\|_\infty,
\qquad \eta=(a,\mu,\sigma).
\]
Use a coordinate grid of mesh \(\delta=\epsilon/(CK)\). Brackets centered at the grid log-likelihoods
with half-width \(CK\delta B_n\) have \(L^2(\Pzero)\)-width at most \(\epsilon F_n\). Since the logits
range over an interval of length \(O\{1+\log(\pi_{\min,n}^{-1})\}\), the means over an interval of
length \(2M_n\), and the scales over a bounded interval, the union over \(K\le K_n\) satisfies
\[
\log N_{[]}(
\epsilon F_n,\mathcal D_n,L^2(\Pzero))
\le
C K_n\log\left\{
\frac{C K_n(1+M_n+\overline\sigma)\{1+\log(\pi_{\min,n}^{-1})\}}
{\epsilon}
\right\}.
\]
This is the entropy condition with \(V_n, L_n\). The finite-dimensional parameter classes
are separable; otherwise, the following bound is read in outer probability. The standard
\(L^2(\Pzero)\)-bracketing entropy integral gives
\[
\sup_{d\in\mathcal D_n}|(\Pn-\Pzero)d|
=O_p\left[
F_n\sqrt{\frac{V_n\log L_n}{n}}
\right].
\]
By the definition of \(c_n^{\mathrm G}\), this term is \(O_p(c_n^{\mathrm G})\). its additional linear term is retained as a
conservative remainder and is asymptotically smaller under the stated calibrations. Substituting the polynomial and logarithmic calibrations establishes the claimed empirical-process bound. Notice that the argument never divides by a
global density floor and never uses \(\sigma_{\min,n}^{-1}\).
\end{proof}

\subsection{Observed norm conversion from KL control}
\label{app:observed-norm-conversion}

The following lemma converts the observed-domain Kullback--Leibler rate
into the norm bounds used by the inverse-stability arguments in the main
paper.

\begin{lemma}[From observed KL to observed sup-norm and \(L^2\) control]
\label{lem:supp-kl-to-l2}
Suppose that
\(\KL(g_{f_0},g_{\widehat f})^{1/2}=O_p(e_n)\)
and that \(f_\varepsilon\) is globally Lipschitz, with
\(|f_\varepsilon(u)-f_\varepsilon(v)|
\le L_\varepsilon|u-v|\) for all \(u,v\in\mathbb R\), where
\(L_\varepsilon<\infty\). Then
\[
\|g_{\widehat f}-g_{f_0}\|_1=O_p(e_n),
\qquad
\|g_{\widehat f}-g_{f_0}\|_2=O_p(e_n),
\qquad
\|g_{\widehat f}-g_{f_0}\|_\infty=O_p(e_n^{1/2}).
\]
The same conclusions hold for the EM implementation with \(e_n\)
replaced by \(e_n^{\mathrm{EM}}\).
\end{lemma}

\begin{proof}
Pinsker's inequality gives
\(\|g_{\widehat f}-g_{f_0}\|_1
\le \{2\KL(g_{f_0},g_{\widehat f})\}^{1/2}=O_p(e_n)\).

Because a globally Lipschitz density is bounded,
\(g_f(y)\le\|f_\varepsilon\|_\infty\) for every latent density \(f\).
Hence the Hellinger--Kullback--Leibler comparison gives
\(\|g_{\widehat f}-g_{f_0}\|_2^2
\le 4\|f_\varepsilon\|_\infty
H^2(g_{\widehat f},g_{f_0})\),
and therefore
\(\|g_{\widehat f}-g_{f_0}\|_2=O_p(e_n)\).

Moreover, every convolved density \(g_f=f*f_\varepsilon\) is
\(L_\varepsilon\)-Lipschitz, since
\(|g_f(y)-g_f(y')|\le L_\varepsilon|y-y'|\).
Thus \(g_{\widehat f}-g_{f_0}\) is
\(2L_\varepsilon\)-Lipschitz. The one-dimensional
Gagliardo--Nirenberg inequality yields
\[
\|g_{\widehat f}-g_{f_0}\|_\infty
\le
\{4L_\varepsilon
\|g_{\widehat f}-g_{f_0}\|_1\}^{1/2}
=
O_p(e_n^{1/2}).
\]
The same argument applies to the EM implementation after replacing
\(e_n\) by \(e_n^{\mathrm{EM}}\).
\end{proof}

\subsection{Global observed-domain analysis for shrinking-scale sieves}
\label{app:global-shrinking-scale}

The results below complement the main growing-sieve theory with a broad
observed-domain analysis for Gaussian-mixture sieves whose component scales
may decrease with the approximation resolution.

\begin{proposition}[Global observed-scale approximation]
\label{prop:supp-obs-scale-appr}
Assume that \(f_0\) belongs to a class of densities for which there exist
constants \(s>0\), \(\kappa_2\ge0\), \(c_0>0\), and \(C<\infty\) such
that, for every sufficiently small \(h>0\), there is a finite Gaussian
mixture \(f_h^\circ=\sum_{k=1}^{K_h}
\pi_k\phiN(\cdot;\mu_k,\sigma_k^2)\) satisfying
\(K_h\le Ch^{-1}\log(h^{-1})\),
\(\sigma_k\in[c_0h,C]\),
\(|\mu_k|\le C\{\log(h^{-1})\}^{1/2}\), and
\(\KL(f_0,f_h^\circ)\le
Ch^{2s}\{\log(h^{-1})\}^{\kappa_2}\).

Let \(h_n\to0\) and suppose that, for all sufficiently large \(n\), the
Gaussian-mixture sieve satisfies
\(K_n\ge Ch_n^{-1}\log(h_n^{-1})\),
\(\sigma_{\min,n}\le c_0h_n\),
\(M_n\ge C\{\log(h_n^{-1})\}^{1/2}\), and an upper scale bound compatible
with the approximation class, so that \(\sigma_{\max,n}\ge C\)
eventually. Then Assumption~\ref{main-ass:primitive-approximation} holds
with
\(b_n\asymp h_n^s\{\log(h_n^{-1})\}^{\kappa_2/2}\).
\end{proposition}

\begin{proof}
For every sufficiently small \(h>0\), the assumed approximation property
provides a finite Gaussian mixture \(f_h^\circ\) satisfying the stated
component, scale, location, and KL bounds. Such bounds are standard under
the corresponding smoothness, tail, and local-regularity conditions; see
\citet{kruijer2010} and \citet{ghosalvdv2001}.

Take \(h=h_n\). The sieve-size and parameter-range conditions imply
\(f_{h_n}^\circ\in\mathcal F_n\) for all sufficiently large \(n\).
Convolution with \(f_\varepsilon\) is a Markov kernel, so the
data-processing inequality gives
\[
\KL(g_{f_0},g_{f_{h_n}^\circ})
\le
\KL(f_0,f_{h_n}^\circ)
\le
Ch_n^{2s}\{\log(h_n^{-1})\}^{\kappa_2}.
\]
Taking \(f_n^\circ=f_{h_n}^\circ\) and
\(b_n\asymp h_n^s\{\log(h_n^{-1})\}^{\kappa_2/2}\)
proves Assumption~\ref{main-ass:primitive-approximation}.
\end{proof}
\begin{proposition}[Broad global observed-domain rate under Gaussian likelihood control]
\label{prop:supp-self-contained-en}
Suppose that the conditions of
Proposition~\ref{prop:supp-obs-scale-appr} hold and that the measurement
error is Gaussian. Assume additionally that \(\E_0Y^4<\infty\),
\(K_n\asymp h_n^{-1}\{\log(h_n^{-1})\}^{a_K}\) for some \(a_K\ge1\),
\(M_n\asymp\{\log(h_n^{-1})\}^{1/2}\), and retain the sieve-membership
restrictions of Proposition~\ref{prop:supp-obs-scale-appr}, including
\(\sigma_{\min,n}\le c_0h_n\). Suppose also that
\(K_n\pi_{\min,n}\le1\), that \(\pi_{\min,n}^{-1}\) grows at most
polynomially, and that Assumption~\ref{main-ass:primitive-penalty} holds.

Then
\(e_n=O[h_n^s\{\log(h_n^{-1})\}^{\kappa_2/2}
+n^{-1/4}h_n^{-1/4}(\log n)^C]\).
With \(h_n=n^{-1/(4s+1)}\), up to a fixed logarithmic calibration,
\[
\KL(g_{f_0},g_{\widehat f})^{1/2}
=
O_p\{n^{-s/(4s+1)}(\log n)^C\},
\qquad
\|g_{\widehat f}-g_{f_0}\|_2
=
O_p\{n^{-s/(4s+1)}(\log n)^C\}.
\]
\end{proposition}

\begin{proof}
The fourth-moment condition allows
Proposition~\ref{prop:gaussian-global-likelihood} to be applied, while
Proposition~\ref{prop:supp-obs-scale-appr} gives
\(b_n\asymp h_n^s\{\log(h_n^{-1})\}^{\kappa_2/2}\).
Under the stated calibration,
\(F_n=O\{\log(h_n^{-1})\}\) and
\(V_n\log L_n=O\{h_n^{-1}(\log n)^C\}\), so
Proposition~\ref{prop:gaussian-global-likelihood} yields
\(c_n^{1/2}=O\{n^{-1/4}h_n^{-1/4}(\log n)^C\}\).

Along the approximating sequence, the square root of the normalized
iterated-log BIC penalty is no larger than this empirical-process term.
Substitution into \(e_n\) therefore gives the stated bound. Balancing
\(h_n^s\) with \(n^{-1/4}h_n^{-1/4}\) yields
\(h_n=n^{-1/(4s+1)}\), up to logarithmic factors.
Proposition~\ref{main-prop:observed-kl-rate} then gives the KL rate, and
Lemma~\ref{lem:supp-kl-to-l2} gives the corresponding observed \(L^2\)
rate.
\end{proof}

The likelihood bound places no lower-rate restriction on
\(\sigma_{\min,n}\): Gaussian measurement error gives every induced
observed component a variance bounded below by \(\sigma_\varepsilon^2\),
irrespective of the latent component scale.
\subsection{Localization tools for the sharper Gaussian rate}
\label{app:gaussian-localization-tools}

\begin{lemma}[Second log-moment contraction under marginalization]
\label{lem:marginal-log-moment}
Let \(P\) and \(Q\) be mutually absolutely continuous probability measures on
\(\mathcal Z\times\mathcal Y\), and let \(P_Y,Q_Y\) be their \(\mathcal Y\)-marginals. Write
\(L=\log(dP/dQ)\) and \(\ell=\log(dP_Y/dQ_Y)\). If \(E_PL^2+E_QL^2<\infty\), then
\[
E_{P_Y}\ell^2\le E_PL^2+E_QL^2.
\]
\end{lemma}
\begin{proof}
Conditional expectation under \(P\) gives \(e^{-\ell(Y)}=E_P\{e^{-L}\mid Y\}\). Jensen's inequality
implies \(\ell(Y)\le E_P(L\mid Y)\), and hence
\(\ell_+(Y)\le E_P(L_+\mid Y)\). Conditional Jensen therefore gives
\(E_{P_Y}\ell_+^2\le E_PL_+^2\). Interchanging \(P\) and \(Q\) yields
\(E_{Q_Y}\ell_-^2\le E_QL_-^2\). On \(\{\ell<0\}\), one has
\(dP_Y/dQ_Y=e^\ell\le1\), so \(E_{P_Y}\ell_-^2\le E_{Q_Y}\ell_-^2\). Adding the positive and
negative parts proves the statement.
\end{proof}

\begin{lemma}[Localized comparison for a penalized sieve likelihood]
\label{lem:penalized-likelihood-localization}
Let \(\mathcal G_n\) be a pointwise measurable class of densities with true density \(g_0\). Suppose
that \(\widehat g_n,g_n^\circ\in\mathcal G_n\) satisfy
\[
\Pn\log\widehat g_n\ge\Pn\log g_n^\circ-\lambda_n.
\]
Let \(\epsilon_n\downarrow0\) with \(n\epsilon_n^2\to\infty\), and assume that, for fixed constants
\(C_a,C_\lambda<\infty\),
\(\KL(g_0,g_n^\circ)+\V(g_0,g_n^\circ)\le C_a\epsilon_n^2\) and
\(\lambda_n\le C_\lambda\epsilon_n^2\). Let \(c_{\mathrm{LR}}>0\) denote the universal likelihood-ratio
constant in Theorem~1 of \citet{wongshen1995}. Choose a fixed \(A>0\) such that
\(c_{\mathrm{LR}}A^2>C_a+C_\lambda+1\). Suppose that, for all sufficiently large \(n\),
\[
\int_{c_0(A\epsilon_n)^2}^{c_1A\epsilon_n}
\left\{1+H_{[]}(u,\mathcal G_n,\dH)\right\}^{1/2}du
\le
c_{\mathrm{WS}}\sqrt n\,(A\epsilon_n)^2,
\]
where \(c_0,c_1,c_{\mathrm{WS}}>0\) are the constants in that likelihood-ratio inequality. Then
\(\dH(\widehat g_n,g_0)=O_p(\epsilon_n)\). The same statement holds in outer probability under the
usual separability convention.
\end{lemma}
\begin{proof}
Write \(Z_n(g)=\Pn\log(g/g_0)\). Under the assumption, Theorem~1 of
\citet{wongshen1995} gives constants \(c_{\mathrm{LR}},C_{\mathrm{LR}}>0\) such that
\[
\Pzero^n\left\{
\sup_{g\in\mathcal G_n:\,\dH(g,g_0)\ge A\epsilon_n}Z_n(g)
\ge-c_{\mathrm{LR}}A^2\epsilon_n^2
\right\}
\le4\exp(-C_{\mathrm{LR}}nA^2\epsilon_n^2).
\]
For \(W_i=\log\{g_n^\circ(Y_i)/g_0(Y_i)\}\), one has
\(E_0W_i=-\KL(g_0,g_n^\circ)\) and \(E_0W_i^2=\V(g_0,g_n^\circ)\). Hence Chebyshev's inequality gives
\[
\Pzero^n\{Z_n(g_n^\circ)<-(C_a+1)\epsilon_n^2\}
\le\frac{C_a}{n\epsilon_n^2}=o(1).
\]
By~assumption, with probability tending to one,
\(Z_n(\widehat g_n)\ge-(C_a+C_\lambda+1)\epsilon_n^2\). Since
\(c_{\mathrm{LR}}A^2>C_a+C_\lambda+1\), the complement of the event obtained by \citet{wongshen1995} forces \(\dH(\widehat g_n,g_0)<A\epsilon_n\). This proves the claim.
\end{proof}

\subsection{Hellinger bracketing for the observed fixed-scale Gaussian sieve}
\label{app:gaussian-fixed-scale-hellinger}

For \(1\le K\le K_n\), let
\[
\mathcal F_{K,n}^G
=
\left\{
f\in\mathcal F_n^G:
f\text{ has }K\text{ mixture components}
\right\},
\qquad
\mathcal G_{K,n}^G
=
\left\{
g_f:f\in\mathcal F_{K,n}^G
\right\},
\]
and write
\[
\mathcal G_n^G
=
\bigcup_{1\le K\le K_n}\mathcal G_{K,n}^G.
\]

Throughout this subsection, a Hellinger bracket \([l,u]\) has nonnegative integrable endpoints
satisfying \(l\le g\le u\); the endpoints need not themselves integrate to one.

\begin{lemma}[Hellinger bracketing entropy of the observed fixed-scale Gaussian sieve]
\label{lem:gaussian-fixed-scale-hellinger}
Suppose that the measurement error is Gaussian with \(\sigma_\varepsilon>0\), and that the component
scales satisfy the fixed bounds $0<\underline\sigma_G\le\sigma_k\le\overline\sigma_G<\infty$.
There exist constants \(C_1,C_2<\infty\), depending only on
\(\underline\sigma_G\), \(\overline\sigma_G\), and \(\sigma_\varepsilon\), such that, for every
\(K\ge2\) and \(0<\epsilon<1\),
\[
\log N_{[]}\left(
\epsilon,\mathcal G_{K,n}^G,\dH
\right)
\le
C_1K
\log\left\{
\frac{C_2K(1+M_n+\overline\sigma_G)}{\epsilon^2}
\right\}.
\]
Consequently,
\[
\log N_{[]}\left(
\epsilon,\mathcal G_n^G,\dH
\right)
\le
\log K_n
+
C_1K_n
\log\left\{
\frac{C_2K_n(1+M_n+\overline\sigma_G)}{\epsilon^2}
\right\}.
\]
\end{lemma}

\begin{proof}
For fixed \(K\), write
\(g_\theta(y)=\sum_{k=1}^K\pi_kp_{\mu_k,\sigma_k}(y)\), where
\(p_{\mu,\sigma}(y)=\phiN(y;\mu,\sigma^2+\sigma_\varepsilon^2)\). Put
\(\underline v^2=\underline\sigma_G^2+\sigma_\varepsilon^2\) and
\(\overline v^2=\overline\sigma_G^2+\sigma_\varepsilon^2\). For every \(\mu_0\in\mathbb R\),
\(|\mu-\mu_0|\le1\), and \(\sigma\in[\underline\sigma_G,\overline\sigma_G]\), the inequality
\((y-\mu)^2\ge2^{-1}(y-\mu_0)^2-1\) gives
\[
p_{\mu,\sigma}(y)\le C\exp\{-c(y-\mu_0)^2\}.
\]
The component derivatives satisfy
\(\partial_\mu p_{\mu,\sigma}(y)=(y-\mu)p_{\mu,\sigma}(y)/(\sigma^2+\sigma_\varepsilon^2)\) and
\(\partial_\sigma p_{\mu,\sigma}(y)=\sigma\{(y-\mu)^2/(\sigma^2+\sigma_\varepsilon^2)^2
-1/(\sigma^2+\sigma_\varepsilon^2)\}p_{\mu,\sigma}(y)\). Thus,
\begin{align*}
\sup_{|\mu-\mu_0|\le1,\,\underline\sigma_G\le\sigma\le\overline\sigma_G}
|\partial_\mu p_{\mu,\sigma}(y)|
&\le C\{1+|y-\mu_0|\}e^{-c(y-\mu_0)^2},
\\
\sup_{|\mu-\mu_0|\le1,\,\underline\sigma_G\le\sigma\le\overline\sigma_G}
|\partial_\sigma p_{\mu,\sigma}(y)|
&\le C\{1+(y-\mu_0)^2\}e^{-c(y-\mu_0)^2}.
\end{align*}
The right-hand sides of the three envelope bounds above have uniformly bounded integrals.

Use \(K-1\) free weight coordinates, with \(\pi_K=1-\sum_{j<K}\pi_j\). Since
\(\partial_{\pi_j}g_\theta=p_{\mu_j,\sigma_j}-p_{\mu_K,\sigma_K}\), every free-weight derivative has
an integrable local envelope with integral at most two. Partition the admissible free-coordinate set
into sup-norm cells of side \(\delta\le1\), and select a representative \(\theta_m\) from each nonempty
cell. The intersection of a cell with the weight simplex and the scale-location rectangle is convex.
Summing the local envelopes of the \(3K-1\) free coordinates therefore gives a function \(H_m\) such
that
\[
\int H_m(y)\,dy\le CK
\]
uniformly over all cells, \(K\), and \(n\), and the integral form of the mean-value theorem yields
\(|g_\theta(y)-g_{\theta_m}(y)|\le\delta H_m(y)\) throughout the cell.

Define \(l_m=(g_{\theta_m}-\delta H_m)_+\) and \(u_m=g_{\theta_m}+\delta H_m\). Then
\(l_m\le g_\theta\le u_m\), and
\[
\int(\sqrt{u_m}-\sqrt{l_m})^2
\le\int(u_m-l_m)
\le2\delta\int H_m
\le CK\delta.
\]
Choosing \(\delta=c\epsilon^2/K\) therefore produces an \(\epsilon\)-bracket. The \(K-1\) free
weights range over a subset of \([0,1]^{K-1}\), the \(K\) locations range over an interval of length
\(2M_n\), and the \(K\) scales range over a fixed compact interval. The number of nonempty cells is
bounded by
\[
\left(\frac{C}{\delta}\right)^{K-1}
\left\{\frac{C(1+M_n)}{\delta}\right\}^{K}
\left\{\frac{C(1+\overline\sigma_G)}{\delta}\right\}^{K}.
\]
Taking logarithms and substituting \(\delta=c\epsilon^2/K\) proves
the first claim. The case \(K=1\) is absorbed by increasing the constants.
Finally,
\(N_{[]}(\epsilon,\mathcal G_n^G,\dH)\le
\sum_{K=1}^{K_n}N_{[]}(\epsilon,\mathcal G_{K,n}^G,\dH)\), which proves
the final statement.
\end{proof}

\subsection{Fixed-scale Gaussian approximation and observed-domain rates}
\label{app:gaussian-fixed-scale-observed-rates}

This subsection collects the approximation and observed-domain likelihood
bounds used by the fixed-scale Gaussian results in the main paper.
\begin{proposition}[Gaussian-mixture approximation]
\label{prop:supp-gaussian-smoothed-klv-approximation}
Let \(f_0=\phi_\tau*Q_0\in\mathcal C_G\), and write
$g_0=f_0*\phi_{\sigma_\varepsilon}.$
For every \(M\geq1\), integer \(K\geq2\), and \(\pi_{\min}>0\) such that \(K\pi_{\min}\leq1/2\), there exists a \(K\)-component mixture
\[
f_{M,K}^\circ
=
\sum_{k=1}^K
\pi_k^\circ\phiN(\cdot;\mu_k^\circ,\tau^2)
\]
with \(|\mu_k^\circ|\leq M\) and \(\pi_k^\circ\geq\pi_{\min}\) such that
\begin{equation}
\label{eq:gaussian-smoothed-kl-approximation}
\KL(g_0,g_{f_{M,K}^\circ})
\leq
C\left\{
\frac{M^2}{K^2}
+e^{-cM^2}
+K\pi_{\min}
\right\}.
\end{equation}
Moreover, with \(g_{M,K}^\circ=g_{f_{M,K}^\circ}\) and
$\V(p,q)=\int p(y)\{\log[p(y)/q(y)]\}^2dy,$
the mixture satisfies
\begin{equation}
\label{eq:gaussian-klv-approximation}
\KL(g_0,g_{M,K}^\circ)
+
\V(g_0,g_{M,K}^\circ)
\leq
C\left\{
\frac{M^2}{K^2}
+
\frac{M^4}{K^4}
+
 e^{-cM^2}
+
K\pi_{\min}(1+M^4)
\right\}.
\end{equation}

\end{proposition}

\begin{proof}
Let \(U\sim Q_0\). Choose \(K\) equally spaced grid points in
\([-M,M]\), clip \(U\) to this interval, and let \(q_{M,K}(U)\) be a
nearest grid point, with any fixed rule for breaking ties. For
\(m\in\{2,4\}\),
\[
\E\left|U-q_{M,K}(U)\right|^m
\leq
C\left\{
\frac{M^m}{K^m}
+
\E(|U|-M)_+^m
\right\}.
\]
The subgaussian tail bound on \(Q_0\) gives
\[
\E(|U|-M)_+^m
\leq
Ce^{-cM^2},
\qquad m\in\{2,4\}.
\]
Let \(Q_{M,K}\) denote the law of \(q_{M,K}(U)\), and set
$
s_0^2=\tau^2+\sigma_\varepsilon^2.$
The true and discretized observed variables can be coupled as
\[
Y=U+s_0Z,
\qquad
Y^\circ=q_{M,K}(U)+s_0Z,
\]
where \(Z\sim N(0,1)\) is independent of \(U\). Conditional on \(U\),
\[
\KL\left\{
N(U,s_0^2),
N(q_{M,K}(U),s_0^2)
\right\}
=
\frac{\{U-q_{M,K}(U)\}^2}{2s_0^2}.
\]
The data-processing inequality therefore yields
\[
\KL(g_0,\phi_{s_0}*Q_{M,K})
\leq
C\left\{
\frac{M^2}{K^2}
+
 e^{-cM^2}
\right\}.
\]

To enforce the weight floor, let \(H_{M,K}\) be the uniform distribution
on the same grid and define
\[
\alpha=K\pi_{\min},
\qquad
Q_{M,K}^\dagger
=
(1-\alpha)Q_{M,K}
+
\alpha H_{M,K}.
\]
Because \(\alpha/K=\pi_{\min}\), every grid point has weight at least
\(\pi_{\min}\). Moreover,
\[
\phi_{s_0}*Q_{M,K}^\dagger
\geq
(1-\alpha)\phi_{s_0}*Q_{M,K},
\]
so
\[
\KL(g_0,\phi_{s_0}*Q_{M,K}^\dagger)
\leq
\KL(g_0,\phi_{s_0}*Q_{M,K})
-
\log(1-\alpha).
\]
Since \(\alpha\leq1/2\),
 we have $-\log(1-\alpha)\leq2\alpha$, and it follows that
\[
\KL(g_0,\phi_{s_0}*Q_{M,K}^\dagger)
\leq
C\left\{
\frac{M^2}{K^2}
+
 e^{-cM^2}
+
K\pi_{\min}
\right\},
\]
which proves~\eqref{eq:gaussian-smoothed-kl-approximation}.

We now control the second log-likelihood moment for the same
approximating mixture. Let \(B\sim\operatorname{Bernoulli}(\alpha)\)
and \(W\sim H_{M,K}\) be independent of \(U\), and define
\[
T
=
(1-B)q_{M,K}(U)+BW,
\qquad
\Delta=U-T.
\]
The law of \(T\) is \(Q_{M,K}^\dagger\). Consider two joint laws on
\((U,B,W,Y)\) with the same marginal distribution for \((U,B,W)\):
under \(P\), $Y=U+s_0Z$, whereas under \(Q\), $Y=T+s_0Z$.

Their \(Y\)-marginals are \(g_0\) and
\(\phi_{s_0}*Q_{M,K}^\dagger\), respectively. Under \(P\), the joint
log-likelihood ratio is
\[
L
=
\frac{\Delta Z}{s_0}
+
\frac{\Delta^2}{2s_0^2},
\]
whereas under \(Q\),
\[
L
=
\frac{\Delta Z}{s_0}
-
\frac{\Delta^2}{2s_0^2}.
\]
Consequently,
\[
\E_P(L^2)
=
\E_Q(L^2)
=
\E\left\{
\frac{\Delta^2}{s_0^2}
+
\frac{\Delta^4}{4s_0^4}
\right\}.
\]
Lemma~\ref{lem:marginal-log-moment} and data processing for
Kullback--Leibler divergence give
\[
\V(g_0,\phi_{s_0}*Q_{M,K}^\dagger)
\leq
2\E\left\{
\frac{\Delta^2}{s_0^2}
+
\frac{\Delta^4}{4s_0^4}
\right\},
\]
and
\[
\KL(g_0,\phi_{s_0}*Q_{M,K}^\dagger)
\leq
\frac{\E\Delta^2}{2s_0^2}.
\]

For \(m\in\{2,4\}\), conditioning on \(B\) gives
\[
\E|\Delta|^m
\leq
(1-\alpha)
\E\left|U-q_{M,K}(U)\right|^m
+
\alpha\E|U-W|^m.
\]
The subgaussian assumption implies \(\E|U|^4\leq C\) uniformly over
\(\mathcal C_G\), while \(|W|\leq M\). Therefore,
\[
\E|\Delta|^m
\leq
C\left\{
\frac{M^m}{K^m}
+
 e^{-cM^2}
+
\alpha(1+M^m)
\right\},
\qquad m\in\{2,4\}.
\]
Substituting the preceding moment bound into the marginal log-moment inequality above, and using
\(\alpha=K\pi_{\min}\), proves~\eqref{eq:gaussian-klv-approximation}.

\end{proof}
The constants \(C,c>0\) may depend on the fixed class parameters and on
\(\sigma_\varepsilon\), but not on \(f_0,M,K\), or \(\pi_{\min}\). If
\(M/K\leq1\), the first two terms on the right of
\eqref{eq:gaussian-klv-approximation} are bounded by \(CM^2/K^2\).
Finally, \(f_{M,K}^\circ\in\mathcal F_n^G\) whenever
\(K\leq K_n\), \(M\leq M_n\), and the scale interval of
\(\mathcal F_n^G\) contains \([\underline\tau,\overline\tau]\).

\begin{proposition}[Global observed-domain rate]
\label{prop:supp-gaussian-smoothed-observed-rate}
Suppose that \(f_0\in\mathcal C_G\), and that the estimator is computed over \(\mathcal F_n^G\) with
$K_n
=
\max\left\{
2,
\left\lfloor n^{1/5}(\log n)^{-1/5}\right\rfloor
\right\},
\qquad
M_n=C_M\sqrt{\log n},
\pi_{\min,n}=n^{-u},$
where \(u\geq3/5\), \(K_n\pi_{\min,n}\leq1/2\) for all sufficiently large \(n\), and \(C_M\) is sufficiently large. Under the iterated-log BIC penalty and exact global maximization,
\[
\KL(g_0,g_{\widehat f})^{1/2}
=
O_p\left\{
n^{-1/5}(\log n)^{7/10}
\right\}.
\]
Under the conditions of Lemma~\ref{lem:supp-kl-to-l2},
\[
\|g_{\widehat f}-g_0\|_2
=
O_p\left\{
n^{-1/5}(\log n)^{7/10}
\right\}.
\]
\end{proposition}

\begin{proof}
Because \(f_0\in\mathcal C_G\), the observed distribution is
subgaussian and, in particular, \(\E_0Y^4<\infty\). Hence
Proposition~\ref{prop:gaussian-global-likelihood} applies.

Use Proposition~\ref{prop:supp-gaussian-smoothed-klv-approximation} with
$K=K_n,
M=M_n$, and $\pi_{\min}=\pi_{\min,n}.$
Let \(c_{\mathrm{app}}>0\) be the exponential-tail constant in
\eqref{eq:gaussian-smoothed-kl-approximation}. Since \(u\geq3/5\),
\[
K_n\pi_{\min,n}
=
o\left\{
n^{-2/5}(\log n)^{7/5}
\right\}.
\]
Choosing \(C_M\) so that \(c_{\mathrm{app}}C_M^2>2/5\) gives
\[
e^{-c_{\mathrm{app}}M_n^2}
=
o(n^{-2/5}).
\]
Therefore Assumption~\ref{main-ass:primitive-approximation} holds with
\[
b_n
=
O\left(\frac{M_n}{K_n}\right)
=
O\left\{
n^{-1/5}(\log n)^{7/10}
\right\}.
\]

For the fixed-scale sieve,
Proposition~\ref{prop:gaussian-global-likelihood} gives
\[
F_n=O(\log n),
\qquad
V_n=O(K_n),
\qquad
\log L_n=O(\log n).
\]
Since
$K_n
\asymp
n^{1/5}(\log n)^{-1/5}$,
the square-root part of the empirical-process bound is
\[
F_n
\sqrt{\frac{V_n\log L_n}{n}}
=
O\left\{
n^{-2/5}(\log n)^{7/5}
\right\},
\]
whereas the linear part is
\[
F_n
\frac{V_n\log L_n}{n}
=
O\left\{
n^{-4/5}(\log n)^{9/5}
\right\}.
\]
Consequently,
\[
c_n^{1/2}
=
O\left\{
n^{-1/5}(\log n)^{7/10}
\right\}.
\]

Along the approximating sequence, the square root of the normalized
iterated-log BIC penalty is
\[
O\left[
\left(\frac{K_n\log n}{n}\right)^{1/2}
\{\mathsf L^{\circ3}(n)\}^{1/2}
\right]
=
O\left\{
n^{-2/5}(\log n)^{2/5}
\{\mathsf L^{\circ3}(n)\}^{1/2}
\right\},
\]
which is smaller in polynomial order. Proposition~\ref{main-prop:observed-kl-rate}
therefore yields
\[
\KL(g_0,g_{\widehat f})^{1/2}
=
O_p\left\{
n^{-1/5}(\log n)^{7/10}
\right\}.
\]
Lemma~\ref{lem:supp-kl-to-l2} gives the same order for
\(\|g_{\widehat f}-g_0\|_2\).
\end{proof}

\subsection{Fixed-scale Gaussian inverse and posterior bounds}
\label{app:gaussian-fixed-scale-inverse-posterior}

The following results provide the inverse-stability, latent-tail, and
posterior empirical-process bounds used by the fixed-scale Gaussian
theorems in the main paper.

\begin{lemma}[Gaussian analytic inverse stability]
\label{lem:supp-gaussian-analytic-stability}
Let \(f_1,f_2\) be densities satisfying
\(|\varphi_{f_j}(t)|\leq\exp(-\underline\sigma_G^2t^2/2)\) for
\(t\in\mathbb R\) and \(j=1,2\), and let
\(g_j=f_j*\phi_{\sigma_\varepsilon}\). Define
\(d=\|g_1-g_2\|_2\) and
\(\vartheta=\underline\sigma_G^2
(\underline\sigma_G^2+\sigma_\varepsilon^2)^{-1}\). There exist constants
\(C,d_0>0\), depending only on \(\underline\sigma_G\),
\(\sigma_\varepsilon\), and the Fourier convention, such that, whenever
\(0<d\leq d_0\),
\[
\|f_1-f_2\|_2
\leq
C d^\vartheta\{\log(d^{-1})\}^{-(1-\vartheta)/4}.
\]
If \(d=0\), then \(f_1=f_2\). Possibly after enlarging \(C\), the
coarser bound \(\|f_1-f_2\|_2\leq Cd^\vartheta\) holds for all
\(d\geq0\).
\end{lemma}

\begin{proof}
Set \(a=\underline\sigma_G^2\), \(b=\sigma_\varepsilon^2\), and
\(L=\log(d^{-1})\). The claim for \(d\) bounded away from zero follows
after enlarging \(C\), so assume that \(L\) is sufficiently large.
Plancherel's identity, Gaussian deconvolution on \(|t|\leq T\), and the
analytic envelope on \(|t|>T\) give
\[
\|f_1-f_2\|_2^2
\leq
C e^{bT^2}d^2+C(1+T)^{-1}e^{-aT^2}.
\]
Choose \(T^2=\{2L-(\log L)/2\}(a+b)^{-1}\). Then both terms on the
right are of order
\(d^{2a/(a+b)}L^{-b/[2(a+b)]}\). Since
\(\vartheta=a(a+b)^{-1}\), taking square roots gives the stated bound.
If \(d=0\), injectivity of Gaussian convolution gives \(f_1=f_2\); the
coarser bound for larger \(d\) follows from the uniform Fourier envelope.
\end{proof}

\begin{lemma}[Latent \(L^1\) conversion]
\label{lem:supp-gaussian-smoothed-l1-transfer}
Let \(f_0\in\mathcal C_G\), and let
\(\widetilde f_n\in\mathcal F_n^G\) be possibly random. Suppose that
\(M_n=C_M\sqrt{\log n}\) and
\(\|\widetilde f_n-f_0\|_2=O_p(r_n)\). For every fixed \(D>0\), one can
choose \(C_A\) sufficiently large so that, with
\(A_n=M_n+C_A\sqrt{\log n}\),
\[
\|\widetilde f_n-f_0\|_1
=
O_p\{(\log n)^{1/4}r_n+n^{-D}\}.
\]
\end{lemma}

\begin{proof}
Every density in \(\mathcal F_n^G\) is a Gaussian mixture with
\(|\mu_k|\leq M_n\) and \(\sigma_k\leq\overline\sigma_G\). For
sufficiently large \(C_A\), Gaussian tail bounds give
\(\sup_{f\in\mathcal F_n^G}\int_{|x|>A_n}f(x)\,dx\leq n^{-D}\).
If \(X\sim f_0\), write \(X=U+\tau Z\), where \(U\sim Q_0\),
\(Z\sim N(0,1)\), and
\(\tau\in[\underline\tau,\overline\tau]\). The subgaussian tail of
\(U\) and a Gaussian tail bound imply
\(\int_{|x|>A_n}f_0(x)\,dx\leq n^{-D}\) after increasing \(C_A\).
On \([-A_n,A_n]\), Cauchy--Schwarz therefore yields
\[
\|\widetilde f_n-f_0\|_1
\leq
(2A_n)^{1/2}\|\widetilde f_n-f_0\|_2+2n^{-D}.
\]
Since \(A_n\asymp\sqrt{\log n}\), the result follows.
\end{proof}

\begin{lemma}[Posterior-kernel parameter regularity]
\label{lem:supp-gaussian-posterior-lipschitz}
For fixed \(K\), parameterize the mixture weights by logits and write
\(\eta=(a,\mu,\sigma)\). Under the scale bounds of \(\mathcal F_n^G\),
there exists a constant \(C\), depending only on
\(\underline\sigma_G\), \(\overline\sigma_G\), and
\(\sigma_\varepsilon\), and independent of \(K,n,y\), such that
\(\|\partial_{a_j}q_\eta(\cdot,y)\|_1\leq2\),
\(\|\partial_{\mu_j}q_\eta(\cdot,y)\|_1\leq
C\{1+|y|+M_n\}\), and
\(\|\partial_{\sigma_j}q_\eta(\cdot,y)\|_1\leq
C\{1+(|y|+M_n)^2\}\). Consequently,
\[
\|q_\eta(\cdot,y)-q_{\widetilde\eta}(\cdot,y)\|_1
\leq
CK\{1+(|y|+M_n)^2\}\|\eta-\widetilde\eta\|_\infty.
\]
\end{lemma}

\begin{proof}
For component \(k\), put
\(v_k=\sigma_k^2+\sigma_\varepsilon^2\),
\(\lambda_k=\sigma_k^2v_k^{-1}\),
\(w_k^2=\sigma_k^2\sigma_\varepsilon^2v_k^{-1}\), and
\(m_k(y)=\mu_k+\lambda_k(y-\mu_k)\). Gaussian conjugacy gives
\(q_\eta(\cdot,y)=\sum_{k=1}^K r_k(y)p_k(\cdot\mid y)\), where
\(p_k(\cdot\mid y)=\phiN\{\cdot;m_k(y),w_k^2\}\).
For a logit coordinate,
\(\partial_{a_j}q_\eta=r_j(p_j-q_\eta)\), which gives the first bound.
For \(z_j\in\{\mu_j,\sigma_j\}\), differentiating the responsibilities
gives
\(\partial_{z_j}q_\eta=r_js_{z,j}(p_j-q_\eta)
+r_j\partial_{z_j}p_j\), with
\(s_{z,j}=\partial_{z_j}\log\phiN(y;\mu_j,v_j)\). The fixed scale
bounds imply
\(|s_{\mu,j}(y)|\leq C(|y|+M_n)\) and
\(|s_{\sigma,j}(y)|\leq C\{1+(|y|+M_n)^2\}\), while the conditional
standard deviations \(w_j\) are bounded above and away from zero.
Standard Gaussian score bounds therefore give the stated coordinate
inequalities. Integrating the gradient along the line segment joining
\(\eta\) and \(\widetilde\eta\), and summing over the \(O(K)\)
coordinates, proves the final Lipschitz bound.
\end{proof}

\begin{proposition}[Posterior empirical-process rate]
\label{prop:supp-gaussian-smoothed-posterior-complexity}
Suppose that \(f_0\in\mathcal C_G\), the measurement error is Gaussian,
and the candidate class is \(\mathcal F_n^G\). If
\(M_n=C_M\sqrt{\log n}\), \(\pi_{\min,n}^{-1}\) grows at most
polynomially, and \(K_n\log n=o(n)\), then
\[
\sup_{f\in\mathcal F_n^G}
\|(\Pn-\Pzero)q_f\|_1
=
O_p\{(K_n\log n/n)^{1/2}\}.
\]
Under the baseline calibration
\(K_n=\max\{2,\lfloor n^{1/5}(\log n)^{-1/5}\rfloor\}\), this becomes
\(O_p\{n^{-2/5}(\log n)^{2/5}\}\).
\end{proposition}

\begin{proof}
For fixed \(f\), set \(Z_f=\|(\Pn-\Pzero)q_f\|_1\). Replacing one
observation changes \(Z_f\) by at most \(2/n\), so McDiarmid's
inequality gives \(\Pr\{Z_f-\E Z_f>t\}\leq\exp(-cnt^2)\).
Let \(A_n=M_n+C_A\sqrt{\log n}\). The subgaussian tail of the observed
variable and the fixed conditional Gaussian scale bounds imply, for any
fixed \(D>0\) and sufficiently large \(C_A\),
\(\sup_{f\in\mathcal F_n^G}\E_0\int_{|x|>A_n}q_f(x\mid Y)\,dx
\leq Cn^{-D}\). On \([-A_n,A_n]\), Cauchy--Schwarz, Jensen's inequality,
and Fubini's theorem, together with the uniform conditional \(L^2\) bound,
give \(\sup_f\E Z_f=O\{n^{-1/2}(\log n)^{1/4}\}\).

For each \(K\leq K_n\), cover the ambient logit, location, and scale box
by a sup-norm grid of mesh \(\delta_n=n^{-B}\), with \(B\) sufficiently
large. Polynomial growth of \(\pi_{\min,n}^{-1}\) gives
\(\log N_n=O(K_n\log n)\). McDiarmid's inequality and a union bound then
show that the maximum of \(Z_f\) over the grid is
\(O_p\{(K_n\log n/n)^{1/2}+n^{-1/2}(\log n)^{1/4}\}\).
Lemma~\ref{lem:supp-gaussian-posterior-lipschitz} makes the grid remainder
negligible because
\((\Pn+\Pzero)\{1+(|Y|+M_n)^2\}=O_p(\log n)\). The first term dominates,
which proves the result; without a separability convention, the same
conclusion holds in outer probability.
\end{proof}

\subsection{Truncated lower-density bounds for posterior stability}
\label{app:posterior-truncation-material}
The bounds below provide a generic sufficient route for posterior kernels outside the fixed-scale Gaussian branch.

\subsubsection*{Sieve-uniform lower-density calibration}

The quantities \(\mathcal Y_n\) and \(m_n^{\mathrm{sieve}}\) are introduced for the generic truncated posterior-kernel
verification in Proposition~\ref{prop:posterior-stability-gaussian}. The following lemmas provide the
required sieve-uniform lower density on \(\mathcal Y_n\).

\begin{lemma}[Explicit lower bound for the convolved Gaussian sieve]
\label{lem:explicit-mn}
Let \(\mathcal Y_n=[-B_n,B_n]\). Let \(\mathcal F_n\) be the regularized
Gaussian-mixture sieve defined in 7--6.
Assume that \(f_\varepsilon\) is strictly positive on compact intervals. Then,
for every \(h_n>0\),
\[
\inf_{f\in\mathcal F_n}\inf_{y\in\mathcal Y_n} g_f(y)
\ge
\eta_{\varepsilon,n}(h_n)
\left\{
2\Phi\left(\frac{h_n}{\sigma_{\max,n}}\right)-1
\right\},
\]
where
\[
\eta_{\varepsilon,n}(h_n)
=
\inf_{|u|\le B_n+M_n+h_n} f_\varepsilon(u).
\]
In particular, choosing \(h_n=\sigma_{\max,n}\) gives
\[
\inf_{f\in\mathcal F_n}\inf_{y\in\mathcal Y_n} g_f(y)
\ge
\{2\Phi(1)-1\}
\inf_{|u|\le B_n+M_n+\sigma_{\max,n}} f_\varepsilon(u).
\]
The bound is uniform over \(K\le K_n\), over the mixture weights, and over
\(\sigma_{\min,n}\).
\end{lemma}
\begin{proof}
For a single Gaussian component,
\[
    g_k(y;\mu_k,\sigma_k^2)
    =
    \int \phi(x;\mu_k,\sigma_k^2)f_\varepsilon(y-x)\,dx .
\]
Restrict the integral to the interval \(|x-\mu_k|\le h_n\). If
\(y\in[-B_n,B_n]\), \(|\mu_k|\le M_n\), and \(|x-\mu_k|\le h_n\), then
\(
    |y-x|\le B_n+M_n+h_n .
\)
Therefore
\(
    f_\varepsilon(y-x)
    \ge
    \eta_{\varepsilon,n}(h_n)
\)
on this restricted interval. Hence
\[
    g_k(y;\mu_k,\sigma_k^2)
    \ge
    \eta_{\varepsilon,n}(h_n)
    \Pr\{|Z-\mu_k|\le h_n\},
    \qquad
    Z\sim N(\mu_k,\sigma_k^2).
\]
Since \(\sigma_k\le\sigma_{\max,n}\),
\[
    \Pr\{|Z-\mu_k|\le h_n\}
    =
    2\Phi\left(\frac{h_n}{\sigma_k}\right)-1
    \ge
    2\Phi\left(\frac{h_n}{\sigma_{\max,n}}\right)-1 .
\]
The same lower bound holds for every component. Averaging with the mixture
weights, which sum to one,  yields the stated lower bound for \(g_f\).
\end{proof}

\begin{corollary}[Laplace measurement error]
\label{cor:explicit-mn-laplace}
Suppose that \(\varepsilon\sim\mathrm{Laplace}(0,b)\), 
and let \(\mathcal Y_n=[-B_n,B_n]\). Under the conditions of
Lemma~\ref{lem:explicit-mn}, choosing \(h_n=\sigma_{\max,n}\) gives
\[
\inf_{f\in\mathcal F_n}\inf_{y\in\mathcal Y_n} g_f(y)
\ge
\{2\Phi(1)-1\}\frac{1}{2b}
\exp\left\{-\frac{B_n+M_n+\sigma_{\max,n}}{b}\right\}.
\]
Consequently, if \(\sigma_{\max,n}\) is bounded above, then
\[
(m_n^{\mathrm{sieve}})^{-1}
\le
C\exp\{C(B_n+M_n)\},
\]
so the exponent is linear, rather than quadratic, in \(B_n+M_n\).
\end{corollary}
\begin{proof}
For Laplace error, \(f_\varepsilon(u)=(2b)^{-1}\exp(-|u|/b)\) is decreasing in \(|u|\). Therefore
\[
    \eta_{\varepsilon,n}(h_n)
    =\inf_{|u|\le B_n+M_n+h_n}f_\varepsilon(u)
    =\frac{1}{2b}\exp\left\{-\frac{B_n+M_n+h_n}{b}\right\}.
\]
Substituting \(h_n=\sigma_{\max,n}\) in Lemma~\ref{lem:explicit-mn} gives the obtained bound. If \(\sigma_{\max,n}\) is bounded, the stated exponential order follows immediately.
\end{proof}

The Gaussian-error case is treated separately,
since the convolution with a Gaussian component admits a closed-form
calculation.
\begin{corollary}[Gaussian measurement error]
\label{cor:explicit-mn-gaussian}
Suppose that \(\varepsilon\sim N(0,\sigma_\varepsilon^2)\), with
\(\sigma_\varepsilon>0\), and let \(\mathcal Y_n=[-B_n,B_n]\). Then
\[
\inf_{f\in\mathcal F_n}\inf_{y\in\mathcal Y_n} g_f(y)
\ge
\frac{1}{\sqrt{2\pi}}
\min_{v\in\{v_{-,n},v_{+,n}\}}
v^{-1/2}
\exp\left\{-\frac{(B_n+M_n)^2}{2v}\right\},
\]
where
\(
v_{-,n}
=
\sigma_\varepsilon^2+\sigma_{\min,n}^2,
\) and
\(
v_{+,n}
=
\sigma_\varepsilon^2+\sigma_{\max,n}^2.
\)
Consequently, if \(\sigma_{\max,n}\) is bounded above, then, for
\(S_n=B_n+M_n\), we have
\(
(m_n^{\mathrm{sieve}})^{-1/2}
\le
C\exp\{C S_n^2\}
\)
for a finite constant \(C>0\). In particular, if \(S_n\le C_S\sqrt{\log n}\),
then
\[
(m_n^{\mathrm{sieve}})^{-1/2}
\le
C n^{C C_S^2},
\]
while if \(S_n=o(\sqrt{\log n})\), then
\[
(m_n^{\mathrm{sieve}})^{-1/2}
=
n^{o(1)}.
\]
\end{corollary}
\begin{proof}
Under Gaussian measurement error, a latent Gaussian component with variance
\(\sigma^2\) is observed as a Gaussian density with variance
\(v=\sigma_\varepsilon^2+\sigma^2\):
\(
    g_k(y;\mu,\sigma)
    =
    \phi(y;\mu,v).
\)
For \(y\in[-B_n,B_n]\) and \(|\mu|\le M_n\),
\(|y-\mu|\le B_n+M_n\). Also
\(v\in[v_{-,n},v_{+,n}]\).
Thus
\[
    g_k(y;\mu,\sigma)
    \ge
    \frac{1}{\sqrt{2\pi}}
    \min_{v\in[v_{-,n},v_{+,n}]}
    v^{-1/2}
    \exp\left\{-\frac{(B_n+M_n)^2}{2v}\right\}.
\]
The function
\(v\mapsto v^{-1/2}\exp\{-(B_n+M_n)^2/(2v)\}\) has its minimum over the compact
interval \([v_{-,n},v_{+,n}]\) at one of the two endpoints. This gives the
 lower bound after averaging over components.
If \(\sigma_{\max,n}\) is bounded above, then \(v_{+,n}\) is bounded above. Taking reciprocal square roots of the preceding lower bound gives \(C\exp\{C(B_n+M_n)^2\}\). If \(S_n\le C_S\sqrt{\log n}\), then this factor is bounded by
\(C n^{C C_S^2}\). If \(S_n=o(\sqrt{\log n})\), it is \(n^{o(1)}\).
\end{proof}

The quadratic loss is not an artifact of the proof. For the one-component
candidate \(f(x)=\phi(x;M_n,\sigma_{\min,n}^2)\), the observed Gaussian
convolution at \(y=-B_n\) is proportional to
\(\exp\{-(B_n+M_n)^2/[2(\sigma_{\min,n}^2+\sigma_\varepsilon^2)]\}\).
Hence any lower bound uniform over \(f\in\mathcal F_n\) and
\(y\in[-B_n,B_n]\) must allow this worst-case quadratic exponent.

\subsection{Fourier and latent-tail bounds for inverse transfer}
\label{app:fourier-inverse-material}

Proposition~\ref{main-prop:latent-rate-transfer} transfers observed-density
accuracy to latent-density accuracy through the inverse factor
\(\kappa_\varepsilon(T)\), Fourier truncation, and latent-tail control.
We first record the specialization to exponential inverse growth and then
give the additional bounds used for the Gaussian-mixture sieve.
\begin{corollary}[Conditional exponential inverse-growth transfer]
\label{cor:supp-generic-supersmooth-transfer}
Suppose that the conditions of
Proposition~\ref{main-prop:latent-rate-transfer} hold with \(d_n\to0\),
and that, for some
\(\alpha,c_\varepsilon,C_\varepsilon>0\) and \(q_\varepsilon\ge0\),
\(\kappa_\varepsilon(T)\le
C_\varepsilon(1+T)^{q_\varepsilon}\exp(c_\varepsilon T^\alpha)\)
for all sufficiently large \(T\).

Fix \(\rho\in(0,1)\) and set
\(T_n=\{\rho c_\varepsilon^{-1}\log(d_n^{-1})\}^{1/\alpha}\).
Then
\[
\begin{aligned}
\|\widehat f-f_0\|_2
&=
O_p\!\left[
d_n^{1-\rho}\{\log(d_n^{-1})\}^{q_\varepsilon/\alpha}
+
(R_n+R_0)\{\log(d_n^{-1})\}^{-s/\alpha}
\right],\\
\|\widehat f-f_0\|_1
&=
O_p\!\left[
A_n^{1/2}
\left\{
d_n^{1-\rho}\{\log(d_n^{-1})\}^{q_\varepsilon/\alpha}
+
(R_n+R_0)\{\log(d_n^{-1})\}^{-s/\alpha}
\right\}
+\eta_n
\right].
\end{aligned}
\]
\end{corollary}

\begin{proof}
For the stated \(T_n\),
\(\exp(c_\varepsilon T_n^\alpha)d_n=d_n^{1-\rho}\),
\((1+T_n)^{q_\varepsilon}
=O[\{\log(d_n^{-1})\}^{q_\varepsilon/\alpha}]\), and
\(T_n^{-s}=O[\{\log(d_n^{-1})\}^{-s/\alpha}]\).
Hence
\(\kappa_\varepsilon(T_n)d_n
=O[d_n^{1-\rho}\{\log(d_n^{-1})\}^{q_\varepsilon/\alpha}]\).
Substitution into the \(L^2\) and \(L^1\) conclusions of
Proposition~\ref{main-prop:latent-rate-transfer} gives the result.
\end{proof}

\begin{lemma}[Fourier and tail bounds of the Gaussian-mixture sieve]
\label{lem:gaussian-sieve-budgets}
Let \(f\in\mathcal F_{K,n}^{\mathrm{reg}}\). Then, for every \(s>0\), there
exists a finite constant \(C_s>0\), depending only on \(s\), such that
\[
J_s(f)
\le
R_n,
\qquad
R_n
=
C_s\left\{1+\sigma_{\min,n}^{-(s+1/2)}\right\}.
\]
Moreover, for every \(A>M_n\),
\[
\sup_{f\in\mathcal F_n}
\int_{|x|>A}f(x)\,dx
\le
2
\left\{
1-\Phi\left(
\frac{A-M_n}{\sigma_{\max,n}}
\right)
\right\}.
\]
\end{lemma}

\begin{proof}
For a Gaussian mixture,
$\varphi_f(t)
=
\sum_{k=1}^K
\pi_k\exp(it\mu_k)\exp(-\sigma_k^2t^2/2)$.

Since \(\sum_k\pi_k=1\) and \(\sigma_k\ge\sigma_{\min,n}\),
\[
|\varphi_f(t)|
\le
\sum_{k=1}^K\pi_k\exp(-\sigma_k^2t^2/2)
\le
\exp(-\sigma_{\min,n}^2t^2/2).
\]
Therefore
\[
    J_s(f)^2
    \le
    \int_{\R}(1+t^2)^s\exp(-\sigma_{\min,n}^2t^2)\,dt
    \le
    C_s\{1+\sigma_{\min,n}^{-(2s+1)}\}.
\]
The last inequality follows by splitting the integral over \(|t|\le1\) and \(|t|>1\), and by the change of variable \(u=\sigma_{\min,n}t\). For the tail bound, if \(|\mu_k|\le M_n\), \(\sigma_k\le\sigma_{\max,n}\), and \(A>M_n\), then
\[
    \Pr\{|Z_k|>A\}
    \le
    2\left\{1-\Phi\left(\frac{A-M_n}{\sigma_{\max,n}}\right)\right\},
    \qquad Z_k\sim N(\mu_k,\sigma_k^2).
\]
Averaging this inequality with weights \(\pi_k\) gives the bound.
\end{proof}
\begin{lemma}[Bias-separated inverse bound for the Gaussian-mixture sieve]
\label{lem:supp-bias-separated-inverse}
Let \(f_n^\circ\in\mathcal F_n\) be the approximating density in
Assumption~\ref{main-ass:primitive-approximation}. Suppose that
\(\varphi_\varepsilon(t)\neq0\) for every \(t\), and that
\(\widehat f\) and \(f_n^\circ\) are Gaussian mixtures whose component
standard deviations are bounded below by \(\sigma_{\min,n}\). Then, for
every \(T>0\),
\[
\|\widehat f-f_0\|_2
\le
\|f_n^\circ-f_0\|_2
+
C\kappa_\varepsilon(T)\|g_{\widehat f}-g_{f_n^\circ}\|_2
+
C\sigma_{\min,n}^{-1}T^{-1/2}
\exp(-\sigma_{\min,n}^2T^2/2).
\]
The constant \(C\) is independent of \(n\), \(T\), \(\widehat f\), and
\(f_n^\circ\), and may depend only on the Fourier-transform convention.
\end{lemma}

\begin{proof}
By the lower scale restriction, both \(\widehat f\) and \(f_n^\circ\)
can be written as
\(f=\phi(\,\cdot\,;0,\sigma_{\min,n}^2)*P_f\), where \(P_f\) is a
probability measure. Indeed, if
\(f=\sum_k\pi_k\phi(\,\cdot\,;\mu_k,\sigma_k^2)\), one may take
\(P_f=\sum_k\pi_k
\mathcal N(\mu_k,\sigma_k^2-\sigma_{\min,n}^2)\).
Consequently,
\(|\varphi_{\widehat f-f_n^\circ}(t)|
\le2\exp(-\sigma_{\min,n}^2t^2/2)\).

Since
\(\varphi_{g_{\widehat f}-g_{f_n^\circ}}
=\varphi_\varepsilon\varphi_{\widehat f-f_n^\circ}\),
splitting the Fourier domain at \(T\) and applying Plancherel's identity
gives
\[
\|\widehat f-f_n^\circ\|_2^2
\le
C\kappa_\varepsilon(T)^2
\|g_{\widehat f}-g_{f_n^\circ}\|_2^2
+
C\int_{|t|>T}
\exp(-\sigma_{\min,n}^2t^2)\,dt.
\]
Using
\(\int_T^\infty e^{-a^2t^2}\,dt
\le (2a^2T)^{-1}e^{-a^2T^2}\)
and taking square roots gives
\(\|\widehat f-f_n^\circ\|_2
\le C\kappa_\varepsilon(T)\|g_{\widehat f}-g_{f_n^\circ}\|_2
+C\sigma_{\min,n}^{-1}T^{-1/2}
\exp(-\sigma_{\min,n}^2T^2/2)\).
The result follows from
\(\|\widehat f-f_0\|_2
\le\|\widehat f-f_n^\circ\|_2+\|f_n^\circ-f_0\|_2\).
\end{proof}

\subsection{Localized ordinary-smooth rate machinery}
\label{app:localized-ordinary-rate-material}

This subsection collects the observed-density, inverse, and latent-tail bounds
used in the ordinary-smooth theorem of the main paper. Throughout,
\(\mathcal F_h^{\mathrm{reg}}\) denotes the local sieve defined in
Equation~\ref{main-eq:local-regularized-sieve}.

\begin{proposition}[Localized observed-density rate]
\label{prop:supp-local-observed-rate}
Under Assumptions~\ref{main-ass:local-ap}--\ref{main-ass:local-ll},
\[
\dH(g_{\widehat f_h},g_0)
=O_p\left\{h^{s+\beta}L_h^{a_A}
+\left(\frac{K_hL_h^a}{n}\right)^{1/2}\right\},
\]
for a finite \(a\). Since \(f_\varepsilon\) is bounded, the same order holds
for \(\|g_{\widehat f_h}-g_0\|_2\).
\end{proposition}

\begin{proof}
Assumption~\ref{main-ass:local-ap} supplies a Kullback--Leibler approximant
with second log-moment of order \(h^{2(s+\beta)}L_h^{2a_A}\). The Hellinger
bracketing-integral and penalty conditions in
Assumption~\ref{main-ass:local-ll} are precisely the hypotheses of the
likelihood-ratio inequality for a sieve centered at a Kullback--Leibler
approximant. Applying that inequality to the observed sieve \(\mathcal G_h\)
gives
\[
\dH(g_{\widehat f_h},g_0)
=O_p\left\{h^{s+\beta}L_h^{a_A}
+\left(\frac{K_hL_h^a}{n}\right)^{1/2}\right\}.
\]
The BIC-type penalty is absorbed because its normalized value along the
approximating sequence is of the same or smaller squared order. This is the
standard sieve-MLE argument of \citet{wongshen1995}, with the finite-mixture
entropy calculation used in \citet{ghosalvdv2001}. Finally, every observed
density satisfies \(g_f\le\|f_\varepsilon\|_\infty\), so
\[
\|g_{\widehat f_h}-g_0\|_2^2
\le4\|f_\varepsilon\|_\infty
\dH^2(g_{\widehat f_h},g_0),
\]
which proves the \(L^2\) conclusion.
\end{proof}

\begin{lemma}[Direct inverse bound for the Gaussian sieve]
\label{lem:supp-local-direct-inverse}
Suppose Assumption~\ref{main-ass:local-os} holds and
\(f_0\in H^s(\mathbb R)\). For every
\(f\in\mathcal F_h^{\mathrm{reg}}\) and \(T\ge1\),
\[
\|f-f_0\|_2
\le CT^\beta\|g_f-g_0\|_2+CT^{-s}
+C h^{-1}T^{-1/2}\exp(-h^2T^2/2).
\]
\end{lemma}

\begin{proof}
Let \(u=f-f_0\). By Plancherel's identity and the convolution theorem,
\[
\int_{|t|\le T}|\varphi_u(t)|^2\,dt
\le C T^{2\beta}\|g_f-g_0\|_2^2.
\]
The Sobolev condition gives
\[
\left\{\int_{|t|>T}|\varphi_{f_0}(t)|^2\,dt\right\}^{1/2}
\le C T^{-s}.
\]
If \(f=\sum_j\pi_j\phiN(\cdot;\mu_j,\sigma_j^2)\) and
\(\sigma_j\ge h\), then \(|\varphi_f(t)|\le e^{-h^2t^2/2}\). Therefore
\[
\left\{\int_{|t|>T}|\varphi_f(t)|^2\,dt\right\}^{1/2}
\le C h^{-1}T^{-1/2}e^{-h^2T^2/2}.
\]
Combining the three frequency ranges proves the result.
\end{proof}

\begin{lemma}[Tail conversion for the local Gaussian sieve]
\label{lem:supp-local-l2-l1}
Suppose \(f_0\) has a subgaussian tail, \(M_h\le C_ML_h\), and the component
scales in \eqref{main-eq:local-regularized-sieve} are bounded above by
\(\bar\sigma\). If \(\|f-f_0\|_2=O_p(h^sL_h^c)\) for
\(f\in\mathcal F_h^{\mathrm{reg}}\), then
\(\|f-f_0\|_1=O_p(h^sL_h^{c+1/2})\).
\end{lemma}

\begin{proof}
Let \(A_h=M_h+d\sqrt{L_h}\), where \(d\) is fixed and sufficiently large.
On \([-A_h,A_h]\), Cauchy--Schwarz gives
\[
\int_{|x|\le A_h}|f(x)-f_0(x)|\,dx
\le(2A_h)^{1/2}\|f-f_0\|_2.
\]
Since \(A_h=O(L_h)\), this contributes at most one additional factor
\(L_h^{1/2}\). For every mixture in \(\mathcal F_h^{\mathrm{reg}}\),
\[
\int_{|x|>A_h}f(x)\,dx
\le2\left\{1-\Phi\left(d\sqrt{L_h}/\bar\sigma\right)\right\}
\le C h^{d^2/(2\bar\sigma^2)}.
\]
The subgaussian tail of \(f_0\) is smaller than every fixed power of \(h\)
at \(A_h\asymp L_h\). Choose \(d^2/(2\bar\sigma^2)>s\). The two tail
terms are then \(o(h^s)\), proving the result.
\end{proof}

\subsection{Verification under Laplace measurement error}
\label{app:laplace-ordinary-verification}

This subsection verifies the approximation, localized likelihood, and posterior
empirical-process conditions used in the Laplace specialization of the main
paper. The class \(\mathcal C_{\mathrm{sg}}\) is defined in
Equation~\ref{main-eq:concrete-class}, and the local sieve
\(\mathcal F_h^{\mathrm{reg}}\) is defined in
Equation~\ref{main-eq:local-regularized-sieve}.

\begin{proposition}[Verification of the approximation condition]
\label{prop:supp-concrete-laplace-ap}
Fix \(s>0\) and let \(f_0\in\mathcal C_{\mathrm{sg}}\). Suppose, for all
sufficiently small \(h\), that \(h\le\tau_-\), \(\bar\sigma\ge\tau_+\),
\(M_h\ge D_s\sqrt{L_h}\), \(K_h\ge C_sL_h\),
\(K_h\pi_{\min,h}\le1/2\), and
\(K_h\pi_{\min,h}\le C_\pi h^{2(s+2)}\). Then there exists
\(f_h^\circ\in\mathcal F_h^{\mathrm{reg}}\), with \(O(L_h)\) active
components, such that
\begin{equation}
\KL(g_0,g_h^\circ)+\V(g_0,g_h^\circ)
\le Ch^{2(s+2)}L_h.
\label{eq:concrete-ap}
\end{equation}
Thus Assumption~\ref{main-ass:local-ap} holds for
\eqref{main-eq:concrete-class} under Laplace error.
\end{proposition}

\begin{proof}
Write \(f_0=\phi_\tau*Q_0\), with
\(\tau\in[\tau_-,\tau_+]\), let \(U\sim Q_0\), and set
\(\psi_{\tau,b}=\phi_\tau*k_b\), so that
\(g_0=\psi_{\tau,b}*Q_0\). Put $a=2(s+2)$, and choose sufficiently large constants \(D_s\) and \(C_s\) such that
\[
D_s^2>\frac{a+3}{c_0},\qquad
\frac{D_s}{\tau_-}\sqrt{\frac e{C_s}}\leq\frac12,
\qquad
C_s\log 2>a+2.
\]
Then set $A_h=D_s\sqrt{L_h},$
 and $r_h=\lceil C_sL_h\rceil$, and let \(Q_h^A\) be the conditional law of \(U\) given
\(|U|\leq A_h\). The subgaussian tail bound gives
\[
\delta_h:=Q_0(|U|>A_h)=O(h^{a+3}),
\qquad
\|\psi_{\tau,b}*Q_0-\psi_{\tau,b}*Q_h^A\|_1
\leq 2\delta_h.
\]

By Carathéodory's theorem applied to the moment curve
\(u\mapsto(u,\ldots,u^{r_h-1})\), there is a probability measure
\(Q_h\), supported on at most \(r_h\) points of \([-A_h,A_h]\),
matching the first \(r_h-1\) moments of \(Q_h^A\). Taylor expansion
and moment cancellation, together with
\[
\|\psi_{\tau,b}^{(r)}\|_1
\leq \|\phi_\tau^{(r)}\|_1
\leq \tau^{-r}\sqrt{r!},
\]
yield
\[
\begin{aligned}
\|\psi_{\tau,b}*Q_h^A-\psi_{\tau,b}*Q_h\|_1
&\leq
2\frac{(A_h/\tau_-)^{r_h}}{\sqrt{r_h!}}  \\
&\leq
2\left(
\frac{D_s}{\tau_-}
\sqrt{\frac{eL_h}{r_h}}
\right)^{r_h}
=O(h^{a+2}).
\end{aligned}
\]

Remove from \(Q_h\) all atoms having weight smaller than
\(\pi_{\min,h}\) and add their total mass to any retained atom.
Since
$r_h\pi_{\min,h}
\leq K_h\pi_{\min,h}\leq 1/2$,
at least one atom is retained. The resulting probability measure
\(\widetilde Q_h\) is supported on a subset of the same atoms, and
all its nonzero weights are at least \(\pi_{\min,h}\). Moreover,
\[
|Q_h-\widetilde Q_h|(\mathbb R)
\leq 2r_h\pi_{\min,h}.
\]
Hence, by the \(L^1\)-contractivity of convolution,
\[
\|\psi_{\tau,b}*Q_h-\psi_{\tau,b}*\widetilde Q_h\|_1
\leq
|Q_h-\widetilde Q_h|(\mathbb R)
\leq
2r_h\pi_{\min,h}
\leq
2K_h\pi_{\min,h}
=
O(h^a).
\]
Define
\[
f_h^\circ=\phi_\tau*\widetilde Q_h,
\qquad
g_h^\circ=\psi_{\tau,b}*\widetilde Q_h.
\]
Its locations lie in
\([-A_h,A_h]\subset[-M_h,M_h]\), its common scale satisfies
\(h\leq\tau\leq\bar\sigma\), it has at most \(r_h\le K_h\) active components, with
\(r_h=O(L_h)\), and all weights are at
least \(\pi_{\min,h}\). For all sufficiently small $h$, the assumptions imply that
$A_h \leq M_h$, $r_h \leq K_h$, and $h \leq \tau \leq \bar{\sigma}$, therefore  \(f_h^\circ\in\mathcal F_h^{\mathrm{reg}}\).
The preceding bounds imply
\[
\|g_0-g_h^\circ\|_1\leq Ch^a,
\qquad
H^2(g_0,g_h^\circ)\leq Ch^a.
\]

It remains to control the likelihood ratio. The subgaussian
assumption implies \(Ee^{|U|/b}<\infty\). Therefore, by the triangle
inequality and uniformly in \(\tau\in[\tau_-,\tau_+]\),
\[
g_0(y)\leq C_1e^{-|y|/b},
\qquad
\inf_{|u|\leq A_h}\psi_{\tau,b}(y-u)
\geq C_2e^{-A_h/b}e^{-|y|/b}.
\]
Averaging the second inequality with respect to \(\widetilde Q_h\)
gives
\[
\frac{g_h^\circ(y)}{g_0(y)}
\geq
\lambda_h:=
\min\{1,C_3e^{-A_h/b}\},
\qquad y\in\mathbb R,
\]
and hence
\[
\log(1/\lambda_h)=O(A_h)=O(\sqrt{L_h}).
\]

Finally, the standard Hellinger--likelihood-ratio comparison
\[
\KL(p,q)+V(p,q)
\leq
C\{1+\log^2(1/\lambda)\}H^2(p,q),
\qquad q/p\geq\lambda,
\]
applied with \(p=g_0\), \(q=g_h^\circ\), and
\(\lambda=\lambda_h\), yields
\[
\KL(g_0,g_h^\circ)+V(g_0,g_h^\circ)
\leq
Ch^aL_h
=
Ch^{2(s+2)}L_h.
\]
Thus Assumption~\ref{main-ass:local-ap} holds uniformly over the class $\mathcal{C}_{\mathrm{sg}}$ defined in~\eqref{main-eq:concrete-class}, with $\beta=2$ and $a_A=1/2$.
\end{proof}

A compatible calibration is, for example, \(K_h\asymp h^{-1}L_h^{a_K}\), \(M_h=C_ML_h\), 
and \(\pi_{\min,h}=h^{2s+6}\), with \(\bar\sigma\ge\tau_+\). Indeed, \(K_h\pi_{\min,h}\asymp 
h^{2s+5}L_h^{a_K}=o\{h^{2(s+2)}\}\), so both the simplex-feasibility condition and the weight-floor 
requirement hold for sufficiently small \(h\). This example makes explicit that the displayed 
calibration depends on \(s\).

\begin{proposition}[Localized likelihood complexity under Laplace error]
\label{prop:supp-concrete-laplace-ll}
Suppose the measurement error is Laplace with scale \(b>0\),
\(M_h\le C_ML_h\), and \(h\le\sigma_j\le\bar\sigma\). For the observed
sieve \(\mathcal G_h=\{g_f:f\in\mathcal F_h^{\mathrm{reg}}\}\), there is a
constant \(C<\infty\) such that, for every \(0<u\le1/2\),
\begin{equation}
H_{[]}(u,\mathcal G_h,d_{\mathrm H})
\le
C K_h\log\left\{
\frac{C\{M_h+\log(e/u)\}}{hu}
\right\}.
\label{eq:concrete-laplace-hellinger-entropy}
\end{equation}
Let \(h=h_n\) be a polynomial resolution, in the sense that
\(n^{-\rho_+}\le h_n\le n^{-\rho_-}\) for fixed
\(0<\rho_-\le\rho_+<\infty\). If \(K_hL_h=o(n)\), then the entropy-integral
condition in Assumption~\ref{main-ass:local-ll} holds with
\(r_{n,h}\asymp\{K_hL_h/n\}^{1/2}\), hence with \(a_E=1\).

Using the approximant in Proposition~\ref{prop:supp-concrete-laplace-ap} and the iterated-log BIC penalty in \eqref{main-eq:bic}, the
penalty condition in Assumption~\ref{main-ass:local-ll} also holds for a finite
\(a_P\). In particular, both localized-likelihood requirements hold at
\(h_n=n^{-1/(2s+5)}\), up to fixed logarithmic modifications.
\end{proposition}

\begin{proof}
Write
\(p_{\mu,\sigma}=\phiN(\cdot;\mu,\sigma^2)*k_b\). Since convolution is a
positive linear operator and \(\|k_b\|_\infty=(2b)^{-1}\), differentiation
under the integral gives, uniformly over
\(|\mu|\le M_h\) and \(h\le\sigma\le\bar\sigma\),
\[
 \|\partial_\mu p_{\mu,\sigma}\|_\infty
 \le \|\partial_\mu\phiN(\cdot;\mu,\sigma^2)\|_1\|k_b\|_\infty
 \le C h^{-1},
 \qquad
 \|\partial_\sigma p_{\mu,\sigma}\|_\infty\le C h^{-1}.
\]
The second inequality uses
\(\|\partial_\sigma\phiN(\cdot;\mu,\sigma^2)\|_1
=\sigma^{-1}E|Z^2-1|\). Hence, for two labeled order-\(K\) mixtures,
\[
 \|g_{\pi,\mu,\sigma}-g_{\pi',\mu',\sigma'}\|_\infty
 \le C\left\{\|\pi-\pi'\|_1
 +h^{-1}\max_j|\mu_j-\mu'_j|
 +h^{-1}\max_j|\sigma_j-\sigma'_j|\right\}.
\]
A simplex net for the weights and Cartesian nets for the locations and scales
therefore give, for \(0<\eta<1\),
\[
 \log N\{\eta,\mathcal G_{K,h},\|\cdot\|_\infty\}
 \le CK\log\left(\frac{CM_h}{h^2\eta^3}\right),
\]
where enlarging the right-hand side covers bounded parameter ranges and the
positive-floor subclass. Taking the union over \(K\le K_h\) changes only the
constant.

The Laplace tail converts this uniform cover into Hellinger brackets without a
sieve-wide density floor. Indeed, for a constant depending only on
\(b\) and \(\bar\sigma\), every component and every mixture in the observed
sieve satisfy
\[
 g_f(y)\le R_h(y):=C\exp\{-c(|y|-M_h)_+\}.
\]
To see this, write a component as
\(E\{k_b(y-\mu-\sigma Z)\}\), use
\(e^{-|y-\mu-\sigma Z|/b}\le
 e^{-(|y|-M_h)/b}e^{\bar\sigma|Z|/b}\) outside \([-M_h,M_h]\), and use the
boundedness of \(k_b\) inside that interval.

Fix \(0<u\le1/2\), put
\(B=M_h+C\log(e/u)\), and choose
\(\eta=c_1u^2/(B+1)\), with \(c_1>0\) sufficiently small. For every center
\(g_j\) of an \(\eta\)-cover, define
\(\ell_j=(g_j-\eta)_+\) and
\(v_j=\min(g_j+\eta,R_h)\). If
\(\|g-g_j\|_\infty\le\eta\), then
\(\ell_j\le g\le v_j\). Moreover,
\[
 \int(v_j-\ell_j)
 \le 4B\eta+2\int_{|y|>B}R_h(y)\,dy
 \le u^2.
\]
Because \((\sqrt a-\sqrt b)^2\le|a-b|\) for nonnegative \(a,b\), each such
bracket has Hellinger size at most \(u\). Substitution of the chosen \(B\) and
\(\eta\) into the preceding covering bound yields~\eqref{eq:concrete-laplace-hellinger-entropy}.

It remains to verify the integral in Assumption~\ref{main-ass:local-ll}. For
\(u\in[c_0\epsilon^2,c_1\epsilon]\), monotonicity of the entropy and
\eqref{eq:concrete-laplace-hellinger-entropy} give
\[
 \int_{c_0\epsilon^2}^{c_1\epsilon}
 \{1+H_{[]}(u,\mathcal G_h,d_{\mathrm H})\}^{1/2}du
 \le C\epsilon
 \left[K_h\log\left\{
 \frac{C\{M_h+\log(e/\epsilon)\}}{h\epsilon}
 \right\}\right]^{1/2}.
\]
Under the polynomial-resolution condition, \(L_h\asymp\log n\). Also,
\(\epsilon\ge\{K_hL_h/n\}^{1/2}\ge n^{-1/2}\) for all sufficiently large
\(n\), and \(M_h=O(L_h)\). The logarithm factor on the right-hand side is therefore
\(O(L_h)\), so the integral is bounded by
\(C\epsilon(K_hL_h)^{1/2}\), which is at most
\(C\sqrt n\,\epsilon^2\) whenever
\(\epsilon\ge\{K_hL_h/n\}^{1/2}\). This proves the entropy-integral part of
Assumption~\ref{main-ass:local-ll} with \(a_E=1\).

Finally, Proposition~\ref{prop:supp-concrete-laplace-ap} supplies an approximant
with \(K_h^\circ=O(L_h)\) active components. For the iterated-log BIC penalty,
\(\operatorname{pen}_n(f_h^\circ)/n
=O\{L_h\mathsf L^{\circ3}(n)\log n/n\}\). Since
\(L_h\asymp\log n\), this is bounded by
\(CK_hL_h^{a_P}/n\) for some finite \(a_P\), for example \(a_P=3\).
At \(h_n=n^{-1/(2s+5)}\), one has \(L_{h_n}\asymp\log n\) and
\(K_{h_n}L_{h_n}/n\to0\), completing the verification.
\end{proof}

\begin{proposition}[Posterior empirical complexity under Laplace error]
\label{prop:supp-concrete-laplace-pe}
Under Laplace measurement error, \(M_h\le CL_h\),
\(h\le\sigma_j\le\bar\sigma\), and \(K_hL_h=o(n)\),
\begin{equation}
\sup_{f\in\mathcal F_h^{\mathrm{reg}}}
\|(\Pn-\Pzero)q_f\|_1
=O_p\{n^{-1/2}h^{-3/2}L_h^{C_Q}\}.
\label{eq:concrete-pe}
\end{equation}
Hence Assumption~\ref{main-ass:local-pe} holds with \(\gamma=3/2\).
\end{proposition}

\begin{proof}
All expectations below may be read as outer expectations; alternatively fix a
countable dense parameter subset, which has the same supremum by continuity.
For fixed $x$, normalize the class by $E_h(x)$. Symmetrization and the Dudley
entropy integral for the conditional Rademacher process, together with
Lemma~\ref{lem:laplace-posterior-covering}, give
\[
 E^*\sup_{f\in\mathcal F_h^{\rm reg}}
 |(\Pn-\Pzero)q_f(x\mid\cdot)|
 \le CE_h(x)\sqrt{\frac{K_h\log A_h(x)}{n}}.
\]
Now use
\[
 \sup_f\|(\Pn-\Pzero)q_f\|_1
 \le\int\sup_f|(\Pn-\Pzero)q_f(x\mid\cdot)|dx,
\]
Tonelli's theorem in outer-expectation form, and
Lemma~\ref{lem:laplace-posterior-envelope}. This yields
\[
 E^*\sup_f\|(\Pn-\Pzero)q_f\|_1
 \le C\frac{M_h+1}{h}
       \sqrt{\frac{K_hL_h}{n}}.
\]
Under $M_h=O(L_h)$ and
$K_h\le Ch^{-1}L_h^{a_K}$, Markov's inequality gives~\eqref{eq:concrete-pe}. The condition $K_hL_h=o(n)$ makes the empirical
bound asymptotically meaningful.
\end{proof}

At the resolution \(h_n=n^{-1/(2s+5)}\), the right-hand side of \eqref{eq:concrete-pe} is \(h_n^{s+1}L_{h_n}^{C_Q}\), and therefore tends to zero for every fixed \(s>0\). More generally, 
consistency through Proposition~\ref{prop:supp-concrete-laplace-pe} requires \(n^{-1/2}h^{-3/2}L_h^{C_Q}
\to0\).

\subsection{Posterior empirical stability for the Gaussian-mixture sieve}
\label{app:posterior-stability-gaussian}

This subsection verifies Assumption~\ref{main-ass:primitive-posterior-stability} for the regularized Gaussian-mixture sieve, with an explicit rate
\(\tau_n\). 

\begin{lemma}[Joint Lipschitz bound for the posterior kernel, uniform in \(x\)]
\label{lem:posterior-kernel-lipschitz}
There exists a finite constant \(L_n^Q\), depending only on
\(K_n,\sigma_{\min,n},\sigma_{\max,n},m_n^{\mathrm{sieve}},
\|f_\varepsilon\|_\infty,L_\varepsilon\), such that
for every \(\theta,\widetilde\theta\in\Theta_{K,n}^{\mathrm{reg}}\), every
\(|x|\le A_n\), and every \(y\in\mathcal Y_n\),
\[
    |q_{f_\theta}(x,y)-q_{f_{\widetilde\theta}}(x,y)|
    \le
    L_n^Q\|\theta-\widetilde\theta\|_\infty,
    \qquad
    L_n^Q
    \asymp
    \frac{K_n}{\sigma_{\min,n}^2m_n^{\mathrm{sieve}}}
    +
    \frac{K_n(\|f_\varepsilon\|_\infty+L_\varepsilon)}
    {\sigma_{\min,n}(m_n^{\mathrm{sieve}})^2}.
\]
\end{lemma}

\begin{proof}
Write $
q_{f_\theta}(x,y)
=
\frac{f_\theta(x)f_\varepsilon(y-x)}{g_\theta(y)}$.
On \(\mathcal Y_n\), differentiation gives
\[
D_\theta q_{f_\theta}(x,y)
=
D_\theta f_\theta(x)\frac{f_\varepsilon(y-x)}{g_\theta(y)}
-q_{f_\theta}(x,y)\frac{D_\theta g_\theta(y)}{g_\theta(y)}.
\]
For the direct-weight, location, and scale coordinates of a Gaussian mixture,
\[
\sup_x\|D_\theta f_\theta(x)\|_1
\le
\frac{C K_n}{\sigma_{\min,n}^2}.
\]
Moreover, writing each convolved component as
\(E\{f_\varepsilon(y-\mu_k-\sigma_k Z)\}\), the boundedness and global
Lipschitz property of \(f_\varepsilon\) imply
\[
\sup_y\|D_\theta g_\theta(y)\|_1
\le
C K_n\{\|f_\varepsilon\|_\infty+L_\varepsilon\}.
\]
The scale derivative follows from a difference quotient and
\(E|Z|<\infty\); no likelihood-entropy lemma is used. On
\(|x|\le A_n\), \(y\in\mathcal Y_n\),
\[
g_\theta(y)\ge m_n^{\mathrm{sieve}},
\qquad
q_{f_\theta}(x,y)
\le
\frac{\|f_\varepsilon\|_\infty}
{\sigma_{\min,n}\sqrt{2\pi}\,m_n^{\mathrm{sieve}}}.
\]
Consequently,
\[
\sup_{|x|\le A_n,\,y\in\mathcal Y_n}
\|D_\theta q_{f_\theta}(x,y)\|_1
\le
C\left\{
\frac{K_n}{\sigma_{\min,n}^2m_n^{\mathrm{sieve}}}
+
\frac{K_n(\|f_\varepsilon\|_\infty+L_\varepsilon)}
{\sigma_{\min,n}(m_n^{\mathrm{sieve}})^2}
\right\}
=:L_n^Q.
\]
The parameter space is convex in the direct weights, locations, and scales.
Integrating this derivative bound along the segment joining \(\theta\) and
\(\widetilde\theta\) proves the lemma.
\end{proof}

\begin{remark}
The bound
\(\sup_x |\partial_{\mu_k}f(x)|\vee\sup_x |\partial_{\sigma_k}f(x)|
\le C/\sigma_{\min,n}^2\) used in the proof carries no \(A_n\)-dependence: for
a Gaussian density the map \(u\mapsto u\phi(u;0,\sigma^2)/\sigma^2\) is
maximized at \(|u|=O(\sigma)\), not at the boundary of the observation window.
This derivative bound is uniform in \(x\), so the entropy calculation below does not require a separate covering argument in the \(x\)-direction.
\end{remark}

\begin{lemma}[Bracketing entropy of the truncated posterior-kernel class, uniform in \(x\)]
\label{lem:posterior-kernel-bracketing}
Let
\(
\mathcal Q_n(x)=\{q_f(x\mid\cdot)\mathbf 1_{\mathcal Y_n}:f\in\mathcal F_n\}.
\)
There exist \(C_E>0\) and a sequence \(\widetilde L_n^Q\ge1\), absorbing
polynomial factors in
\(K_n,M_n,\sigma_{\max,n},\pi_{\min,n}^{-1},\sigma_{\min,n}^{-1}\) and
\(L_n^Q\), such that for every \(0<\rho<1\), uniformly in \(|x|\le A_n\),
\[
    \log N_{[]}\{\rho,\mathcal Q_n(x),L^2(\Pzero)\}
    \le
    C_E V_n\log\left(\frac{\widetilde L_n^Q}{\rho}\right).
\]
\end{lemma}

\begin{proof}
For each \(K\le K_n\), cover \(\Theta_{K,n}^{\mathrm{reg}}\) by a
sup-norm grid of mesh \(\delta\). A direct coordinate count gives
\[
\log N_K(\delta)
\le
C K\log\left\{
\frac{C K M_n\sigma_{\max,n}}
{\pi_{\min,n}\sigma_{\min,n}\delta}
\right\}.
\]
By Lemma~\ref{lem:posterior-kernel-lipschitz}, two parameters in the same
grid cell produce truncated kernels whose pointwise distance on
\(\mathcal Y_n\) is at most \(L_n^Q\delta\), uniformly in
\(|x|\le A_n\). Use brackets centered at the grid kernels, clipped below at
zero, with half-width \(L_n^Q\delta\mathbf 1_{\mathcal Y_n}\). Their
\(L^2(\Pzero)\)-width is at most \(2L_n^Q\delta\). Choosing
\(\delta=\rho/(2L_n^Q)\), summing over \(K\le K_n\), and absorbing the
parameter ranges and \(L_n^Q\) into \(\widetilde L_n^Q\) proves the stated
entropy bound. The finite-dimensional classes are separable, so the same
countable grids also settle measurability; alternatively all suprema may be
read as outer suprema.
\end{proof}

\begin{proposition}[Explicit rate for posterior empirical stability]
\label{prop:posterior-stability-gaussian}
Let \(\mathcal Y_n=[-B_n,B_n]\), and let
\(m_n^{\mathrm{sieve}}>0\) be a lower bound for \(g_f\) on
\(\mathcal Y_n\), uniform over \(f\in\mathcal F_n\), as supplied by
Lemma~\ref{lem:explicit-mn} and its corollaries. Assume
\(\Pzero(Y\notin\mathcal Y_n)=o(1)\). Suppose that the tail multiplier
\(c_A\) in \(A_n=(1+c_A)M_n\) satisfies
\[
\eta_n/m_n^{\mathrm{sieve}}\longrightarrow0,
\]
where
\(\eta_n=1-\Phi(c_AM_n/\sigma_{\max,n})\). Let
\[
F_n^Q
=
\frac{\|f_\varepsilon\|_\infty}
{\sigma_{\min,n}\sqrt{2\pi}\,m_n^{\mathrm{sieve}}}.
\]
Write
$
\Gamma_n^Q=V_n\log\widetilde L_n^Q.$
Then
\[
\sup_{f\in\mathcal F_n}
\|\Pn q_f-\Pzero q_f\|_1
=O_p(\tau_n),
\]
where
\[
\tau_n
=
A_nF_n^Q\left\{
\sqrt{\frac{\Gamma_n^Q}{n}}
+\frac{\Gamma_n^Q}{n}
\right\}
+\frac{\eta_n}{m_n^{\mathrm{sieve}}}
+\Pzero(Y\notin\mathcal Y_n)
+n^{-1/2}.
\]
Thus Assumption~\ref{main-ass:primitive-posterior-stability} holds whenever
\(\tau_n\to0\).
\end{proposition}

\begin{proof}
\phantomsection\label{proof:prop-posterior-stability-gaussian}

For \(f\in\mathcal F_n\), write
$
(\Pn-\Pzero)q_f
=(\Pn-\Pzero)\{q_f\mathbf 1_{\mathcal Y_n}\}
+(\Pn-\Pzero)\{q_f\mathbf 1_{\mathcal Y_n^c}\}$.
Since \(q_f(\cdot,y)\) is a density,
\[
\sup_f
\|(\Pn-\Pzero)\{q_f\mathbf 1_{\mathcal Y_n^c}\}\|_1
\le
\Pn(\mathcal Y_n^c)+\Pzero(\mathcal Y_n^c)
=2\Pzero(\mathcal Y_n^c)+O_p(n^{-1/2}).
\]
For \(y\in\mathcal Y_n\), Lemmas~\ref{lem:explicit-mn} and~\ref{lem:gaussian-sieve-budgets} give
\[
\sup_f\int_{|x|>A_n}q_f(x\mid y)\,dx
\le
\frac{2\|f_\varepsilon\|_\infty\eta_n}
{m_n^{\mathrm{sieve}}}.
\]
Thus the latent-tail contribution from the truncated observed set is
\(O\{\eta_n/m_n^{\mathrm{sieve}}\}\).
For \(|x|\le A_n\), define the correctly truncated process
\[
Z_n(x)
=
\sup_{f\in\mathcal F_n}
\left|
(\Pn-\Pzero)
\{q_f(x\mid\cdot)\mathbf 1_{\mathcal Y_n}\}
\right|.
\]
Lemma~\ref{lem:posterior-kernel-bracketing} 
and the bracketing maximal
inequality, including its square-root and linear terms, yield uniformly in
\(|x|\le A_n\),
\[
E^*Z_n(x)
\le
C F_n^Q
\left\{
\sqrt{\frac{\Gamma_n^Q}{n}}
+\frac{\Gamma_n^Q}{n}
\right\}.
\]
Tonelli's theorem for outer expectation and Markov's inequality give
\[
\int_{-A_n}^{A_n}Z_n(x)\,dx
=
O_p\left[
A_nF_n^Q
\left\{
\sqrt{\frac{\Gamma_n^Q}{n}}
+\frac{\Gamma_n^Q}{n}
\right\}
\right].
\]
Combining the observed-tail, latent-tail, and central contributions proves
the rate \(\tau_n\).
\end{proof}

\begin{remark}
Proposition~\ref{prop:gaussian-global-likelihood} controls the full Gaussian log-likelihood class.
Corollary~\ref{cor:supp-generic-supersmooth-transfer} is an inverse statement
conditional on an observed-domain rate, a Sobolev radius, and latent-tail
control. The result in
Section~\ref{main-sec:gaussian-smoothed-end-to-end} fixes a positive Gaussian scale floor, proves a Hoelder
inverse-stability inequality on the resulting analytic class, and verifies posterior empirical complexity
through the Gaussian conditional representation. Proposition~\ref{prop:posterior-stability-gaussian}
provides a sufficient condition for other sieve regimes.
\end{remark}

\subsection{Observed-tail and fitted-denominator calibration}
\label{app:fitted-denominator-material}

The preceding sieve-uniform lower bounds are used for the first-stage
bracketing-entropy calculation. Posterior reconstruction instead involves only
the denominators along the fitted sequence, \(g_{\widehat f}(Y_i)\). The
following calculation records the corresponding fitted-sequence calibration. Let
\[
    m_0(B)
    =
    \inf_{|y|\le B} g_{f_0}(y).
\]
Along the chosen sequence \(B_n\to\infty\), suppose that \(m_0(B_n)>0\) and that
either
\[
    m_0(B_n)
    \ge
    c_0\exp(-c_1B_n^\kappa)
\]
for some \(c_0,c_1,\kappa>0\), or
\[
    m_0(B_n)
    \ge
    c_0B_n^{-\rho}
\]
for some \(c_0,\rho>0\). Assume also that the truncation radii satisfy
\[
    \Pzero(|Y|>B_n)=o(1).
\]

By Lemma~\ref{lem:supp-kl-to-l2},
\[
    \sup_{|y|\le B_n}
    |g_{\widehat f}(y)-g_{f_0}(y)|
    =
    O_p(e_n^{1/2}).
\]
Hence, if $e_n=o\{m_0(B_n)^2\},
$ then
\[
    P\left(
    \inf_{|y|\le B_n}g_{\widehat f}(y)
    \ge
    \frac12 m_0(B_n)
    \right)
    \to1 .
\]
The same conclusion holds for the EM implementation with \(e_n\) replaced by
\(e_n^{\mathrm{EM}}\).

Table~\ref{tab:tail-verification} records the corresponding lower-tail regimes
for the observed density
\(g_{f_0}=f_0*f_\varepsilon\) in the Monte Carlo designs.
\begin{table}[ht]
\centering
\caption{Lower-tail regimes for the observed densities in the simulation. The 
reported exponent describes the behavior of
\(m_0(B)=\inf_{|y|\le B}g_{f_0}(y)\) on \([-B,B]\).}
\label{tab:tail-verification}
\renewcommand{\arraystretch}{1.12}
\begin{tabular}{p{0.38\textwidth}p{0.26\textwidth}p{0.26\textwidth}}
\toprule
Latent design & Gaussian error & Laplace error \\
\midrule
Standard normal
& exponential regime, \(\kappa=2\)
& exponential regime, \(\kappa=1\) \\

Student-\(t_5\)
& polynomial regime, \(\rho=6\)
& polynomial regime, \(\rho=6\) \\

Laplace
& exponential regime, \(\kappa=1\)
& exponential regime, \(\kappa=1\) \\

Cauchy& polynomial regime, \(\rho=2\)
& polynomial regime, \(\rho=2\) \\

Nearly normal
& exponential regime, \(\kappa=2\)
& exponential regime, \(\kappa=1\) \\

Gamma
& exponential regime, \(\kappa=2\)
& exponential regime, \(\kappa=1\) \\

Comte chi-square
& exponential regime, \(\kappa=2\)
& exponential regime, \(\kappa=1\) \\

Cai chi-square
& exponential regime, \(\kappa=2\)
& exponential regime, \(\kappa=1\) \\

Beta
& exponential regime, \(\kappa=2\)
& exponential regime, \(\kappa=1\) \\

Comte Mixed gamma & exponential regime, \(\kappa=2\)
& exponential regime, \(\kappa=1\) \\

Cai Mixed gamma & exponential regime, \(\kappa=2\)
& exponential regime, \(\kappa=1\) \\

Mixed normal G
& exponential regime, \(\kappa=2\)
& exponential regime, \(\kappa=1\) \\

Mixed normal Y
& exponential regime, \(\kappa=2\)
& exponential regime, \(\kappa=1\) \\
\bottomrule
\end{tabular}
\end{table}
With \(B_n\asymp M_n\asymp(\log n)^m\), the polynomial regimes impose no
further restriction on \(m\) when \(e_n\) decays polynomially in \(n\). In the
exponential regimes, cases with \(\kappa=1\) require \(m<1\), whereas the
Gaussian-error cases with \(\kappa=2\) require \(m<1/2\).

\subsection{Two routes to posterior empirical stability}
\label{app:post-stability-calibration}
Posterior empirical stability is logically separate from the likelihood empirical-process control under Gaussian error. For a general Gaussian-mixture sieve,
Appendix~\ref{app:posterior-stability-gaussian} gives a truncated sufficient
route based on an observed window \(\mathcal Y_n\), a sieve-uniform denominator
bound \(m_n^{\mathrm{sieve}}\), latent-tail control, and the explicit rate in
Proposition~\ref{prop:posterior-stability-gaussian}. Rate preservation then
requires the additional comparison \(\tau_n=O(r_n)\), which must be checked for
the calibration under consideration.
The fixed-scale Gaussian-smoothed branch uses a different argument.
Gaussian conjugacy, the positive component-scale floor, and
Lemma~\ref{lem:supp-gaussian-posterior-lipschitz} yield the global bound in
Proposition~\ref{prop:supp-gaussian-smoothed-posterior-complexity}, without a
truncated denominator floor. That proposition is the posterior input used in
Theorem~\ref{main-thm:explicit-sieve-rates}. Neither route is a consequence of the
likelihood result alone, and the two posterior verifications should not be
combined into a single calibration claim.

\subsection{BIC order selection under measurement error}
\label{app:fixed-order-selection}

This subsection justifies BIC-type model-order selection under measurement error. Although the latent variables are unobserved, each latent mixture model induces, after convolution with the known error density, a finite-mixture model for the contaminated observations. The result is restricted to a fixed candidate range
\(
    1\leq K\leq\overline K<\infty,
\)
where \(\overline K\) does not depend on \(n\).

Let \(h_\theta\), \(\theta\in\mathcal T\), be a latent component
density and let
\(
    p_\theta
    =
    h_\theta*f_\varepsilon
\)
be the corresponding observed component density. For each \(K\), define
the observed mixture class
\[
    \mathcal G_K
    =
    \left\{
        \sum_{k=1}^K \pi_k p_{\theta_k}:
        \pi_k\geq0,\quad
        \sum_{k=1}^K\pi_k=1,\quad
        \theta_k\in\mathcal T
    \right\}.
\]
Thus, the mixture weights range over the closed simplex. For each \(K\), the global within-order likelihood maximizer is
\[
\widetilde g_K
\in
\operatorname*{arg\,max}_{g\in\mathcal G_K}
\sum_{i=1}^n\log g(Y_i).
\]

\begin{proposition}[Fixed-order consistency for the observed convolution family]
\label{prop:fixed-order-selection}
Suppose that, for some fixed
\(1\leq K_0\leq\overline K<\infty\),
\(
    g_0
    \in
    \mathcal G_{K_0}
    \setminus
    \mathcal G_{K_0-1},
\)
and that the observed component family
\(\{p_\theta:\theta\in\mathcal T\}\) satisfies
Assumptions A2--A5 of \citet{nguyennguyen2026}.
Suppose also that \(d_K\) is strictly increasing in \(K\).

Define
\(
    \widehat K_\nu
    =
    \min\operatorname*{arg\,min}_{1\leq K\leq\overline K}
    \left\{
        -2\sum_{i=1}^n\log\widetilde g_K(Y_i)
        +
        d_K\mathsf L^{\circ3}(n)\log n
    \right\}.
\)
Then
\[
    \Pr(\widehat K_\nu=K_0)
    \longrightarrow
    1.
\]
\end{proposition}

\begin{proof}
Apply Corollary~1 of \citet{nguyennguyen2026} to the observed
component family
\(\{p_\theta:\theta\in\mathcal T\}\), where
\(p_\theta=h_\theta*f_\varepsilon\), and to the associated mixture
classes \(\{\mathcal G_K:1\leq K\leq\overline K\}\).
The iid sampling condition and
\(g_0\in\mathcal G_{K_0}\setminus\mathcal G_{K_0-1}\)
give Assumption A1 of that result, while Assumptions A2--A5 are imposed
directly on the observed component family.
On the average negative-log-likelihood scale, take
\[
    \operatorname{pen}_{K,n}^{\nu}
    =
    \frac{d_K}{2n}\,
    \mathsf L^{\circ3}(n)\log n.
\]
This is the \(\nu\)-BIC penalty with \(\nu=3\) and
\(\alpha(K)=d_K/2\). Since \(d_K\) is strictly increasing,
\(\alpha(K)\) satisfies the monotonicity condition required by
Corollary~1 of \citet{nguyennguyen2026}. Hence,
\[
    \Pr(\widehat K_\nu=K_0)\longrightarrow1.
\]

Multiplication of the criterion by \(2n\) does not change its minimizer
and gives
\[
    -2\sum_{i=1}^n\log\widetilde g_K(Y_i)
    +
    d_K\mathsf L^{\circ3}(n)\log n,
\]
which is the criterion used in the statement.
\end{proof}

\begin{corollary}[Known Gaussian measurement error]
\label{cor:fixed-order-gaussian}
Suppose that the latent component family is Gaussian,
\(
    h_{\mu,\sigma}(x)
    =
    \phiN(x;\mu,\sigma^2),
\)
with
\(
    (\mu,\sigma)
    \in
    \mathcal T
    =
    [-M,M]\times
    [\underline\sigma,\overline\sigma],
    \) and \(
    0<\underline\sigma<\overline\sigma<\infty.
\)
Suppose that
\(\varepsilon\sim N(0,\sigma_\varepsilon^2)\), with
\(\sigma_\varepsilon^2\) known, and that the true latent density is a
finite Gaussian mixture of minimal order \(K_0\), with all component
parameters in \(\mathcal T\). Let
\(1\leq K\leq\overline K<\infty\), with \(\overline K\) fixed, and use
\(d_K=3K-1\).
Then the assumptions of
Proposition~\ref{prop:fixed-order-selection} hold and
\[
    \Pr(\widehat K_\nu=K_0)
    \longrightarrow
    1.
\]
\end{corollary}
\begin{proof}
The observed component family is again Gaussian
location--scale, with parameter transformation
\((\mu,\sigma)\mapsto(\mu,\sigma^2+\sigma_\varepsilon^2)\), because for every \((\mu,\sigma)\in\mathcal T\),
$N(\mu,\sigma^2)*N(0,\sigma_\varepsilon^2)=   N\!\left(\mu,\sigma^2+\sigma_\varepsilon^2\right)$.
This transformation is one-to-one on the compact set
\(\mathcal T=[-M,M]\times[\underline\sigma,\overline\sigma]\), and its
image is compact, with variances bounded away from zero and infinity.
Consequently, the positivity, continuity, log-envelope, and
\(L^1\)-Lipschitz conditions follow from the Gaussian-family
verification in \citet{nguyennguyen2026}.

It remains to verify preservation of the minimal order. If \(g_0\)
admitted an observed Gaussian-mixture representation with fewer than
\(K_0\) components, the one-to-one variance transformation and
identifiability of finite Gaussian mixtures would yield a latent
Gaussian-mixture representation of \(f_0\) with fewer than \(K_0\)
components. This contradicts the assumed minimality of \(K_0\).
Therefore,
\[
    g_0\in\mathcal G_{K_0}\setminus\mathcal G_{K_0-1}.
\]
Since \(d_K=3K-1\) is strictly increasing, the conclusion follows from
Proposition~\ref{prop:fixed-order-selection}.
\end{proof}

\begin{corollary}[Known Laplace measurement error]
\label{cor:fixed-order-laplace}
Suppose that the latent component family is Gaussian,
\(
    h_{\mu,\sigma}(x)
    =
    \phiN(x;\mu,\sigma^2),
\)
with
\(
    (\mu,\sigma)
    \in
    \mathcal T
    =
    [-M,M]\times
    [\underline\sigma,\overline\sigma],
    \) and
    \(0<\underline\sigma<\overline\sigma<\infty.
\)
Suppose that the measurement error has the known Laplace density
\(
    f_\varepsilon(u)
    =
    1/(2b)\exp\left(-{|u|}/{b}\right),
\) with \( b>0,\)
and that the true latent density is a finite Gaussian mixture of
minimal order \(K_0\), with all component parameters in
\(\mathcal T\). Let
\(1\leq K\leq\overline K<\infty\), with \(\overline K\) fixed, and use
\(d_K=3K-1\).
Then the assumptions of
Proposition~\ref{prop:fixed-order-selection} hold and
\[
    \Pr(\widehat K_\nu=K_0)
    \longrightarrow
    1.
\]
\end{corollary}

\begin{proof}
Write \(\theta=(\mu,\sigma)\in\mathcal T\), with
\(\mathcal T=[-M,M]\times[\underline\sigma,\overline\sigma]\), and let
\(\phi_\theta\) denote the \(N(\mu,\sigma^2)\) density. The observed
component density is
\(p_\theta(y)=\int\phi_\theta(x)f_\varepsilon(y-x)\,dx\).

We verify Assumptions A1--A5 of \citet{nguyennguyen2026}.

\emph{A1.}
The characteristic function of the Laplace error is $\varphi_\varepsilon(t)=
    {1+b^2t^2}^{-1}$, $t\in\mathbb R$,
and is everywhere nonzero. Hence convolution with
\(f_\varepsilon\) is injective. In particular, if
\(f_0*f_\varepsilon=f_L*f_\varepsilon\) for some Gaussian mixture
\(f_L\), equality of characteristic functions implies \(f_0=f_L\).
Thus an observed representation with fewer than \(K_0\) components
would contradict the minimality of the latent order, and therefore
\(g_0\in\mathcal G_{K_0}\setminus\mathcal G_{K_0-1}\).

\emph{A2--A3.}
The parameter space \(\mathcal T\) is compact. Since both
\(\phi_\theta\) and \(f_\varepsilon\) are strictly positive,
\(p_\theta(y)>0\) for every \(y\) and \(\theta\). Measurability in \(y\)
is immediate. Continuity in \(\theta\) follows from continuity of the
Gaussian density and dominated convergence, using the compactness of
\(\mathcal T\) and the lower bound
\(\sigma\geq\underline\sigma>0\).

\emph{A4.}
If \(X_\theta\sim N(\mu,\sigma^2)\), then
\[
    p_\theta(y)
    \geq
    \frac{e^{-|y|/b}}{2b}
    \E_\theta\!\left(e^{-|X_\theta|/b}\right).
\]
The expectation on the right is continuous and strictly positive on
the compact parameter set. Hence its infimum, denoted by \(c_b\), is
positive. Since also \(p_\theta(y)\leq(2b)^{-1}\), there are finite
constants \(C_1,C_2\) such that
\[
    \sup_{\theta\in\mathcal T}
    |\log p_\theta(y)|
    \leq
    C_1+C_2|y|.
\]
The true observed variable is the sum of a finite Gaussian-mixture
variable and a Laplace variable, so \(\E_0|Y|<\infty\). Above
bound therefore provides the required integrable log-envelope.

\emph{A5.}
The Gaussian derivatives satisfy
\(\partial_\mu\phi_\theta(x)
=(x-\mu)\phi_\theta(x)/\sigma^2\) and
\(\partial_\sigma\phi_\theta(x)
=\{(x-\mu)^2/\sigma^3-1/\sigma\}\phi_\theta(x)\).
Because \(\mu\) and \(\sigma\) range over a compact set with
\(\sigma\geq\underline\sigma>0\), the functions
\[
    H_0(x)=\sup_{\theta\in\mathcal T}\phi_\theta(x),\qquad
    H_\mu(x)=\sup_{\theta\in\mathcal T}
    |\partial_\mu\phi_\theta(x)|,\qquad
    H_\sigma(x)=\sup_{\theta\in\mathcal T}
    |\partial_\sigma\phi_\theta(x)|
\]
are integrable on \(\mathbb R\).

Differentiation under the integral gives
\(\partial_\mu p_\theta=(\partial_\mu\phi_\theta)*f_\varepsilon\) and
\(\partial_\sigma p_\theta=(\partial_\sigma\phi_\theta)*f_\varepsilon\).
Consequently, with
\(G_2=(H_0+H_\mu+H_\sigma)*f_\varepsilon\), Young's inequality gives
\(G_2\in L^1(\mathbb R)\), and
\[
    \sup_{\theta\in\mathcal T}
    \left\{
        p_\theta(y)
        +
        |\partial_\mu p_\theta(y)|
        +
        |\partial_\sigma p_\theta(y)|
    \right\}
    \leq
    G_2(y).
\]
Since \(\mathcal T\) is convex, the mean-value theorem yields
\[
    |p_{\theta_1}(y)-p_{\theta_2}(y)|
    \leq
    G_2(y)\|\theta_1-\theta_2\|_1,
    \qquad
    \theta_1,\theta_2\in\mathcal T.
\]
This verifies the required \(L^1\)-Lipschitz condition.

Thus, A1--A5 hold for the normal--Laplace observed component family.
Since \(d_K=3K-1\) is strictly increasing, the conclusion follows from
Proposition~\ref{prop:fixed-order-selection}.
\end{proof}

\subsection{Laplace posterior-kernel bounds}
For $\mu\in\mathbb R$ and $0<\sigma\le\bar\sigma$, define
\[
 b_{\mu,\sigma}(y)
 =\int\phiN_\sigma(x-\mu)k_b(y-x)dx.
\]

\begin{lemma}[Laplace component comparison and score bounds]
\label{lem:laplace-component-comparison}
Uniformly in $x,y,\mu\in\mathbb R$ and
$0<\sigma\le\bar\sigma$,
\[
 C^{-1}k_b(y-\mu)\le b_{\mu,\sigma}(y)
 \le Ck_b(y-\mu),
\]
\[
 \frac{k_b(y-x)}{b_{\mu,\sigma}(y)}
 \le C e^{|x-\mu|/b},
\]
and
\[
 \frac{|\partial_\mu b_{\mu,\sigma}(y)|
       +|\partial_\sigma b_{\mu,\sigma}(y)|}
      {b_{\mu,\sigma}(y)}
 \le\frac C\sigma.
\]
If $|\mu_j|,|\mu_l|\le M$, then
\[
 \sup_y\frac{b_{\mu_j,\sigma_j}(y)}
               {b_{\mu_l,\sigma_l}(y)}
 \le C e^{2M/b}.
\]
\end{lemma}

\begin{proof}
For all $a,u\in\mathbb R$,
\[
 e^{-|u|/b}k_b(a)\le k_b(a-u)\le e^{|u|/b}k_b(a).
\]
Integrate with $u=\sigma Z$, $Z\sim N(0,1)$, and use
$\sigma\le\bar\sigma$ to obtain the first comparison. The second follows by
applying the same ratio inequality to $a=y-\mu$ and
$u=x-\mu$, followed by the component lower bound.

For the derivatives, move differentiation to the Gaussian density:
\[
 |\partial_\mu\phiN_\sigma(u)|
 =\frac{|u|}{\sigma^2}\phiN_\sigma(u),
\qquad
 |\partial_\sigma\phiN_\sigma(u)|
 \le\left(\frac1\sigma+\frac{u^2}{\sigma^3}\right)\phiN_\sigma(u).
\]
Use $k_b(y-\mu-u)\le e^{|u|/b}k_b(y-\mu)$ and the uniformly finite
Gaussian exponential moments after $u=\sigma z$. This gives an upper bound
$C\sigma^{-1}k_b(y-\mu)$ for the sum of the derivative magnitudes. Divide
by the component lower bound. The last bound follows by comparing both
components with their centered Laplace kernels and using
$k_b(y-\mu_j)/k_b(y-\mu_l)\le e^{|\mu_j-\mu_l|/b}$.
\end{proof}

\begin{lemma}[Global posterior envelope under Laplace error]
\label{lem:laplace-posterior-envelope}
Let
\[
 E_h(x)=C\sup_{|\mu|\le M_h,\ h\le\sigma\le\bar\sigma}
 \phiN_\sigma(x-\mu)e^{|x-\mu|/b}.
\]
Then $q_f(x\mid y)\le E_h(x)$ for every
$f\in\mathcal F_h^{\rm reg}$, and
\[
 E_h(x)\le\frac Ch
 \exp\{-c(|x|-M_h-1)_+^2\}.
\]
Moreover, it holds that
\[
 \int E_h(x)dx\le C\frac{M_h+1}{h},
\]
and, consequently,
\[
 \int E_h(x)\sqrt{L_h+\log(e+|x|)}dx
 \le C\frac{M_h+1}{h}\sqrt{L_h}.
\]
\end{lemma}

\begin{proof}
Write the mixture posterior as
\[
 q_f(x\mid y)=\sum_jr_j(y)
 \frac{\phiN_{\sigma_j}(x-\mu_j)k_b(y-x)}
      {b_{\mu_j,\sigma_j}(y)},
 \qquad
 r_j(y)=\frac{\pi_jb_{\mu_j,\sigma_j}(y)}{g_f(y)}.
\]
The weights $r_j(y)$ sum to one. Apply
Lemma~\ref{lem:laplace-component-comparison} componentwise. Finally,
complete the square in
$-u^2/(2\sigma^2)+|u|/b$, use
$h\le\sigma\le\bar\sigma$, and minimize $|x-\mu|$ over
$|\mu|\le M_h$. The pointwise and integrated bounds follow from elementary
Gaussian-tail integration.
\end{proof}

\begin{lemma}[Covering of the Laplace posterior-kernel class]
\label{lem:laplace-posterior-covering}
For fixed $x$, let
$\mathcal Q_h(x)=\{q_f(x\mid\cdot):f\in\mathcal F_h^{\rm reg}\}$. There is
$A_h(x)\ge e$ such that
\[
 \log A_h(x)\le C\{L_h+\log(e+|x|)\}
\]
and, for $0<\epsilon\le1$,
\[
 \log N\{\epsilon E_h(x),\mathcal Q_h(x),\|\cdot\|_\infty\}
 \le C K_h\log\{A_h(x)/\epsilon\}.
\]
The same bound holds on the closed simplex; hence it holds on the
positive-floor subclass in~\eqref{main-eq:local-regularized-sieve}.
\end{lemma}

\begin{proof}
For an order-$K$ mixture set
$a_j(x)=\phiN_{\sigma_j}(x-\mu_j)$,
$b_j(y)=b_{\mu_j,\sigma_j}(y)$,
$c_j(x,y)=a_j(x)k_b(y-x)$,
$N=\sum_j\pi_jc_j$, and $D=\sum_j\pi_jb_j$.
Choose an index $l$ with $\pi_l\ge1/K$. By the component comparison,
$b_j/D\le CKe^{2M_h/b}$.

Use independent simplex coordinates
$\alpha_j=\pi_j$, $j<K$, and
$\pi_K=1-\sum_{j<K}\alpha_j$. Then
\[
 \partial_{\alpha_j}q=(c_j-c_K)/D-q(b_j-b_K)/D,
\]
so
$|\partial_{\alpha_j}q|\le CKe^{2M_h/b}E_h(x)$.
For locations and scales,
\[
 \partial_{\mu_j}q
 =\frac{\pi_j}{D}\{\partial_{\mu_j}c_j-q\partial_{\mu_j}b_j\},
\]
and analogously for $\sigma_j$. Since
$\pi_jb_j/D\le1$, Lemma~\ref{lem:laplace-component-comparison} 
gives
\[
 |\partial_{\mu_j}q|
 \le CE_h(x)\{h^{-1}+|x-\mu_j|h^{-2}\},
\]
\[
 |\partial_{\sigma_j}q|
 \le CE_h(x)\{h^{-1}+|x-\mu_j|^2h^{-3}\}.
\]
Thus,
\[
 \sup_y\|\nabla_\theta q_f(x\mid y)\|_\infty
 \le E_h(x)B_h(x),
\]
where
$B_h(x)=C\{K_he^{2M_h/b}+h^{-3}(1+|x|+M_h)^2\}$.
Construct a net inside the simplex and Cartesian nets for locations and
scales with mesh proportional to
$\epsilon/(K_hB_h(x))$. The line segment between two simplex points remains
inside the simplex, so the mean-value theorem gives a uniform
$\epsilon E_h(x)$ net. Counting the grid points, and then taking the union
over $K\le K_h$, proves the result. Since $M_h=O(L_h)$, the exponential
factor contributes only $O(L_h)$ to the logarithm of the covering constant.
\end{proof}

\section{Optimization gap of the EM implementation}
\label{supp_EMoptim}
Optimization-gap bounds and their consequences for the direct and posterior estimators.

Let
\(\mathcal M_n(f)\)
denote the likelihood objective function in~\eqref{main-eq:penalized-sieve-estimator}, and
\(\widehat f^{\,\mathrm{EM}}\) be the density estimated by the numerical EM implementation.

\begin{corollary}[Effect of approximate EM optimization]
\label{cor:em-error}
Suppose that Assumptions~\ref{main-ass:primitive-approximation}--\ref{main-ass:primitive-penalty} hold. Define
\(
    e_n^{\mathrm{EM}}
    =
    b_n
    +
    c_n^{1/2}
    +
    \left\{
        \frac{\operatorname{pen}_n(f_n^\circ)}{n}
    \right\}^{1/2}
    +
    a_n^{1/2}.
\)
If there exists a deterministic sequence
\(a_n\to0\) such that
\[
    \Delta_n
    =
    \sup_{f\in\mathcal F_n}\mathcal M_n(f)
    -
    \mathcal M_n(\widehat f^{\,\mathrm{EM}})
    =O_p(a_n),
\]
then
\[
    \KL(g_{f_0},g_{\widehat f^{\,\mathrm{EM}}})^{1/2}
    =
    O_p(e_n^{\mathrm{EM}}).
\]
\end{corollary}

\begin{proof}
    Since $\Delta_n=O_p(a_n)$, we have
\(
    \mathcal M_n(\widehat f^{\,\mathrm{EM}})
    \ge
    \sup_{f\in\mathcal F_n}\mathcal M_n(f)-\Delta_n
    \ge
    \mathcal M_n(f_n^\circ)-\Delta_n .
\)
Repeating the full-class likelihood comparison in the proof of Proposition~\ref{main-prop:observed-kl-rate},
and adding the optimization gap \(\Delta_n\), gives
\[
    \KL(g_{f_0},g_{\widehat f^{\,\mathrm{EM}}})=O_p\left(b_n^2+c_n+\frac{\operatorname{pen}_n(f_n^\circ)}{n}+a_n\right).
\]
\end{proof}

\section{Posterior moments under Laplace errors}
\label{app:laplace-posterior-moments}
This section contains the conditional posterior-moment calculations used when the measurement error is Laplace, and their numerically stable implementation.

Suppose that \(X\mid C=k\sim N(\mu,\sigma^2)\) and \(\varepsilon\sim \mathrm{Laplace}(0,b)\), independently. Then
\[
p(x\mid y,C=k)\propto \phi(x;\mu,\sigma^2)\exp\{-|y-x|/b\}.
\]
Splitting the support at \(x=y\), the conditional density is a two-piece mixture of truncated normal kernels. For \(x\le y\), the kernel is proportional to a normal density with mean \(m_L=\mu+\sigma^2/b\), truncated to \((-\infty,y]\). For \(x>y\), it is proportional to a normal density with mean \(m_R=\mu-\sigma^2/b\), truncated to \([y,\infty)\). The corresponding unnormalized weights are
\[
A_L=\exp\left\{\frac{\mu-y}{b}+\frac{\sigma^2}{2b^2}\right\}\Phi\left(\frac{y-\mu-\sigma^2/b}{\sigma}\right),
\]
and
\[
A_R=\exp\left\{\frac{y-\mu}{b}+\frac{\sigma^2}{2b^2}\right\}\left[1-\Phi\left(\frac{y-\mu+\sigma^2/b}{\sigma}\right)\right].
\]
After normalization, the first two posterior moments are obtained using standard truncated-normal formulas. In the implementation, all weights are evaluated on the log scale to avoid numerical underflow in tail regions.

\section{Numerical safeguards}
\label{sec:numerical-safeguards}
This section contains numerical safeguards and tuning choices for optimization, posterior reconstruction, and the benchmark estimators.

All estimators are implemented with fixed numerical safeguards to avoid
unstable likelihood evaluations, degenerate mixture components, and
ill-conditioned algebra steps. The Gaussian-mixture sieve is the
constrained parameter space used in the likelihood optimization. For each candidate order, the constraints are
imposed at initialization and at every EM update: component weights are
projected onto the simplex with the stated lower bound by the exact
water-filling solution, locations are clipped to their admissible interval,
and component variances are clipped to the squared scale bounds. The selected
fit is then audited against all constraints; a replication is not accepted if
the selected fit fails to converge or violates a bound.

Estimator $\widehat f_{\post}$ uses the exact fitted observed density in the
denominator. Under Gaussian error it is evaluated through the conjugate
conditional-mixture representation, and under Laplace error through log-scale
posterior ratios. No denominator floor is used. MIX, POST, DEC, and PC are
reported in their raw form, without final positivity clipping or finite-grid
renormalization; QP retains only the nonnegativity and unit-mass constraints
that define that estimator. Table~\ref{tab:numerical-safeguards} summarizes the
main implementation choices.

\begingroup
\small
\setlength{\tabcolsep}{4pt}
\renewcommand{\arraystretch}{1.10}
\begin{longtable}
{p{0.18\textwidth}p{0.45\textwidth}p{0.3\textwidth}}
\caption{Numerical safeguards and tuning choices used in the simulations.}
\label{tab:numerical-safeguards}\\
\toprule
Estimator or step & Implementation & Purpose \\
\midrule
\endfirsthead
\caption[]{Numerical safeguards and tuning choices used in the simulations.}\\
\toprule
Estimator or step & Implementation & Purpose \\
\midrule
\endhead
\midrule
\multicolumn{3}{r}{\textit{Continued on next page}}\\
\endfoot
\bottomrule
\endlastfoot
Mixture order
& Ordinary BIC over $K=1,\ldots,K_{\max,n}$, with
$K_{\max,n}=\max\{2,\lfloor n^{1/3}(\log n)^{-1/3}\rfloor\}$; the caps are
$3,4,5,8$ for $n=200,500,1000,5000$
& Controls mixture complexity within the theoretical sieve \\
Mixture scales
& $0.10\leq \sigma_j\leq 5$ for every component
& Prevents collapsing and excessively diffuse components \\
Mixture locations
& $|\mu_j|\leq 8\sqrt{\log n}$ for every component
& Restricts the location range while allowing it to grow with $n$ \\
Mixture weights
& $\pi_j\geq n^{-0.60}$ and $\sum_j\pi_j=1$
& Prevents asymptotically vanishing fitted components \\
Sieve-constraint enforcement
& Exact lower-bounded simplex projection for the weights, clipping of locations
and variances at every EM M-step, constrained initialization, and a final
constraint audit; invalid selected fits trigger a deterministic retry
& Implements the declared sieve throughout optimization \\
Gaussian-mixture optimization
& Adaptive multistart constrained EM: one start for $K=1$; for $K>1$, a
quantile start, a split warm start when available, and additional seeded starts
(up to 8, with at least 3 unless early agreement is reached); 600 iterations
and tolerance $10^{-6}$
& Reduces sensitivity to local optima and requires convergence of the selected fit \\
Likelihood evaluation
& Log-sum-exp computations
& Stable likelihood and responsibility evaluation \\
Posterior estimator
& Closed-form conjugate conditional mixtures under Gaussian error and log-scale
posterior ratios under Laplace error; no denominator floor, clipping, or
post-hoc renormalization
& Evaluates the posterior reconstruction as defined \\
Numerical evaluation
& Common 801-point design-specific grid; raw finite-grid evaluations for all
methods, with no positivity clipping or mass correction except for the defining
QP simplex constraint
& Makes the reported ISE comparisons use the estimators themselves \\
Fourier deconvolution
& Delaigle--Gijbels plug-in bandwidth; frequencies with
$|\phi_\varepsilon(t)|<10^{-8}$ are excluded
& Stabilizes Fourier inversion \\
PC estimator
& Penalized contrast over the stated projection grid; frequencies with
$|\phi_\varepsilon(t)|<10^{-12}$ are excluded
& Avoids unstable projection coefficients \\
QP estimator
& Nonnegative unit-mass simplex, second-difference regularization, SURE choice
of the penalty over a fixed grid, and ADMM with Cholesky jitter and
warm-started continuation
& Controls roughness and improves conditioning \\
Skew-normal shape regularization
& Maximize $\ell(\theta)-Q(\boldsymbol{\alpha})$,
$Q(\boldsymbol{\alpha})=c_1\sum_{j=1}^{K}\log(1+c_2\alpha_j^2)$,
$c_1=0.875913$ and $c_2=0.856250$
& Regularizes extreme shape parameters \\
Skew-normal shape bounds
& $-30\leq\alpha_j\leq30$
& Numerical safeguard against diverging shape estimates \\
Skew-normal likelihood evaluation
& Closed form under Gaussian error; 48-node Gauss--Hermite quadrature under
Laplace error, with analytic gradient
& Stable convolution-likelihood evaluation \\
Skew-normal optimization
& L-BFGS-B from 30 initial configurations for each candidate $K$
& Reduces sensitivity to local optima \\
\end{longtable}
\endgroup

\section{Selected mixture orders}
\label{app:selected-k-details}
This section contains selected mixture orders across simulation designs and sample sizes.
The average selected mixture
order for each latent density is reported in Table~\ref{tab:selected-k-by-latent-n}. It shows
which designs drive the increase in $\widehat K$. The Gaussian and nearly-normal
baselines remain essentially one-component, the heavy-tailed Cauchy design
requires the largest number of Gaussian components, and the two normal-mixture
designs are represented by about two components at moderate and large sample
sizes. The asymmetric positive-support designs show increasing selected
complexity with $n$. The two mixed-gamma designs also approach two components,
but this does not imply that a symmetric Gaussian sieve reproduces their
component shapes exactly.
\begin{table}[htb]
\centering
\caption{Average selected number of Gaussian mixture components
$\widehat K$ by latent density and sample size. Entries are averages over
error laws, noise levels and Monte Carlo replications.}
\label{tab:selected-k-by-latent-n}
\scriptsize
\setlength{\tabcolsep}{4pt}
\renewcommand{\arraystretch}{1.10}
\begin{tabular}{lrrrr}
\toprule
Latent density & $n=200$ & $n=500$ & $n=1000$ & $n=5000$ \\
\midrule
\multicolumn{5}{l}{\textbf{Panel A. Symmetric-unimodal}} \\
\addlinespace[2pt]
Standard normal & 1.00 & 1.00 & 1.00 & 1.00 \\
\(t_5\) & 1.24 & 1.56 & 1.85 & 2.16 \\
Laplace & 1.40 & 1.78 & 1.98 & 2.13 \\
Cauchy & 3.00 & 3.99 & 4.07 & 5.44 \\
Nearly normal & 1.00 & 1.00 & 1.00 & 1.03 \\
\addlinespace[3pt]
\midrule
\multicolumn{5}{l}{\textbf{Panel B. Asymmetric-unimodal}} \\
\addlinespace[2pt]
Gamma & 1.38 & 1.84 & 2.04 & 2.60 \\
Comte chi-square & 1.97 & 2.24 & 2.54 & 3.33 \\
Cai chi-square & 1.87 & 2.15 & 2.41 & 3.18 \\
Beta & 1.24 & 1.68 & 1.95 & 2.38 \\
\addlinespace[3pt]
\midrule
\multicolumn{5}{l}{\textbf{Panel C. Bimodal}} \\
\addlinespace[2pt]
Comte Gamma mixture & 1.25 & 1.59 & 1.80 & 2.02 \\
Cai Gamma mixture & 1.23 & 1.58 & 1.83 & 2.01 \\
Normal mixture G & 1.77 & 1.97 & 2.00 & 2.00 \\
Normal mixture Y & 1.89 & 2.01 & 2.01 & 2.00 \\
\bottomrule
\end{tabular}
\end{table}

\section{Complete scenario-level simulation tables}
\label{app:complete-simulation-table}
This section contains complete scenario-level MISE results and standard errors for all methods and simulation designs. See Tables \ref{tab:symmetric-unimodal-mise-se}--\ref{tab:bimodal-mise-se}.


\begingroup
\scriptsize
\setlength{\tabcolsep}{4.0pt}
\renewcommand{\arraystretch}{1.05}
\begin{longtable}{rlcccccc}
\caption{Comparison for the symmetric unimodal latent densities in terms of \(10^3\times\mathrm{MISE}\) (SE). Row minima for MISE are shown in bold.}\label{tab:symmetric-unimodal-mise-se}\\
\toprule
$n$ & Error & $c$ & $\mathrm{MIX}$ & ${\mathrm{POST}}$ & DEC & PC & QP \\
\midrule
\endfirsthead
\multicolumn{8}{c}{\tablename\ \thetable\ -- continued from previous page}\\
\toprule
$n$ & Error & $c$ & $\mathrm{MIX}$ & ${\mathrm{POST}}$ & DEC & PC & QP \\
\midrule
\endhead
\midrule\multicolumn{8}{r}{Continued on next page}\\
\endfoot
\bottomrule
\endlastfoot
\multicolumn{8}{l}{\textit{Standard normal}} \\
\multirow{4}{*}{200} & \multirow{2}{*}{Gaussian} & 0.35 & \textbf{1.60} (0.13) & 3.37 (0.15) & 5.21 (0.15) & 6.06 (0.20) & 12.87 (1.02) \\
 &  & 0.60 & \textbf{2.12} (0.26) & 2.61 (0.27) & 9.06 (0.22) & 24.57 (0.93) & 15.76 (1.30) \\
 & \multirow{2}{*}{Laplace} & 0.35 & \textbf{1.60} (0.07) & 4.39 (0.11) & 4.95 (0.16) & 5.79 (0.19) & 14.88 (1.07) \\
 &  & 0.60 & \textbf{2.07} (0.10) & 3.04 (0.12) & 6.92 (0.20) & 11.53 (0.40) & 16.57 (1.24) \\
\multirow{4}{*}{500} & \multirow{2}{*}{Gaussian} & 0.35 & \textbf{0.60} (0.03) & 1.23 (0.04) & 2.90 (0.07) & 2.28 (0.08) & 6.60 (0.75) \\
 &  & 0.60 & \textbf{0.77} (0.03) & 0.95 (0.04) & 5.42 (0.12) & 9.61 (0.35) & 8.87 (1.00) \\
 & \multirow{2}{*}{Laplace} & 0.35 & \textbf{0.69} (0.08) & 1.73 (0.08) & 2.55 (0.07) & 2.23 (0.08) & 7.34 (0.75) \\
 &  & 0.60 & \textbf{0.86} (0.04) & 1.20 (0.04) & 3.82 (0.11) & 4.56 (0.16) & 9.25 (0.83) \\
\multirow{4}{*}{1000} & \multirow{2}{*}{Gaussian} & 0.35 & \textbf{0.28} (0.01) & 0.63 (0.02) & 1.82 (0.04) & 1.19 (0.04) & 5.61 (0.81) \\
 &  & 0.60 & \textbf{0.41} (0.02) & 0.52 (0.02) & 4.45 (0.09) & 4.92 (0.19) & 7.20 (1.02) \\
 & \multirow{2}{*}{Laplace} & 0.35 & \textbf{0.31} (0.02) & 0.83 (0.02) & 1.58 (0.04) & 1.09 (0.04) & 4.16 (0.60) \\
 &  & 0.60 & \textbf{0.39} (0.02) & 0.57 (0.02) & 2.79 (0.06) & 2.29 (0.08) & 6.76 (0.79) \\
\multirow{4}{*}{5000} & \multirow{2}{*}{Gaussian} & 0.35 & \textbf{0.06} (0.00) & 0.13 (0.00) & 0.85 (0.02) & 0.25 (0.01) & 2.54 (0.55) \\
 &  & 0.60 & \textbf{0.08} (0.00) & 0.10 (0.00) & 2.40 (0.04) & 1.02 (0.04) & 4.47 (0.76) \\
 & \multirow{2}{*}{Laplace} & 0.35 & \textbf{0.06} (0.00) & 0.17 (0.00) & 0.66 (0.01) & 0.23 (0.01) & 1.48 (0.22) \\
 &  & 0.60 & \textbf{0.08} (0.00) & 0.11 (0.00) & 1.23 (0.02) & 0.45 (0.01) & 2.32 (0.43) \\
\addlinespace[2pt]
\multicolumn{8}{l}{\textit{\(t_5\)}} \\
\multirow{4}{*}{200} & \multirow{2}{*}{Gaussian} & 0.35 & 7.24 (0.33) & \textbf{5.26} (0.32) & 8.00 (0.22) & 6.58 (0.21) & 11.58 (1.03) \\
 &  & 0.60 & 8.41 (0.55) & \textbf{6.60} (0.54) & 14.99 (0.40) & 27.76 (0.97) & 15.64 (1.14) \\
 & \multirow{2}{*}{Laplace} & 0.35 & 7.44 (0.65) & 6.58 (0.64) & 7.64 (0.21) & \textbf{6.24} (0.20) & 13.43 (1.12) \\
 &  & 0.60 & 7.03 (0.36) & \textbf{5.73} (0.33) & 11.70 (0.29) & 12.11 (0.37) & 16.29 (1.34) \\
\multirow{4}{*}{500} & \multirow{2}{*}{Gaussian} & 0.35 & 2.99 (0.12) & \textbf{2.33} (0.09) & 4.69 (0.13) & 3.13 (0.09) & 8.95 (1.07) \\
 &  & 0.60 & 6.45 (0.57) & \textbf{5.07} (0.57) & 11.25 (0.21) & 11.15 (0.42) & 11.03 (1.10) \\
 & \multirow{2}{*}{Laplace} & 0.35 & 3.13 (0.13) & \textbf{2.64} (0.10) & 4.07 (0.10) & 2.68 (0.07) & 6.89 (0.72) \\
 &  & 0.60 & 5.15 (0.18) & \textbf{3.77} (0.16) & 7.07 (0.17) & 5.35 (0.17) & 11.38 (1.14) \\
\multirow{4}{*}{1000} & \multirow{2}{*}{Gaussian} & 0.35 & 1.31 (0.06) & \textbf{1.20} (0.05) & 3.61 (0.08) & 1.82 (0.04) & 8.24 (1.15) \\
 &  & 0.60 & 2.50 (0.12) & \textbf{2.07} (0.09) & 9.60 (0.18) & 5.92 (0.20) & 8.38 (0.93) \\
 & \multirow{2}{*}{Laplace} & 0.35 & 1.29 (0.05) & \textbf{1.28} (0.04) & 2.67 (0.08) & 1.69 (0.04) & 4.67 (0.64) \\
 &  & 0.60 & 3.06 (0.13) & \textbf{2.39} (0.11) & 4.56 (0.12) & 2.99 (0.08) & 8.35 (0.90) \\
\multirow{4}{*}{5000} & \multirow{2}{*}{Gaussian} & 0.35 & 0.44 (0.01) & \textbf{0.34} (0.01) & 1.64 (0.03) & 0.82 (0.01) & 4.32 (0.86) \\
 &  & 0.60 & 0.79 (0.03) & \textbf{0.71} (0.02) & 5.61 (0.07) & 1.55 (0.04) & 5.37 (0.81) \\
 & \multirow{2}{*}{Laplace} & 0.35 & 0.43 (0.01) & \textbf{0.33} (0.01) & 1.00 (0.02) & 0.78 (0.01) & 1.37 (0.25) \\
 &  & 0.60 & 0.62 (0.02) & \textbf{0.54} (0.02) & 1.96 (0.04) & 1.05 (0.02) & 2.93 (0.48) \\
\addlinespace[2pt]
\multicolumn{8}{l}{\textit{Laplace}} \\
\multirow{4}{*}{200} & \multirow{2}{*}{Gaussian} & 0.35 & 22.94 (1.04) & 17.36 (1.02) & 18.94 (0.30) & \textbf{16.64} (0.22) & 23.25 (1.48) \\
 &  & 0.60 & 33.79 (1.58) & \textbf{28.07} (1.60) & 30.70 (0.53) & 37.77 (1.03) & 28.44 (1.13) \\
 & \multirow{2}{*}{Laplace} & 0.35 & 21.39 (1.15) & 15.62 (1.13) & 16.44 (0.33) & \textbf{15.49} (0.17) & 21.08 (1.53) \\
 &  & 0.60 & 30.10 (1.66) & 24.23 (1.69) & 24.75 (0.47) & \textbf{22.65} (0.41) & 26.28 (1.11) \\
\multirow{4}{*}{500} & \multirow{2}{*}{Gaussian} & 0.35 & 8.42 (0.26) & \textbf{7.66} (0.23) & 13.03 (0.19) & 12.99 (0.08) & 16.37 (1.19) \\
 &  & 0.60 & 17.81 (0.56) & \textbf{15.49} (0.47) & 27.45 (0.29) & 20.88 (0.36) & 21.36 (0.86) \\
 & \multirow{2}{*}{Laplace} & 0.35 & 8.62 (0.35) & \textbf{7.58} (0.32) & 11.78 (0.21) & 12.73 (0.08) & 15.12 (1.08) \\
 &  & 0.60 & 18.86 (0.87) & 15.58 (0.84) & 17.54 (0.30) & \textbf{15.22} (0.17) & 21.93 (1.39) \\
\multirow{4}{*}{1000} & \multirow{2}{*}{Gaussian} & 0.35 & 6.44 (0.11) & \textbf{5.74} (0.10) & 12.12 (0.15) & 11.79 (0.05) & 13.53 (1.02) \\
 &  & 0.60 & 9.19 (0.27) & \textbf{8.78} (0.22) & 21.51 (0.36) & 15.40 (0.18) & 19.35 (0.96) \\
 & \multirow{2}{*}{Laplace} & 0.35 & 5.79 (0.12) & \textbf{4.93} (0.11) & 7.88 (0.12) & 11.66 (0.04) & 10.91 (0.84) \\
 &  & 0.60 & 8.33 (0.25) & \textbf{7.58} (0.19) & 12.60 (0.18) & 12.86 (0.08) & 15.31 (0.98) \\
\multirow{4}{*}{5000} & \multirow{2}{*}{Gaussian} & 0.35 & 4.68 (0.07) & \textbf{4.09} (0.06) & 7.06 (0.07) & 10.78 (0.01) & 8.60 (0.87) \\
 &  & 0.60 & 6.44 (0.07) & \textbf{6.33} (0.07) & 17.34 (0.14) & 11.53 (0.04) & 11.84 (0.59) \\
 & \multirow{2}{*}{Laplace} & 0.35 & 3.86 (0.07) & \textbf{3.09} (0.05) & 4.68 (0.05) & 10.75 (0.01) & 4.67 (0.53) \\
 &  & 0.60 & 5.25 (0.06) & \textbf{4.88} (0.06) & 7.43 (0.08) & 11.00 (0.02) & 7.52 (0.45) \\
\addlinespace[2pt]
\multicolumn{8}{l}{\textit{Cauchy}} \\
\multirow{4}{*}{200} & \multirow{2}{*}{Gaussian} & 0.35 & 14.55 (0.43) & \textbf{4.62} (0.11) & 46.99 (0.50) & 7.30 (0.18) & 6.78 (0.42) \\
 &  & 0.60 & 17.28 (0.44) & \textbf{6.35} (0.15) & 50.63 (0.50) & 27.76 (0.76) & 8.53 (0.46) \\
 & \multirow{2}{*}{Laplace} & 0.35 & 14.92 (0.43) & \textbf{5.39} (0.15) & 47.27 (0.48) & 6.77 (0.17) & 6.85 (0.47) \\
 &  & 0.60 & 16.48 (0.44) & \textbf{6.19} (0.14) & 49.76 (0.51) & 13.09 (0.31) & 9.21 (0.59) \\
\multirow{4}{*}{500} & \multirow{2}{*}{Gaussian} & 0.35 & 2.87 (0.14) & \textbf{1.99} (0.05) & 43.89 (0.31) & 3.12 (0.07) & 4.03 (0.38) \\
 &  & 0.60 & 3.63 (0.16) & \textbf{2.77} (0.07) & 46.39 (0.25) & 11.77 (0.32) & 6.58 (0.50) \\
 & \multirow{2}{*}{Laplace} & 0.35 & 2.57 (0.06) & \textbf{2.38} (0.05) & 43.45 (0.34) & 2.89 (0.07) & 3.63 (0.25) \\
 &  & 0.60 & 2.93 (0.07) & \textbf{2.43} (0.05) & 45.99 (0.26) & 5.37 (0.14) & 4.92 (0.34) \\
\multirow{4}{*}{1000} & \multirow{2}{*}{Gaussian} & 0.35 & 2.55 (0.13) & \textbf{1.30} (0.03) & 37.94 (0.11) & 1.70 (0.04) & 2.68 (0.32) \\
 &  & 0.60 & 3.16 (0.14) & \textbf{2.09} (0.05) & 37.99 (0.10) & 5.90 (0.17) & 4.64 (0.40) \\
 & \multirow{2}{*}{Laplace} & 0.35 & 2.13 (0.04) & \textbf{1.45} (0.03) & 37.94 (0.10) & 1.58 (0.03) & 2.12 (0.15) \\
 &  & 0.60 & 2.55 (0.05) & \textbf{1.82} (0.04) & 38.03 (0.11) & 2.84 (0.06) & 3.76 (0.29) \\
\multirow{4}{*}{5000} & \multirow{2}{*}{Gaussian} & 0.35 & 1.13 (0.07) & \textbf{0.41} (0.01) & 20.28 (0.02) & 0.57 (0.01) & 1.12 (0.17) \\
 &  & 0.60 & 2.18 (0.10) & \textbf{1.28} (0.04) & 20.27 (0.02) & 1.43 (0.03) & 2.87 (0.38) \\
 & \multirow{2}{*}{Laplace} & 0.35 & 0.42 (0.01) & \textbf{0.33} (0.01) & 20.31 (0.02) & 0.56 (0.01) & 0.76 (0.09) \\
 &  & 0.60 & 0.82 (0.04) & \textbf{0.59} (0.02) & 20.29 (0.02) & 0.80 (0.01) & 1.39 (0.14) \\
\addlinespace[2pt]
\multicolumn{8}{l}{\textit{Nearly Normal}} \\
\multirow{4}{*}{200} & \multirow{2}{*}{Gaussian} & 0.35 & \textbf{12.29} (0.53) & 22.03 (0.64) & 32.39 (0.91) & 40.30 (1.37) & 56.79 (6.57) \\
 &  & 0.60 & \textbf{15.79} (0.82) & 18.53 (0.84) & 57.32 (1.35) & 165.33 (6.25) & 69.37 (6.62) \\
 & \multirow{2}{*}{Laplace} & 0.35 & \textbf{11.66} (0.54) & 29.33 (0.83) & 31.76 (0.99) & 39.18 (1.40) & 61.80 (6.84) \\
 &  & 0.60 & \textbf{16.64} (0.73) & 22.06 (0.79) & 40.93 (1.30) & 78.89 (2.55) & 60.51 (5.93) \\
\multirow{4}{*}{500} & \multirow{2}{*}{Gaussian} & 0.35 & \textbf{6.18} (0.24) & 9.07 (0.29) & 17.58 (0.46) & 16.21 (0.55) & 33.26 (4.77) \\
 &  & 0.60 & \textbf{8.34} (0.33) & 8.64 (0.33) & 32.57 (0.76) & 67.38 (2.47) & 52.75 (7.26) \\
 & \multirow{2}{*}{Laplace} & 0.35 & \textbf{6.59} (0.24) & 12.48 (0.34) & 17.18 (0.48) & 16.20 (0.56) & 26.57 (2.91) \\
 &  & 0.60 & \textbf{7.49} (0.31) & 9.55 (0.36) & 26.23 (0.74) & 32.26 (1.10) & 44.64 (5.12) \\
\multirow{4}{*}{1000} & \multirow{2}{*}{Gaussian} & 0.35 & \textbf{4.23} (0.14) & 4.81 (0.15) & 11.54 (0.28) & 8.43 (0.27) & 21.59 (4.12) \\
 &  & 0.60 & 4.93 (0.18) & \textbf{4.47} (0.17) & 26.56 (0.48) & 32.33 (1.11) & 43.80 (6.69) \\
 & \multirow{2}{*}{Laplace} & 0.35 & \textbf{4.11} (0.13) & 5.92 (0.15) & 9.89 (0.24) & 7.37 (0.25) & 13.44 (1.61) \\
 &  & 0.60 & \textbf{4.68} (0.17) & 4.96 (0.17) & 16.18 (0.40) & 15.35 (0.52) & 25.11 (3.15) \\
\multirow{4}{*}{5000} & \multirow{2}{*}{Gaussian} & 0.35 & 2.75 (0.06) & \textbf{1.43} (0.06) & 5.66 (0.11) & 1.79 (0.06) & 12.13 (2.85) \\
 &  & 0.60 & 2.91 (0.07) & \textbf{1.93} (0.07) & 14.26 (0.21) & 6.60 (0.24) & 21.12 (4.36) \\
 & \multirow{2}{*}{Laplace} & 0.35 & 2.63 (0.05) & 1.57 (0.05) & 3.89 (0.08) & \textbf{1.55} (0.05) & 4.95 (0.84) \\
 &  & 0.60 & 2.40 (0.07) & \textbf{1.71} (0.07) & 7.62 (0.14) & 3.44 (0.12) & 11.08 (1.73) \\
\end{longtable}
\endgroup

\begingroup
\scriptsize
\setlength{\tabcolsep}{4.0pt}
\renewcommand{\arraystretch}{1.05}
\begin{longtable}{rlcccccc}
\caption{Comparison for the asymmetric unimodal latent densities in terms of \(10^3\times\mathrm{MISE}\) (SE). Row minima for MISE are shown in bold.}\label{tab:asymmetric-unimodal-gaussian-sieve-mise-se}\\
\toprule
$n$ & Error & $c$ & $\mathrm{MIX}$ & ${\mathrm{POST}}$ & DEC & PC & QP \\
\midrule
\endfirsthead
\multicolumn{8}{c}{\tablename\ \thetable\ -- continued from previous page}\\
\toprule
$n$ & Error & $c$ & $\mathrm{MIX}$ & ${\mathrm{POST}}$ & DEC & PC & QP \\
\midrule
\endhead
\midrule\multicolumn{8}{r}{Continued on next page}\\
\endfoot
\bottomrule
\endlastfoot
\multicolumn{8}{l}{\textit{Gamma}} \\
\multirow{4}{*}{200} & \multirow{2}{*}{Gaussian} & 0.35 & 15.96 (0.78) & 11.07 (0.79) & 6.94 (0.19) & \textbf{6.57} (0.21) & 16.49 (1.67) \\
 &  & 0.60 & 22.61 (1.46) & 17.42 (1.50) & \textbf{13.40} (0.33) & 25.05 (0.87) & 17.32 (1.26) \\
 & \multirow{2}{*}{Laplace} & 0.35 & 14.55 (0.68) & 10.56 (0.67) & 6.71 (0.17) & \textbf{5.96} (0.19) & 13.38 (1.03) \\
 &  & 0.60 & 21.21 (1.78) & 15.98 (1.81) & \textbf{10.31} (0.23) & 11.64 (0.40) & 19.47 (1.50) \\
\multirow{4}{*}{500} & \multirow{2}{*}{Gaussian} & 0.35 & 6.25 (0.28) & 4.79 (0.28) & 4.39 (0.10) & \textbf{2.97} (0.08) & 8.68 (1.04) \\
 &  & 0.60 & 12.19 (0.40) & 10.06 (0.40) & \textbf{9.54} (0.17) & 9.98 (0.38) & 11.49 (1.18) \\
 & \multirow{2}{*}{Laplace} & 0.35 & 5.52 (0.16) & 4.16 (0.15) & 3.71 (0.09) & \textbf{2.67} (0.08) & 5.64 (0.53) \\
 &  & 0.60 & 11.52 (0.42) & 8.73 (0.41) & 5.96 (0.13) & \textbf{4.95} (0.15) & 10.01 (0.93) \\
\multirow{4}{*}{1000} & \multirow{2}{*}{Gaussian} & 0.35 & 4.41 (0.15) & 3.03 (0.16) & 3.10 (0.06) & \textbf{1.63} (0.04) & 4.20 (0.56) \\
 &  & 0.60 & 6.21 (0.21) & 5.64 (0.21) & 8.76 (0.12) & \textbf{5.21} (0.16) & 10.73 (1.26) \\
 & \multirow{2}{*}{Laplace} & 0.35 & 4.07 (0.11) & 2.80 (0.11) & 2.52 (0.06) & \textbf{1.56} (0.04) & 4.39 (0.60) \\
 &  & 0.60 & 5.91 (0.33) & 5.14 (0.32) & 4.26 (0.10) & \textbf{2.87} (0.08) & 8.63 (0.98) \\
\multirow{4}{*}{5000} & \multirow{2}{*}{Gaussian} & 0.35 & 1.64 (0.04) & 1.39 (0.04) & 1.45 (0.02) & \textbf{0.77} (0.01) & 4.64 (0.83) \\
 &  & 0.60 & 4.19 (0.07) & 3.80 (0.07) & 4.64 (0.05) & \textbf{1.48} (0.04) & 4.64 (0.68) \\
 & \multirow{2}{*}{Laplace} & 0.35 & 1.29 (0.03) & 1.02 (0.03) & 0.92 (0.02) & \textbf{0.74} (0.01) & 1.40 (0.21) \\
 &  & 0.60 & 2.95 (0.04) & 2.46 (0.04) & 1.77 (0.03) & \textbf{0.98} (0.02) & 3.97 (0.59) \\
\addlinespace[2pt]
\multicolumn{8}{l}{\textit{Comte chi-square}} \\
\multirow{4}{*}{200} & \multirow{2}{*}{Gaussian} & 0.35 & 53.28 (1.97) & 46.44 (2.03) & 34.87 (0.40) & \textbf{30.91} (0.23) & 35.34 (1.53) \\
 &  & 0.60 & 86.61 (4.30) & 82.21 (4.35) & 51.41 (0.55) & 49.78 (0.98) & \textbf{49.76} (1.46) \\
 & \multirow{2}{*}{Laplace} & 0.35 & 46.44 (1.41) & 38.45 (1.45) & \textbf{29.95} (0.38) & 29.98 (0.19) & 31.31 (1.26) \\
 &  & 0.60 & 71.27 (3.31) & 65.26 (3.37) & 40.56 (0.53) & \textbf{35.73} (0.41) & 41.37 (1.35) \\
\multirow{4}{*}{500} & \multirow{2}{*}{Gaussian} & 0.35 & 39.70 (1.07) & 35.24 (1.11) & 26.98 (0.20) & 27.24 (0.09) & \textbf{26.00} (1.14) \\
 &  & 0.60 & 49.48 (1.45) & 47.42 (1.46) & 47.18 (0.38) & \textbf{33.99} (0.38) & 38.30 (1.33) \\
 & \multirow{2}{*}{Laplace} & 0.35 & 33.43 (0.67) & 27.67 (0.69) & \textbf{21.73} (0.28) & 26.86 (0.08) & 22.60 (1.03) \\
 &  & 0.60 & 42.69 (1.21) & 39.28 (1.22) & 31.52 (0.40) & \textbf{29.29} (0.17) & 33.27 (1.21) \\
\multirow{4}{*}{1000} & \multirow{2}{*}{Gaussian} & 0.35 & 29.84 (0.86) & 28.08 (0.87) & \textbf{22.85} (0.26) & 25.94 (0.05) & 23.44 (1.32) \\
 &  & 0.60 & 45.15 (1.05) & 43.32 (1.07) & 37.57 (0.41) & \textbf{29.78} (0.21) & 36.11 (1.31) \\
 & \multirow{2}{*}{Laplace} & 0.35 & 23.12 (0.46) & 20.45 (0.45) & 18.39 (0.12) & 25.81 (0.04) & \textbf{16.58} (0.73) \\
 &  & 0.60 & 37.16 (0.85) & 34.38 (0.87) & 26.67 (0.20) & 26.94 (0.08) & \textbf{24.91} (1.12) \\
\multirow{4}{*}{5000} & \multirow{2}{*}{Gaussian} & 0.35 & 20.40 (0.33) & 19.75 (0.33) & 17.49 (0.08) & 25.05 (0.02) & \textbf{15.40} (0.96) \\
 &  & 0.60 & 29.02 (0.48) & 28.86 (0.48) & 34.50 (0.16) & 25.79 (0.05) & \textbf{24.44} (1.07) \\
 & \multirow{2}{*}{Laplace} & 0.35 & 13.00 (0.17) & 12.17 (0.15) & \textbf{9.28} (0.08) & 22.25 (0.27) & 9.29 (0.60) \\
 &  & 0.60 & 19.72 (0.23) & 19.08 (0.23) & 18.01 (0.10) & 25.25 (0.02) & \textbf{14.29} (0.71) \\
\addlinespace[2pt]
\multicolumn{8}{l}{\textit{Cai chi-square}} \\
\multirow{4}{*}{200} & \multirow{2}{*}{Gaussian} & 0.35 & 33.18 (1.52) & 27.96 (1.54) & 18.52 (0.26) & \textbf{14.42} (0.19) & 25.21 (1.63) \\
 &  & 0.60 & 61.90 (3.49) & 56.64 (3.56) & \textbf{30.13} (0.45) & 32.45 (0.86) & 33.37 (1.34) \\
 & \multirow{2}{*}{Laplace} & 0.35 & 29.58 (1.12) & 23.59 (1.12) & 16.96 (0.32) & \textbf{14.24} (0.21) & 21.84 (1.30) \\
 &  & 0.60 & 50.96 (2.76) & 44.69 (2.82) & 23.57 (0.42) & \textbf{19.80} (0.41) & 30.48 (1.61) \\
\multirow{4}{*}{500} & \multirow{2}{*}{Gaussian} & 0.35 & 24.08 (0.78) & 19.96 (0.81) & 12.90 (0.16) & \textbf{11.26} (0.09) & 17.91 (1.29) \\
 &  & 0.60 & 30.48 (1.48) & 28.81 (1.49) & 26.64 (0.24) & \textbf{17.66} (0.36) & 24.04 (1.03) \\
 & \multirow{2}{*}{Laplace} & 0.35 & 19.72 (0.47) & 15.06 (0.49) & 11.67 (0.17) & \textbf{10.96} (0.08) & 12.81 (0.86) \\
 &  & 0.60 & 25.98 (1.15) & 23.51 (1.16) & 17.58 (0.28) & \textbf{13.34} (0.17) & 20.68 (1.15) \\
\multirow{4}{*}{1000} & \multirow{2}{*}{Gaussian} & 0.35 & 16.18 (0.36) & 14.14 (0.36) & 12.02 (0.13) & \textbf{10.20} (0.05) & 14.63 (1.24) \\
 &  & 0.60 & 25.29 (0.80) & 23.82 (0.81) & 20.95 (0.31) & \textbf{13.06} (0.18) & 21.66 (1.24) \\
 & \multirow{2}{*}{Laplace} & 0.35 & 13.54 (0.28) & 11.12 (0.27) & \textbf{7.99} (0.09) & 10.02 (0.04) & 9.31 (0.69) \\
 &  & 0.60 & 20.81 (0.64) & 18.64 (0.65) & 12.89 (0.15) & \textbf{11.11} (0.09) & 17.25 (1.16) \\
\multirow{4}{*}{5000} & \multirow{2}{*}{Gaussian} & 0.35 & 10.30 (0.18) & 9.69 (0.18) & \textbf{7.06} (0.05) & 9.24 (0.02) & 8.49 (0.89) \\
 &  & 0.60 & 15.82 (0.32) & 15.62 (0.32) & 17.18 (0.12) & \textbf{9.86} (0.05) & 13.69 (0.83) \\
 & \multirow{2}{*}{Laplace} & 0.35 & 6.72 (0.10) & 5.92 (0.09) & 4.67 (0.04) & 9.22 (0.01) & \textbf{4.14} (0.36) \\
 &  & 0.60 & 10.01 (0.20) & 9.58 (0.20) & \textbf{7.47} (0.07) & 9.46 (0.03) & 7.96 (0.70) \\
\addlinespace[2pt]
\multicolumn{8}{l}{\textit{Beta}} \\
\multirow{4}{*}{200} & \multirow{2}{*}{Gaussian} & 0.35 & 19.93 (0.65) & 13.91 (0.71) & 7.64 (0.16) & \textbf{7.52} (0.21) & 15.99 (1.12) \\
 &  & 0.60 & 25.48 (2.05) & 20.48 (2.09) & \textbf{13.20} (0.25) & 23.83 (0.77) & 19.99 (1.19) \\
 & \multirow{2}{*}{Laplace} & 0.35 & 19.48 (0.79) & 13.55 (0.83) & 7.38 (0.16) & \textbf{7.00} (0.19) & 15.79 (1.14) \\
 &  & 0.60 & 20.59 (1.11) & 15.09 (1.15) & \textbf{10.20} (0.21) & 12.40 (0.40) & 19.97 (1.24) \\
\multirow{4}{*}{500} & \multirow{2}{*}{Gaussian} & 0.35 & 9.32 (0.42) & 7.54 (0.41) & 5.19 (0.09) & \textbf{3.90} (0.08) & 11.90 (1.24) \\
 &  & 0.60 & 18.68 (0.54) & 15.16 (0.58) & \textbf{9.06} (0.14) & 10.21 (0.32) & 12.47 (1.00) \\
 & \multirow{2}{*}{Laplace} & 0.35 & 8.42 (0.26) & 6.72 (0.25) & 4.26 (0.08) & \textbf{3.65} (0.08) & 7.61 (0.58) \\
 &  & 0.60 & 16.56 (0.62) & 12.80 (0.65) & 6.46 (0.13) & \textbf{6.04} (0.17) & 13.16 (1.08) \\
\multirow{4}{*}{1000} & \multirow{2}{*}{Gaussian} & 0.35 & 6.69 (0.20) & 5.23 (0.20) & 3.63 (0.05) & \textbf{2.70} (0.04) & 7.18 (0.80) \\
 &  & 0.60 & 9.77 (0.32) & 8.96 (0.30) & 8.40 (0.10) & \textbf{6.00} (0.19) & 10.27 (0.97) \\
 & \multirow{2}{*}{Laplace} & 0.35 & 6.00 (0.14) & 4.49 (0.15) & 3.42 (0.05) & \textbf{2.62} (0.04) & 5.75 (0.60) \\
 &  & 0.60 & 9.50 (0.40) & 8.24 (0.39) & 5.05 (0.09) & \textbf{3.94} (0.09) & 8.82 (0.79) \\
\multirow{4}{*}{5000} & \multirow{2}{*}{Gaussian} & 0.35 & 3.68 (0.10) & 3.23 (0.10) & 1.95 (0.02) & \textbf{1.78} (0.01) & 4.88 (0.69) \\
 &  & 0.60 & 5.49 (0.08) & 5.19 (0.08) & 5.00 (0.05) & \textbf{2.37} (0.04) & 6.34 (0.69) \\
 & \multirow{2}{*}{Laplace} & 0.35 & 2.81 (0.05) & 2.25 (0.04) & \textbf{1.31} (0.02) & 1.74 (0.01) & 2.14 (0.26) \\
 &  & 0.60 & 4.35 (0.05) & 3.81 (0.05) & 2.24 (0.03) & \textbf{1.93} (0.02) & 5.18 (0.64) \\
\end{longtable}
\endgroup

\begingroup
\scriptsize
\setlength{\tabcolsep}{4.2pt}
\renewcommand{\arraystretch}{1.08}
\begin{longtable}{r l c r r r r r}
\caption{Comparison for the asymmetric unimodal latent densities using the skew-normal sieve, in terms of $10^3\times\mathrm{MISE}$ (SE). Row minima for MISE are shown in bold.}\label{tab:skew-asym-full}\\
\toprule
$n$ & Error & $c$ & $\mathrm{MIX}^{SN}$ & $\mathrm{POST}^{SN}$ & DEC & PC & QP \\
\midrule
\endfirsthead
\multicolumn{8}{c}{\tablename\ \thetable\ -- continued from previous page}\\
\toprule
$n$ & Error & $c$ & $\mathrm{MIX}^{SN}$ & $\mathrm{POST}^{SN}$ & DEC & PC & QP \\
\midrule
\endhead
\midrule\multicolumn{8}{r}{Continued on next page}\\
\endfoot
\bottomrule
\endlastfoot
\multicolumn{8}{l}{\textit{Gamma}}\\
200 & Gaussian & 0.35 & \textbf{3.44} (0.13) & 4.67 (0.15) & 7.03 (0.19) & 6.18 (0.20) & 14.22 (1.29) \\
 &  & 0.60 & \textbf{5.16} (0.21) & 5.32 (0.20) & 13.52 (0.34) & 25.96 (0.97) & 17.35 (1.28) \\
 & Laplace & 0.35 & 8.48 (0.29) & \textbf{5.98} (0.18) & 6.70 (0.17) & 5.98 (0.17) & 13.22 (1.00) \\
 &  & 0.60 & 8.02 (0.29) & \textbf{6.23} (0.22) & 10.39 (0.23) & 11.03 (0.36) & 20.09 (1.51) \\
500 & Gaussian & 0.35 & \textbf{1.54} (0.05) & 1.94 (0.06) & 4.52 (0.12) & 3.06 (0.09) & 7.55 (0.88) \\
 &  & 0.60 & \textbf{2.23} (0.08) & 2.42 (0.09) & 9.31 (0.16) & 9.44 (0.35) & 11.25 (1.14) \\
 & Laplace & 0.35 & 4.76 (0.16) & 3.43 (0.12) & 3.73 (0.09) & \textbf{2.63} (0.07) & 7.26 (0.77) \\
 &  & 0.60 & 3.98 (0.16) & \textbf{3.01} (0.11) & 6.15 (0.13) & 5.10 (0.16) & 9.19 (0.84) \\
1000 & Gaussian & 0.35 & \textbf{1.08} (0.04) & 1.20 (0.04) & 3.01 (0.06) & 1.74 (0.04) & 6.19 (0.88) \\
 &  & 0.60 & \textbf{1.38} (0.05) & 1.43 (0.05) & 8.37 (0.12) & 5.22 (0.18) & 8.29 (1.03) \\
 & Laplace & 0.35 & 2.75 (0.07) & 2.07 (0.06) & 2.73 (0.07) & \textbf{1.67} (0.04) & 4.80 (0.63) \\
 &  & 0.60 & 2.53 (0.10) & \textbf{1.79} (0.06) & 4.22 (0.10) & 2.85 (0.08) & 7.68 (0.90) \\
5000 & Gaussian & 0.35 & 0.65 (0.02) & \textbf{0.46} (0.01) & 1.47 (0.02) & 0.77 (0.01) & 3.19 (0.64) \\
 &  & 0.60 & 0.86 (0.03) & \textbf{0.77} (0.03) & 4.62 (0.05) & 1.47 (0.04) & 4.89 (0.62) \\
 & Laplace & 0.35 & 1.29 (0.04) & 0.93 (0.03) & 0.95 (0.02) & \textbf{0.74} (0.01) & 1.97 (0.37) \\
 &  & 0.60 & 1.41 (0.04) & 1.19 (0.04) & 1.78 (0.03) & \textbf{1.00} (0.02) & 2.84 (0.35) \\
\addlinespace[0.35em]
\multicolumn{8}{l}{\textit{Comte $\chi^2$}}\\
200 & Gaussian & 0.35 & 18.65 (0.54) & \textbf{13.53} (0.53) & 33.68 (0.38) & 29.33 (0.20) & 35.17 (1.39) \\
 &  & 0.60 & 29.11 (0.60) & \textbf{24.10} (0.59) & 49.47 (0.54) & 47.64 (0.95) & 49.52 (1.57) \\
 & Laplace & 0.35 & 24.72 (0.62) & \textbf{19.15} (0.45) & 29.38 (0.37) & 28.67 (0.18) & 31.01 (1.34) \\
 &  & 0.60 & 32.13 (0.73) & \textbf{25.83} (0.62) & 39.52 (0.53) & 34.49 (0.40) & 40.02 (1.38) \\
500 & Gaussian & 0.35 & 12.69 (0.25) & \textbf{8.25} (0.22) & 25.73 (0.20) & 25.88 (0.09) & 25.62 (1.12) \\
 &  & 0.60 & 21.37 (0.47) & \textbf{17.37} (0.47) & 46.30 (0.37) & 32.76 (0.35) & 38.23 (1.19) \\
 & Laplace & 0.35 & 15.64 (0.45) & \textbf{13.16} (0.37) & 21.06 (0.29) & 25.84 (0.08) & 20.72 (1.02) \\
 &  & 0.60 & 22.29 (0.46) & \textbf{18.53} (0.43) & 30.07 (0.38) & 27.88 (0.15) & 32.29 (1.39) \\
1000 & Gaussian & 0.35 & 9.26 (0.26) & \textbf{6.55} (0.23) & 21.49 (0.27) & 24.86 (0.05) & 23.86 (1.37) \\
 &  & 0.60 & 18.51 (0.45) & \textbf{15.43} (0.45) & 36.80 (0.42) & 28.30 (0.20) & 31.22 (1.18) \\
 & Laplace & 0.35 & 11.09 (0.32) & \textbf{9.35} (0.25) & 17.24 (0.12) & 24.69 (0.04) & 16.66 (0.93) \\
 &  & 0.60 & 18.77 (0.38) & \textbf{16.46} (0.36) & 25.76 (0.19) & 25.77 (0.08) & 24.08 (1.04) \\
5000 & Gaussian & 0.35 & 2.12 (0.04) & \textbf{2.02} (0.04) & 16.62 (0.08) & 23.94 (0.01) & 14.13 (0.80) \\
 &  & 0.60 & 5.62 (0.20) & \textbf{5.46} (0.20) & 33.54 (0.17) & 24.67 (0.06) & 20.76 (0.68) \\
 & Laplace & 0.35 & 7.61 (0.18) & \textbf{6.91} (0.16) & 8.60 (0.09) & 21.27 (0.27) & 9.40 (0.69) \\
 &  & 0.60 & 6.64 (0.17) & \textbf{6.40} (0.16) & 16.80 (0.10) & 24.15 (0.03) & 15.52 (0.88) \\
\addlinespace[0.35em]
\multicolumn{8}{l}{\textit{Cai $\chi^2$}}\\
200 & Gaussian & 0.35 & 7.67 (0.30) & \textbf{6.95} (0.29) & 18.79 (0.26) & 14.39 (0.20) & 24.06 (1.46) \\
 &  & 0.60 & 13.83 (0.41) & \textbf{11.65} (0.41) & 29.69 (0.62) & 31.67 (1.18) & 34.13 (2.11) \\
 & Laplace & 0.35 & 14.48 (0.55) & \textbf{10.71} (0.37) & 16.92 (0.28) & 13.93 (0.18) & 23.65 (1.51) \\
 &  & 0.60 & 16.80 (0.50) & \textbf{12.61} (0.38) & 23.96 (0.40) & 19.56 (0.39) & 26.02 (1.21) \\
500 & Gaussian & 0.35 & 4.37 (0.11) & \textbf{3.11} (0.10) & 12.71 (0.15) & 11.11 (0.08) & 17.35 (1.00) \\
 &  & 0.60 & 7.20 (0.19) & \textbf{5.54} (0.18) & 25.97 (0.24) & 17.81 (0.35) & 28.25 (1.39) \\
 & Laplace & 0.35 & 10.21 (0.29) & \textbf{7.85} (0.23) & 11.47 (0.18) & 10.98 (0.09) & 12.93 (0.74) \\
 &  & 0.60 & 9.67 (0.32) & \textbf{7.35} (0.25) & 17.57 (0.25) & 13.11 (0.16) & 19.02 (1.00) \\
1000 & Gaussian & 0.35 & 3.57 (0.10) & \textbf{2.27} (0.09) & 11.90 (0.13) & 10.13 (0.05) & 14.09 (1.18) \\
 &  & 0.60 & 5.06 (0.13) & \textbf{3.66} (0.12) & 20.93 (0.31) & 13.30 (0.17) & 23.38 (1.42) \\
 & Laplace & 0.35 & 7.14 (0.19) & \textbf{5.50} (0.14) & 7.98 (0.09) & 10.04 (0.05) & 9.34 (0.58) \\
 &  & 0.60 & 7.62 (0.23) & \textbf{6.13} (0.21) & 12.82 (0.16) & 11.16 (0.10) & 15.83 (0.93) \\
5000 & Gaussian & 0.35 & 1.39 (0.03) & \textbf{1.36} (0.03) & 7.08 (0.05) & 9.23 (0.02) & 7.32 (0.77) \\
 &  & 0.60 & 2.82 (0.11) & \textbf{2.49} (0.10) & 17.09 (0.11) & 9.86 (0.05) & 11.20 (0.54) \\
 & Laplace & 0.35 & 4.31 (0.09) & \textbf{3.68} (0.08) & 4.58 (0.04) & 9.16 (0.01) & 5.04 (0.51) \\
 &  & 0.60 & 3.70 (0.11) & \textbf{3.39} (0.10) & 7.47 (0.07) & 9.43 (0.03) & 8.19 (0.69) \\
\addlinespace[0.35em]
\multicolumn{8}{l}{\textit{Beta}}\\
200 & Gaussian & 0.35 & 6.37 (0.20) & \textbf{5.91} (0.16) & 7.48 (0.17) & 7.24 (0.21) & 14.57 (1.22) \\
 &  & 0.60 & 9.61 (0.30) & \textbf{8.43} (0.24) & 12.64 (0.24) & 23.01 (0.77) & 17.04 (1.16) \\
 & Laplace & 0.35 & 12.72 (0.28) & 7.68 (0.20) & 7.17 (0.16) & \textbf{7.04} (0.21) & 12.53 (0.93) \\
 &  & 0.60 & 10.61 (0.28) & \textbf{8.09} (0.23) & 10.23 (0.20) & 12.56 (0.40) & 17.90 (1.24) \\
500 & Gaussian & 0.35 & 4.30 (0.09) & \textbf{3.41} (0.08) & 5.44 (0.10) & 3.98 (0.08) & 10.08 (1.04) \\
 &  & 0.60 & 5.66 (0.13) & \textbf{5.00} (0.12) & 9.02 (0.14) & 10.90 (0.34) & 15.09 (1.17) \\
 & Laplace & 0.35 & 8.10 (0.22) & 5.38 (0.15) & 4.35 (0.08) & \textbf{3.73} (0.07) & 8.09 (0.78) \\
 &  & 0.60 & 8.07 (0.20) & \textbf{5.49} (0.14) & 6.55 (0.13) & 6.30 (0.18) & 11.44 (0.96) \\
1000 & Gaussian & 0.35 & 3.46 (0.06) & \textbf{2.29} (0.05) & 3.63 (0.05) & 2.70 (0.05) & 7.17 (0.69) \\
 &  & 0.60 & 4.38 (0.09) & \textbf{3.76} (0.08) & 8.03 (0.10) & 6.09 (0.17) & 13.17 (1.26) \\
 & Laplace & 0.35 & 4.51 (0.10) & 3.57 (0.09) & 3.49 (0.05) & \textbf{2.62} (0.04) & 6.55 (0.62) \\
 &  & 0.60 & 6.17 (0.14) & 4.36 (0.10) & 4.99 (0.09) & \textbf{3.71} (0.08) & 9.55 (0.79) \\
5000 & Gaussian & 0.35 & 2.44 (0.05) & 1.80 (0.04) & 1.96 (0.02) & \textbf{1.77} (0.01) & 4.02 (0.52) \\
 &  & 0.60 & 3.62 (0.07) & 3.00 (0.07) & 4.98 (0.05) & \textbf{2.39} (0.04) & 6.33 (0.62) \\
 & Laplace & 0.35 & 2.41 (0.05) & 1.76 (0.04) & \textbf{1.32} (0.02) & 1.73 (0.01) & 3.00 (0.38) \\
 &  & 0.60 & 2.95 (0.07) & 2.69 (0.06) & 2.22 (0.03) & \textbf{1.95} (0.02) & 4.23 (0.65) \\
\end{longtable}
\endgroup

\begingroup
\scriptsize
\setlength{\tabcolsep}{4.0pt}
\renewcommand{\arraystretch}{1.05}
\begin{longtable}{rlcccccc}
\caption{Comparison for the bimodal latent densities in terms of \(10^3\times\mathrm{MISE}\) (SE). Row minima for MISE are shown in bold.}\label{tab:bimodal-mise-se}\\
\toprule
$n$ & Error & $c$ & $\mathrm{MIX}$ & ${\mathrm{POST}}$ & DEC & PC & QP \\
\midrule
\endfirsthead
\multicolumn{8}{c}{\tablename\ \thetable\ -- continued from previous page}\\
\toprule
$n$ & Error & $c$ & $\mathrm{MIX}$ & ${\mathrm{POST}}$ & DEC & PC & QP \\
\midrule
\endhead
\midrule\multicolumn{8}{r}{Continued on next page}\\
\endfoot
\bottomrule
\endlastfoot
\multicolumn{8}{l}{\textit{Comte Gamma mixture}} \\
\multirow{4}{*}{200} & \multirow{2}{*}{Gaussian} & 0.35 & 17.90 (0.76) & 14.83 (0.81) & 7.91 (0.09) & \textbf{7.90} (0.10) & 12.13 (0.57) \\
 &  & 0.60 & 15.68 (0.65) & 13.69 (0.67) & \textbf{11.22} (0.10) & 16.92 (0.46) & 18.04 (0.68) \\
 & \multirow{2}{*}{Laplace} & 0.35 & 16.62 (0.86) & 13.76 (0.91) & \textbf{7.24} (0.08) & 7.61 (0.09) & 10.67 (0.47) \\
 &  & 0.60 & 16.95 (0.92) & 14.17 (0.95) & \textbf{9.33} (0.08) & 10.38 (0.18) & 15.28 (0.57) \\
\multirow{4}{*}{500} & \multirow{2}{*}{Gaussian} & 0.35 & 6.46 (0.31) & 6.16 (0.30) & \textbf{5.91} (0.07) & 6.29 (0.04) & 7.30 (0.49) \\
 &  & 0.60 & 14.06 (0.27) & 12.27 (0.29) & \textbf{8.72} (0.07) & 9.51 (0.17) & 11.01 (0.52) \\
 & \multirow{2}{*}{Laplace} & 0.35 & 5.11 (0.30) & \textbf{4.86} (0.29) & 4.93 (0.06) & 6.17 (0.04) & 5.41 (0.31) \\
 &  & 0.60 & 13.47 (0.46) & 11.38 (0.47) & \textbf{6.77} (0.07) & 7.08 (0.07) & 8.31 (0.47) \\
\multirow{4}{*}{1000} & \multirow{2}{*}{Gaussian} & 0.35 & 3.02 (0.14) & \textbf{2.98} (0.14) & 4.37 (0.04) & 5.66 (0.02) & 4.79 (0.38) \\
 &  & 0.60 & 10.13 (0.17) & 9.08 (0.15) & 8.43 (0.05) & \textbf{7.40} (0.08) & 9.06 (0.47) \\
 & \multirow{2}{*}{Laplace} & 0.35 & \textbf{2.49} (0.11) & 2.50 (0.11) & 3.99 (0.05) & 5.64 (0.02) & 4.00 (0.28) \\
 &  & 0.60 & 7.43 (0.32) & 6.67 (0.31) & \textbf{5.89} (0.07) & 6.19 (0.04) & 6.33 (0.44) \\
\multirow{4}{*}{5000} & \multirow{2}{*}{Gaussian} & 0.35 & 1.37 (0.03) & \textbf{1.27} (0.03) & 2.49 (0.02) & 5.23 (0.01) & 3.50 (0.45) \\
 &  & 0.60 & 4.05 (0.09) & \textbf{3.85} (0.08) & 6.13 (0.05) & 5.58 (0.02) & 5.49 (0.34) \\
 & \multirow{2}{*}{Laplace} & 0.35 & 0.95 (0.02) & \textbf{0.85} (0.02) & 1.57 (0.02) & 5.22 (0.01) & 1.31 (0.13) \\
 &  & 0.60 & 1.60 (0.04) & \textbf{1.51} (0.04) & 2.66 (0.04) & 5.32 (0.01) & 2.89 (0.22) \\
\addlinespace[2pt]
\multicolumn{8}{l}{\textit{Cai Gamma mixture}} \\
\multirow{4}{*}{200} & \multirow{2}{*}{Gaussian} & 0.35 & 33.06 (0.89) & 26.27 (1.01) & 17.11 (0.18) & \textbf{16.85} (0.23) & 20.40 (0.97) \\
 &  & 0.60 & 35.31 (1.77) & 30.96 (1.81) & \textbf{23.68} (0.20) & 33.10 (0.84) & 30.79 (1.20) \\
 & \multirow{2}{*}{Laplace} & 0.35 & 28.99 (0.80) & 22.85 (0.88) & \textbf{15.84} (0.17) & 16.25 (0.20) & 19.35 (1.19) \\
 &  & 0.60 & 33.54 (1.07) & 27.54 (1.14) & \textbf{19.71} (0.16) & 21.41 (0.39) & 25.71 (1.13) \\
\multirow{4}{*}{500} & \multirow{2}{*}{Gaussian} & 0.35 & 14.03 (0.66) & 13.38 (0.65) & \textbf{12.71} (0.14) & 13.03 (0.09) & 14.98 (0.99) \\
 &  & 0.60 & 30.79 (0.65) & 27.14 (0.69) & \textbf{18.74} (0.14) & 19.62 (0.36) & 24.73 (1.14) \\
 & \multirow{2}{*}{Laplace} & 0.35 & 10.64 (0.51) & \textbf{10.22} (0.48) & 10.67 (0.13) & 12.90 (0.08) & 11.60 (0.72) \\
 &  & 0.60 & 27.71 (0.63) & 23.09 (0.67) & \textbf{14.80} (0.15) & 14.99 (0.16) & 17.26 (0.79) \\
\multirow{4}{*}{1000} & \multirow{2}{*}{Gaussian} & 0.35 & 6.78 (0.30) & \textbf{6.70} (0.31) & 9.57 (0.08) & 11.83 (0.05) & 11.95 (1.03) \\
 &  & 0.60 & 20.76 (0.43) & 18.80 (0.39) & 17.76 (0.10) & \textbf{14.76} (0.17) & 20.45 (0.96) \\
 & \multirow{2}{*}{Laplace} & 0.35 & 5.19 (0.25) & \textbf{5.16} (0.25) & 8.47 (0.11) & 11.75 (0.05) & 7.97 (0.53) \\
 &  & 0.60 & 15.17 (0.56) & 13.64 (0.52) & \textbf{12.53} (0.14) & 12.96 (0.09) & 13.68 (0.68) \\
\multirow{4}{*}{5000} & \multirow{2}{*}{Gaussian} & 0.35 & 2.80 (0.06) & \textbf{2.60} (0.06) & 5.40 (0.04) & 10.84 (0.01) & 7.73 (0.85) \\
 &  & 0.60 & 7.47 (0.18) & \textbf{7.15} (0.18) & 12.98 (0.10) & 11.48 (0.05) & 12.49 (0.75) \\
 & \multirow{2}{*}{Laplace} & 0.35 & 2.12 (0.04) & \textbf{1.91} (0.04) & 3.31 (0.03) & 10.82 (0.01) & 3.56 (0.43) \\
 &  & 0.60 & 3.10 (0.11) & \textbf{2.98} (0.11) & 5.69 (0.07) & 11.06 (0.03) & 6.97 (0.58) \\
\addlinespace[2pt]
\multicolumn{8}{l}{\textit{Normal mixture G}} \\
\multirow{4}{*}{200} & \multirow{2}{*}{Gaussian} & 0.35 & 8.92 (0.52) & \textbf{8.87} (0.50) & 18.11 (0.18) & 15.13 (0.11) & 11.17 (0.48) \\
 &  & 0.60 & 54.34 (1.66) & 48.34 (1.71) & 31.31 (0.31) & 23.58 (0.45) & \textbf{18.38} (0.69) \\
 & \multirow{2}{*}{Laplace} & 0.35 & \textbf{8.17} (0.82) & 8.44 (0.82) & 15.24 (0.24) & 15.09 (0.10) & 12.00 (0.51) \\
 &  & 0.60 & 38.18 (2.11) & 34.29 (2.10) & 23.55 (0.31) & 17.31 (0.17) & \textbf{15.29} (0.51) \\
\multirow{4}{*}{500} & \multirow{2}{*}{Gaussian} & 0.35 & \textbf{2.98} (0.10) & 3.03 (0.10) & 11.61 (0.11) & 13.62 (0.05) & 6.27 (0.42) \\
 &  & 0.60 & 17.34 (0.78) & 16.66 (0.72) & 26.84 (0.18) & 16.98 (0.19) & \textbf{11.23} (0.49) \\
 & \multirow{2}{*}{Laplace} & 0.35 & \textbf{2.43} (0.09) & 2.57 (0.09) & 7.84 (0.17) & 13.56 (0.05) & 5.92 (0.37) \\
 &  & 0.60 & 8.43 (0.67) & \textbf{8.31} (0.65) & 16.04 (0.26) & 14.74 (0.09) & 8.77 (0.43) \\
\multirow{4}{*}{1000} & \multirow{2}{*}{Gaussian} & 0.35 & \textbf{1.42} (0.05) & 1.44 (0.05) & 7.98 (0.16) & 13.04 (0.03) & 3.21 (0.23) \\
 &  & 0.60 & \textbf{6.58} (0.29) & 6.59 (0.29) & 22.91 (0.30) & 14.75 (0.12) & 7.90 (0.37) \\
 & \multirow{2}{*}{Laplace} & 0.35 & \textbf{1.13} (0.04) & 1.20 (0.04) & 6.01 (0.07) & 12.99 (0.02) & 3.82 (0.31) \\
 &  & 0.60 & \textbf{3.22} (0.13) & 3.23 (0.13) & 11.42 (0.12) & 13.57 (0.05) & 4.92 (0.30) \\
\multirow{4}{*}{5000} & \multirow{2}{*}{Gaussian} & 0.35 & 0.29 (0.01) & \textbf{0.29} (0.01) & 5.59 (0.04) & 12.64 (0.01) & 1.71 (0.20) \\
 &  & 0.60 & 2.04 (0.06) & \textbf{2.02} (0.06) & 17.27 (0.10) & 12.89 (0.04) & 3.47 (0.23) \\
 & \multirow{2}{*}{Laplace} & 0.35 & \textbf{0.24} (0.01) & 0.25 (0.01) & 1.75 (0.02) & 0.73 (0.02) & 1.57 (0.18) \\
 &  & 0.60 & 0.78 (0.03) & \textbf{0.77} (0.03) & 4.65 (0.09) & 12.71 (0.02) & 2.25 (0.18) \\
\addlinespace[2pt]
\multicolumn{8}{l}{\textit{Normal mixture Y}} \\
\multirow{4}{*}{200} & \multirow{2}{*}{Gaussian} & 0.35 & 12.45 (0.81) & 12.38 (0.80) & 19.44 (0.41) & \textbf{11.40} (0.19) & 16.00 (1.20) \\
 &  & 0.60 & 62.99 (3.41) & 56.99 (3.40) & 39.98 (0.64) & 30.63 (0.91) & \textbf{26.41} (1.33) \\
 & \multirow{2}{*}{Laplace} & 0.35 & 11.46 (0.86) & 11.93 (0.84) & 15.51 (0.39) & \textbf{11.18} (0.18) & 15.75 (1.29) \\
 &  & 0.60 & 47.00 (2.82) & 42.23 (2.76) & 26.69 (0.60) & \textbf{17.59} (0.40) & 21.11 (1.25) \\
\multirow{4}{*}{500} & \multirow{2}{*}{Gaussian} & 0.35 & 4.44 (0.22) & \textbf{4.34} (0.21) & 12.59 (0.20) & 8.24 (0.09) & 11.90 (1.30) \\
 &  & 0.60 & 19.49 (1.31) & 19.34 (1.31) & 34.85 (0.44) & \textbf{15.79} (0.36) & 17.91 (1.01) \\
 & \multirow{2}{*}{Laplace} & 0.35 & \textbf{3.78} (0.18) & 4.03 (0.17) & 8.46 (0.22) & 8.15 (0.08) & 9.90 (0.87) \\
 &  & 0.60 & 10.48 (0.74) & 10.38 (0.73) & 16.10 (0.40) & \textbf{10.33} (0.15) & 14.83 (1.12) \\
\multirow{4}{*}{1000} & \multirow{2}{*}{Gaussian} & 0.35 & \textbf{2.12} (0.13) & 2.13 (0.12) & 8.76 (0.20) & 7.22 (0.05) & 8.39 (1.00) \\
 &  & 0.60 & 9.40 (0.49) & \textbf{9.33} (0.48) & 22.68 (0.38) & 10.58 (0.18) & 12.24 (1.00) \\
 & \multirow{2}{*}{Laplace} & 0.35 & \textbf{1.67} (0.08) & 1.87 (0.07) & 6.40 (0.12) & 7.10 (0.04) & 6.76 (0.71) \\
 &  & 0.60 & 4.51 (0.21) & \textbf{4.47} (0.20) & 12.95 (0.22) & 8.39 (0.08) & 10.92 (0.89) \\
\multirow{4}{*}{5000} & \multirow{2}{*}{Gaussian} & 0.35 & \textbf{0.34} (0.01) & 0.35 (0.01) & 5.80 (0.06) & 6.30 (0.01) & 5.12 (0.90) \\
 &  & 0.60 & 1.88 (0.10) & \textbf{1.83} (0.09) & 19.95 (0.15) & 6.93 (0.04) & 6.19 (0.69) \\
 & \multirow{2}{*}{Laplace} & 0.35 & \textbf{0.31} (0.01) & 0.35 (0.01) & 2.01 (0.04) & 6.25 (0.01) & 2.32 (0.34) \\
 &  & 0.60 & 0.84 (0.04) & \textbf{0.83} (0.04) & 5.55 (0.10) & 6.49 (0.02) & 3.80 (0.46) \\
\end{longtable}
\endgroup

\scriptsize
\setlength{\tabcolsep}{2.4pt}
\renewcommand{\arraystretch}{1.10}
\setlength{\LTleft}{\fill}
\setlength{\LTright}{\fill}
\begin{longtable}{@{}llr*{7}{c}@{}}
\caption{Pairwise comparisons between estimators across symmetric-unimodal,
asymmetric-unimodal, and bimodal simulation scenarios using paired two-sided
\(t\)-tests on ISEs at the 5\% level. Each cell reports first/second/n.s.,
namely the number of scenarios in which the first estimator has significantly
lower ISE, the second estimator has significantly lower ISE, or the difference
is not statistically significant.}
\label{tab:pairwise-all-expanded}
\\

\toprule
Group & Outcome & No. &
MIX/POST & MIX/DEC & MIX/PC & MIX/QP &
POST/DEC & POST/PC & POST/QP \\
\midrule
\endfirsthead
\multicolumn{10}{c}{\tablename\ \thetable\ -- continued from previous page}\\
\toprule
Group & Outcome & No. &
MIX/POST & MIX/DEC & MIX/PC & MIX/QP &
POST/DEC & POST/PC & POST/QP \\
\midrule
\endhead
\midrule\multicolumn{10}{r}{Continued on next page}\\
\endfoot
\bottomrule
\endlastfoot

\multicolumn{10}{@{}l}{
\textbf{Panel A. Symmetric-unimodal}
}\\*
\addlinespace[0.20em]
Overall
& All scenarios
& 80
& 27/52/1 & 74/3/3 & 59/15/6 & 67/6/7 & 75/0/5 & 73/0/7 & 78/0/2
\\

\addlinespace[0.30em]

Latent density
& Standard normal
& 16
& 16/0/0 & 16/0/0 & 16/0/0 & 16/0/0 & 16/0/0 & 16/0/0 & 16/0/0
\\

& $t_5$
& 16
& 0/15/1 & 15/0/1 & 10/2/4 & 16/0/0 & 15/0/1 & 14/0/2 & 16/0/0
\\

& Laplace
& 16
& 0/16/0 & 11/3/2 & 12/4/0 & 11/2/3 & 12/0/4 & 12/0/4 & 14/0/2
\\

& Cauchy
& 16
& 0/16/0 & 16/0/0 & 7/7/2 & 8/4/4 & 16/0/0 & 16/0/0 & 16/0/0
\\

& Nearly Normal
& 16
& 11/5/0 & 16/0/0 & 14/2/0 & 16/0/0 & 16/0/0 & 15/0/1 & 16/0/0
\\

\addlinespace[0.30em]

Error law
& Gaussian
& 40
& 13/27/0 & 38/1/1 & 31/7/2 & 33/3/4 & 38/0/2 & 39/0/1 & 39/0/1
\\

& Laplace
& 40
& 14/25/1 & 36/2/2 & 28/8/4 & 34/3/3 & 37/0/3 & 34/0/6 & 39/0/1
\\

\addlinespace[0.30em]

Error level
& 0.35
& 40
& 14/25/1 & 37/2/1 & 26/11/3 & 32/2/6 & 37/0/3 & 35/0/5 & 40/0/0
\\

& 0.60
& 40
& 13/27/0 & 37/1/2 & 33/4/3 & 35/4/1 & 38/0/2 & 38/0/2 & 38/0/2
\\

\addlinespace[0.30em]

Sample size
& 200
& 20
& 8/12/0 & 15/3/2 & 12/7/1 & 12/6/2 & 15/0/5 & 16/0/4 & 18/0/2
\\

& 500
& 20
& 8/12/0 & 19/0/1 & 15/2/3 & 20/0/0 & 20/0/0 & 18/0/2 & 20/0/0
\\

& 1000
& 20
& 7/12/1 & 20/0/0 & 17/2/1 & 18/0/2 & 20/0/0 & 20/0/0 & 20/0/0
\\

& 5000
& 20
& 4/16/0 & 20/0/0 & 15/4/1 & 17/0/3 & 20/0/0 & 19/0/1 & 20/0/0
\\

\addlinespace[0.65em]
\midrule
\addlinespace[0.35em]

\multicolumn{10}{@{}l}{
\textbf{Panel B. Asymmetric-unimodal latent densities}
}\\*
\addlinespace[0.20em]
Overall
& All scenarios
& 64
& 0/64/0 & 4/59/1 & 5/57/2 & 5/38/21 & 4/52/8 & 6/54/4 & 17/26/21
\\

\addlinespace[0.30em]

Latent density
& Gamma
& 16
& 0/16/0 & 2/14/0 & 0/15/1 & 4/1/11 & 2/10/4 & 1/14/1 & 10/0/6
\\

& Comte chi-square
& 16
& 0/16/0 & 1/14/1 & 4/12/0 & 0/16/0 & 1/14/1 & 4/11/1 & 0/16/0
\\

& Cai chi-square
& 16
& 0/16/0 & 1/15/0 & 1/15/0 & 0/15/1 & 1/14/1 & 1/14/1 & 0/9/7
\\

& Beta
& 16
& 0/16/0 & 0/16/0 & 0/15/1 & 1/6/9 & 0/14/2 & 0/15/1 & 7/1/8
\\

\addlinespace[0.30em]

Error law
& Gaussian
& 32
& 0/32/0 & 4/27/1 & 1/29/2 & 4/19/9 & 4/20/8 & 2/28/2 & 8/12/12
\\

& Laplace
& 32
& 0/32/0 & 0/32/0 & 4/28/0 & 1/19/12 & 0/32/0 & 4/26/2 & 9/14/9
\\

\addlinespace[0.30em]

Error level
& 0.35
& 32
& 0/32/0 & 0/32/0 & 4/28/0 & 3/18/11 & 0/29/3 & 4/27/1 & 12/11/9
\\

& 0.60
& 32
& 0/32/0 & 4/27/1 & 1/29/2 & 2/20/10 & 4/23/5 & 2/27/3 & 5/15/12
\\

\addlinespace[0.30em]

Sample size
& 200
& 16
& 0/16/0 & 0/16/0 & 0/14/2 & 0/12/4 & 0/16/0 & 1/14/1 & 4/6/6
\\

& 500
& 16
& 0/16/0 & 0/15/1 & 0/16/0 & 2/10/4 & 0/12/4 & 0/14/2 & 3/8/5
\\

& 1000
& 16
& 0/16/0 & 1/15/0 & 1/15/0 & 2/7/7 & 1/13/2 & 1/15/0 & 6/5/5
\\

& 5000
& 16
& 0/16/0 & 3/13/0 & 4/12/0 & 1/9/6 & 3/11/2 & 4/11/1 & 4/7/5
\\

\addlinespace[0.65em]
\midrule
\addlinespace[0.35em]

\multicolumn{10}{@{}l}{
\textbf{Panel C. Bimodal latent densities}
}\\*
\addlinespace[0.20em]
Overall
& All scenarios
& 64
& 12/44/8 & 40/20/4 & 37/19/8 & 38/18/8 & 40/20/4 & 38/17/9 & 39/12/13
\\

\addlinespace[0.30em]

Latent density
& Comte mixed gamma
& 16
& 0/15/1 & 6/8/2 & 7/7/2 & 7/6/3 & 6/8/2 & 8/6/2 & 8/4/4
\\

& Cai mixed gamma
& 16
& 0/15/1 & 6/8/2 & 7/7/2 & 6/7/3 & 6/8/2 & 7/6/3 & 6/3/7
\\

& Mixed normal G
& 16
& 7/5/4 & 14/2/0 & 13/2/1 & 12/3/1 & 14/2/0 & 13/2/1 & 12/3/1
\\

& Mixed normal Y
& 16
& 5/9/2 & 14/2/0 & 10/3/3 & 13/2/1 & 14/2/0 & 10/3/3 & 13/2/1
\\

\addlinespace[0.30em]

Error law
& Gaussian
& 32
& 4/24/4 & 20/10/2 & 17/9/6 & 19/9/4 & 20/10/2 & 18/9/5 & 20/6/6
\\

& Laplace
& 32
& 8/20/4 & 20/10/2 & 20/10/2 & 19/9/4 & 20/10/2 & 20/8/4 & 19/6/7
\\

\addlinespace[0.30em]

Error level
& 0.35
& 32
& 11/15/6 & 24/4/4 & 24/4/4 & 24/4/4 & 24/4/4 & 24/4/4 & 25/4/3
\\

& 0.60
& 32
& 1/29/2 & 16/16/0 & 13/15/4 & 14/14/4 & 16/16/0 & 14/13/5 & 14/8/10
\\

\addlinespace[0.30em]

Sample size
& 200
& 16
& 2/12/2 & 4/12/0 & 2/10/4 & 5/10/1 & 4/12/0 & 3/10/3 & 5/8/3
\\

& 500
& 16
& 3/12/1 & 8/4/4 & 7/5/4 & 5/5/6 & 8/4/4 & 7/5/4 & 6/4/6
\\

& 1000
& 16
& 4/8/4 & 12/4/0 & 12/4/0 & 12/3/1 & 12/4/0 & 12/2/2 & 12/0/4
\\

& 5000
& 16
& 3/12/1 & 16/0/0 & 16/0/0 & 16/0/0 & 16/0/0 & 16/0/0 & 16/0/0
\\
\end{longtable}

\normalsize

\renewcommand{\bibfont}{\fontsize{7.5}{8.5}\selectfont}